\documentclass[11pt,reqno]{amsart}
\usepackage{fullpage}

\usepackage{amsfonts,amssymb,amsmath}
\usepackage{mathrsfs}
\usepackage{dsfont}
\usepackage{mathtools}
\usepackage{amsthm}
\usepackage{mathtools}
\usepackage{leftindex,tensor}
\usepackage[OT1]{fontenc}
\usepackage{MnSymbol}
\usepackage{manfnt}
\usepackage[all,cmtip]{xy}
\usepackage{braket}

\usepackage{tikz-cd}
\tikzset{%
   symbol/.style={%
        draw=none,
        every to/.append style={%
            edge node={node [sloped, allow upside down, auto=false]{$#1$}}}
    }
}

\usepackage[pagebackref,hyperindex,linktocpage=true]{hyperref}
\hypersetup{
    colorlinks,
    linkcolor={red!50!black},
    citecolor={blue!50!black},
    urlcolor={blue!80!black}
}

\newtheorem{theo}{Theorem}[section]
\newtheorem{theo-def}[theo]{Theorem-Definition}
\newtheorem{lem}[theo]{Lemma}
\newtheorem{lem-def}[theo]{Lemma-Definition}
\newtheorem{prop}[theo]{Proposition}
\newtheorem{cor}[theo]{Corollary}

\theoremstyle{definition}
\newtheorem{defi}[theo]{Definition}

\newtheorem{rqe}[theo]{Remark}

\numberwithin{equation}{section}

\def\a{\alpha}

\def\w{\omega}

\def\fM{\mathfrak{M}}
\def\p{\mathfrak{p}}

\def\L{\mathcal{L}}

\def\k{\mathds{k}}

\def\I{\mathcal{I}}
\def\J{\mathcal{J}}
\def\L{\mathcal{L}}
\def\M{\mathcal{M}}
\def\O{\mathcal{O}}

\def\NN{\mathbb{N}}
\def\ZZ{\mathbb{Z}}

\def\GG{\mathbb{G}}
\def\QQ{\mathbb{Q}}
\def\AA{\mathbb{A}}

\def\bfB{\mathbf{B}}

\def\bfO{\mathbf{O}}
\def\bfU{\mathbf{U}}

\def\sI{\mathsf{I}}

\def\sM{\mathsf{M}}
\def\sP{\mathsf{P}}
\def\sQ{\mathsf{Q}}

\def\ab{\mathrm{ab}}
\def\ad{\mathrm{ad}}
\def\Ad{\mathrm{Ad}}
\def\der{\mathrm{der}}
\def\enh{\mathrm{enh}}

\def\gp{\mathrm{gp}}
\def\gr{\mathrm{gr}}
\def\pos{\mathrm{pos}}
\def\Vin{\mathrm{Vin}}
\def\la{\lambda}
\def\La{\Lambda}

\DeclareMathOperator{\As}{As}
\DeclareMathOperator{\End}{End}
\DeclareMathOperator{\GL}{GL}

\DeclareMathOperator{\Mat}{Mat}

\DeclareMathOperator{\Spec}{Spec}

\DeclareMathOperator{\Hom}{Hom}

\newcommand{\epic}{\twoheadrightarrow}
\newcommand{\into}{\hookrightarrow}

\title{Reductive monoids over general base}
\author{Jingren Chi}
\email{jrenchi@amss.ac.cn}
\address{Morningside Center of Mathematics and State Key Laboratory of Mathematical Sciences, Academy of Mathematics and Systems Science, Chinese
Academy of Sciences, Beijing 100190, China.}

\author{Simon Jacques}
\email{simon.jacques@amss.ac.cn}
\address{State Key Laboratory of Mathematical Sciences, Academy of Mathematics and Systems Science, Chinese
Academy of Sciences, Beijing 100190, China.}

\date{}

\begin{document}
\begin{abstract}
    We develop a theory of affine algebraic monoids over connected base schemes whose unit groups are split reductive groups. Our main result is a classification theorem for such objects, generalizing the work of Vinberg and Rittatore over a field. As applications, we obtain combinatorial descriptions and normality properties of orbit closures, prove a Steinberg-type theorem on adjoint quotients of split reductive monoids, and construct finite type integral models of the Vinberg monoids.
\end{abstract}
\maketitle

\tableofcontents
\section{Introduction}
\subsection{Background and motivations}
The theory of reductive monoids over a field, especially the Vinberg monoids, has played an increasingly foundational role in a number of areas, including geometric representation theory (\cite{Drinfeld-Gaitsgory,Bouthier-Springer,Chi-KV,Wang-FL,BZv-Ganev}), automorphic representation theory (\cite{BNS-arc,Ngo-PS,Ngo-Hankel}), and number theory (\cite{XiaoZhu2019,Zhu2020}), to name a few. It is a natural outgrowth of the theory of reductive algebraic groups and affine toric varieties, aiming to construct and classify algebraic monoids with reductive unit groups and to study the geometric properties of orbits under the unit groups.\par 
This theory has been established over any algebraically closed field, mainly through the works of Putcha, Renner, Vinberg and Rittatore. The early studies of linear algebraic monoids over a field were primarily conducted in the works of Putcha and Renner (see \cite{Putcha}, \cite{Renner} and the references therein). For reductive monoids, there have been (at least) three approaches to the classification problem. First, according to Renner's ``extension principle" (see \cite[\S5.1]{Renner} and references therein), the normal reductive monoids over a field are uniquely determined by a cone in the weight lattice that is stable under the Weyl group action. Second,  Vinberg \cite{Vinberg1995} systematically studied reductive monoids over an algebraically closed field of characteristic zero. His classification is in terms of a certain submonoid of the dominant weight monoid. In addition, another contribution of Vinberg was to realize the importance of the abelianisation map and single out the class of very flat reductive monoids that have good geometric and categorical properties: roughly speaking, they can all be constructed from a distinguished monoid, the \emph{universal enveloping monoid} (also known as the \emph{Vinberg monoid}). Later, Rittatore \cite{RittatoreThese,RittatoreMonoidArticle,RittatoreVeryflat} generalized the results of Vinberg to an arbitrary algebraically closed base field and established the link with Luna-Vust theory of spherical varieties. His classification is in terms of colored cones in the coweight lattice, which is in some sense dual to the viewpoint of Renner and Vinberg.\par 
Our goal in this paper is to establish a theory of reductive monoids over arbitrary base schemes and generalize the above-mentioned works. First, let us introduce the basic definitions. Fix a base scheme $S$. 
\begin{defi}\label{def:monoid-scheme}
    A \emph{monoid scheme} over $S$ is a scheme $\M$ over $S$, equipped with 
    \begin{itemize}
        \item the \emph{unit section} $e\colon S\to\M$, which is a section of the structure morphism $\M\to S$ and
        \item the \emph{composition law}, or \emph{product morphism} $\pi_\M\colon\M\times_S\M\to\M$, which is a morphism of $S$-schemes
    \end{itemize}
    that satisfies the usual associativity and unit axioms. The \emph{unit group} of $\M$, denoted $\M^\times$, is the group-valued functor defined by $\M^\times(S')\coloneqq\M(S')^\times$, the group of invertible elements in the monoid $\M(S')$, for any $S$-scheme $S'$.
\end{defi}
Put simply, this notion is obtained by relaxing the condition on the existence of inverses in the definition of a group scheme. If we further relax the condition on the existence of the unit section, we obtain the notion of a \emph{semigroup scheme}.\par 
The basic example of the monoid schemes we are interested in is the multiplicative monoid of $n$ by $n$ matrices (simply called the \emph{matrix monoid}): 
\[\Mat_{n,S}\coloneqq\Spec\O_S[X_{ij};1\le i,j\le n],\]
whose unit group is the open subscheme of invertible $n$ by $n$ matrices:
\[\quad\Mat_{n,S}^\times=\GL_{n,S}\coloneqq\Spec\O_S[X_{ij};1\le i,j\le n][\det(X_{ij})^{-1}].\]
We will study generalizations of matrix monoids, whose most important feature is that the unit group is a reductive group scheme. 
\begin{defi}
    A \emph{reductive monoid scheme} over $S$ (or simply, a \emph{reductive monoid over $S$}) is a monoid scheme $\M$ over $S$ such that the unit group $\M^\times$ is a reductive group scheme\footnote{For us, a reductive group scheme is always assumed to be fiberwise connected over the base.} over $S$ and the structure morphism $\M\to S$ is affine, flat, and of finite presentation, with integral and normal geometric fibers. 
\end{defi}
In particular, when $S=\Spec(k)$, where $k$ is a field, our definition is slightly more restrictive than the usual definition and coincides with the notion of an ``irreducible normal reductive monoid" in the sense of \cite{RittatoreMonoidArticle}. The first basic question we want to address is the following: 
\[\textit{How can one classify reductive monoids over a general base scheme?}
\]
The general classification of pinned split reductive group schemes is in terms of (based) root data, and is one of the major achievements in \cite{SGA3-3}. So the remaining task is to classify reductive monoids whose unit group is isomorphic to a fixed reductive group scheme $G$. In this paper we will give a complete answer to the question above in the case where the unit group $G$ is a split Chevalley group scheme.\par
Fix a split maximal torus $T\subset G$ and a Borel subgroup $B\subset G$ containing $T$. Let $W\coloneqq N_G(T)/T$ be the Weyl group. Let $X\coloneqq X^*(T)$ and $Y\coloneqq X_*(T)$ be respectively the weight and the coweight lattice. Let $X^+\subset X$ be the submonoid consisting of dominant weights (with respect to $B$). Also, we have the partial order $\le$ on $X$ determined by the positive roots. As in the classical theory of toric varieties and the theories of Renner, Vinberg and Rittatore, one expects that the extra parameter (besides the root datum) needed for the classification should be a certain cone in $X_\QQ\coloneqq X\otimes_\ZZ\QQ$, or other combinatorial objects that carry the same information.\par 
Let us recall Vinberg's approach to the classification problem. Assume for the moment that $S=\Spec(k)$ where $k$ is an algebraically closed field of characteristic $0$. The starting point is the algebraic version of the Peter-Weyl theorem which describes the $G\times G$-module structure on $k[G]$ (for the left and right regular representation):
\begin{equation}\label{eq:Peter-Weyl}
    k[G]\cong\bigoplus_{\la\in X^+}V_\la^*\otimes V_\la
\end{equation}
where $V_\la$ is the irreducible $G$-module with highest weight $\la$ and $V_\la^*\coloneqq\Hom_k(V_\la,k)$ is the dual module. For a reductive monoid $\M$ with unit group $G$, its coordinate ring $k[\M]$ is a sub-bialgebra of $k[G]$ and in particular a $G\times G$-submodule. So there is a subset $\L=\L(\M)\subset X^+$ such that 
\[k[\M]\cong\bigoplus_{\la\in\L}V_\la^*\otimes V_\la.\]
Then one shows that $\L$ is a \emph{weight monoid} in the sense of Definition \ref{def:weight-monoid}. This means that it is a finitely generated, saturated and downward closed submonoid of $X^+$ that generates $X$ as an abelian group. Here ``downward closed" means that for any weights $\la,\mu\in X^+$ such that $\mu\le\la$, if $\la\in\L$ then $\mu\in\L$; the condition that $\L$ is saturated in $X^+$ means that $n\la\in\L$ implies $\la\in\L$ for any positive integer $n$ and any dominant $\la\in X^+$, and follows from the normality of $\M$.\par  
The approach of Rittatore is to regard $\M$ as an affine embedding of the group $G$, now viewed as a homogeneous space for $G\times G$, and apply the general theory of spherical embeddings in this setting. This theory was first developed by Luna and Vust \cite{Luna-Vust} over an algebraically closed field of characteristic zero, and generalized by Knop \cite{Knop-LunaVust} to algebraically closed fields of arbitrary characteristic. In particular, Rittatore's approach works in arbitrary characteristic and, unlike Vinberg's work, his classification of reductive monoids is in terms of colored cones in $Y_\QQ$ (which turns out to be dual to the cone $\QQ_{\ge0}\L$ in $X_\QQ$, see \S\ref{sec:description-cone}).\par 
Finally, the approach of Renner is conceptually the most elegant in our opinion. The basic principle is that a reductive monoid $\M$ with unit group $G$ is uniquely determined by the closure of the maximal torus $T$ in it, which is an affine toric variety corresponding to a $W$-stable cone in $X_\QQ$. This is a more canonical invariant of the reductive monoid since it only depends on the underlying root datum, whereas the weight monoid $\L$ in Vinberg's approach and the colored cone in Rittatore's approach also depend on the extra choice of simple roots. However, it is not hard to show that this $W$-stable cone is precisely $W\cdot\QQ_{\ge0}\L$ and we have $\L=W\cdot\QQ_{\ge0}\L\cap X^+$, see Corollary \ref{cor:W-stable-cone}. This clarifies the relation with Vinberg's theory. \par
In this paper we will mainly adopt Vinberg's viewpoint since it makes the construction of (the coordinate rings of) reductive monoids more straightforward. In the same spirit of the classification of reductive group schemes in \cite{SGA3-3}, our classification of reductive monoids has two major steps: existence and uniqueness.\par 
For the existence, we need to construct a reductive monoid explicitly from the datum of a weight monoid $\L\subset X^+$ (or equivalently the corresponding $W$-stable convex polyhedral cone in $X_\QQ$). Of course, it suffices to do the construction over $\Spec(\ZZ)$ (and then base change to arbitrary schemes). 
This construction is carried out in our companion paper \cite{CJ1}, which uses Lusztig's theory of the (dual) canonical bases (cf. \cite{Lusztig-IntroQuantumGroup}) and builds upon his construction of the Chevalley group schemes in \cite{LusztigGroupScheme}.\par 
For the uniqueness, we need to show that all reductive monoids over a scheme $S$ with given unit group can be constructed in this way. In particular, this shows that, just like reductive groups, reductive monoids admit no nontrivial deformation with fixed unit group. This step is our first main objective in this paper, which is based on some technical results established in \cite{CJ1}.

\subsection{Results and strategies}
Let the notations be as before. Let $B^-\subset G$ be the Borel subgroup opposite to $B$. Let $U$ (resp. $U^-$) be the unipotent radical of $B$ (resp. $B^-$). To simplify notations, any base change of the group schemes $G,T,U$ etc. will be denoted by the same symbol. 
\subsubsection{The classification theorem}
We define a \emph{rigidified reductive $G$-monoid} over a scheme $S$ to be a pair $(\M,\iota)$ where $\M$ is a reductive monoid over $S$ and $\iota\colon G\cong\M^\times$ is an isomorphism of $S$-group schemes. Fix such a pair $(\M,\iota)$ and choose a geometric point $s\colon\Spec(k)\to S$ where $k$ is an algebraically closed field. Then the invariant subring $k[\M_s]^{U^-\times U}$ is identified (via $\iota$) with a $T\times T$-submodule of $k[G]^{U^-\times U}$. It is well known that the $T\times T$-weights of $k[G]^{U^-\times U}$ are $\{(-\la,\la)\in X\oplus X\mid \la\in X^+\}$ and each weight space has dimension one. Actually $k[G]^{U^-\times U}$ is isomorphic to the monoid algebra $k[X^+]$. Then we can define 
\[\L_s\coloneqq\{\la\in X^+\mid k[\M_s]^{U^-\times U}\text{ has nonzero }(-\la,\la)\text{ weight space}\}.\]
Now we can state our first main result.
\begin{theo}[Theorem \ref{theo:monoid-classification}]\label{theo:classification-intro}
    Assume that the base scheme $S$ is nonempty and connected and $G$ is a split pinned reductive group over $S$ as above. Then the set $\L_s$ associated to a rigidified reductive $G$-monoid $(\M,\iota)$ as above is independent of the choice of the geometric point $s\in S$. We call it the \emph{weight monoid of $\M$} and denote\footnote{Note that $\L(\M)$ also depends on the isomorphism $\iota\colon G\cong\M^\times$ although we omit it from the notation.} it by $\L(\M)\coloneqq\L_s$. Moreover, the map $(\M,\iota)\mapsto\L(\M)$ defines an isomorphism between the following partially ordered sets:
    \begin{itemize}
        \item isomorphism classes of rigidified reductive $G$-monoids $(\M,\iota)$ over $S$, with the partial order given by morphisms of $S$-monoid schemes that are compatible with the rigidifications;
        \item weight monoids in $X^+$ (see Definition \ref{def:weight-monoid}) with the partial order given by reverse inclusion.
    \end{itemize}
\end{theo}
In \cite{Chi-classification}, the first author uses this result to obtain a classification theorem for general (possibly non-split) reductive monoids over arbitrary (possibly non-connected) base schemes.  \par 
For any weight monoid $\L\subset X^+$, let $\M(\L)_S$ denote the reductive $G$-monoid over $S$ corresponding to $\L$ in the theorem above.  

\subsubsection{Realization in matrix monoids}
The construction of $\M(\L)_S$ is mainly performed in the companion paper \cite{CJ1} using the theory of canonical bases for modified quantum groups and even for the simplest example of matrix monoids the construction seems very abstract. However, under additional hypotheses on the base scheme, we will have very concrete and explicit realizations of $\M(\L)_S$ in terms of the matrix monoids.\par 
Choose a set of generators $\la_1,\dotsc,\la_r$ of the weight monoid $\L$. For simplicity let us assume here that the base $S=\Spec(\k)$ is an affine scheme. For each $1\le i\le r$, we have the Weyl module ${}_\k\La_{\la_i}$, which is a finite free $\k$-module equipped with an algebraic representation of $G$ with highest weight $\la_i$. Altogether we obtain a morphism of monoid schemes
\[\rho\colon G\to\prod_{1\le i\le r}\End({}_\k\La_{\la_i})\]
where the right hand side is a product of matrix monoids.
\begin{theo}[Proposition \ref{prop:lifting hatU module to hatUM module}, Theorem \ref{theo:embed-G-in-matrix}, and Proposition \ref{prop:closure-in-matrix-algebra}]
    Suppose that $\k$ is a normal integral domain. Then $\rho$ is an immersion of schemes and $\M(\L)_S$ is isomorphic to the normalization of the closure of $\rho(G)$. 
\end{theo}
\subsubsection{Adjoint quotients}
Let $\L\subset X^+$ be a weight monoid and let $\M=\M(\L)_S$ be the corresponding reductive $G$-monoid over $S$. Let $W=N_G(T)/T$ be the Weyl group and let $W\cdot\L\coloneqq\bigcup_{w\in W}w(\L)$. Then $W\cdot\L$ is a saturated submonoid of $X$ that generates $X$, and the classification theorem in the special case of tori says that it determines the reductive monoid $\M_T\coloneqq\Spec(\O_S[W\cdot\L])$ over $S$ with unit group $T$. Here $\O_S[W\cdot\L]$ is the sheaf of monoid algebras for $W\cdot\L$ with coefficients in $\O_S$.
\begin{theo}[Theorem \ref{theo:properties AM MT} and Theorem \ref{theo:Chevalley iso}]\label{theo:adjoint-quotient-intro}
    The scheme $\M_T$ is the schematic closure of $T$ in $\M$ and the natural embedding $\M_T\into\M$ induces an isomorphism between GIT quotients 
    \[\M_T//W\cong\M//\Ad(G)\]
    where $\Ad(G)$ denotes the adjoint action of $G$ on $\M$. 
\end{theo}
The adjoint GIT quotient of a semisimple group over a field was first described by Steinberg \cite{Steinberg-regular}. Later Lee \cite{LeeAdjointQuotient} generalized Steinberg's result to reductive group schemes over a general base scheme. Our result is a further generalization of Lee's result.

\subsubsection{Orbit closures}
Let $(\M,\iota)$ be a rigidified reductive $G$-monoid over a connected scheme $S$ with weight monoid $\L\coloneqq\L(\M)$. So we have $\O_S[\M]={}_\ZZ\bfO(\L)\otimes_\ZZ\O_S$.\par 
For any downward closed ideal $\J\subset\L$ (see Appendix \ref{sec:appendix-monoid}, especially Definition \ref{def:monoid-ideal}, for related notions), one can show (using properties of the canonical bases) that ${}_\ZZ\bfO(\J)\subset{}_\ZZ\bfO(\L)$ is an ideal. Then by taking the relative spectrum we get the following closed subscheme:
\[\M(\L/\J)\coloneqq\Spec_S(\O_S\otimes_\ZZ({}_\ZZ\bfO(\L)/{}_\ZZ\bfO(\J)))\subset\M.\]
\begin{theo}[Theorem \ref{theo:parametrization GxG orbits}]
    With notations as above, for any downward closed prime ideal $\J\subset\L$ the scheme $\M(\L/\J)$ is flat over $S$ with integral normal geometric fibers. Suppose moreover that the base scheme $S$ is nonempty and connected. Then the assignment $\J\mapsto\M(\L/\J)$ defines an order-reversing bijection between the following posets:
    \begin{itemize}
        \item[(i)] downward closed prime ideals of $\L$;
        \item[(ii)] $G\times G$-stable closed subschemes of $\M$ that are flat, of finite presentation and with integral geometric fibers over $S$.
    \end{itemize}
    In particular, the geometric $S$-fibers of the schemes in (ii) are automatically normal. 
\end{theo}
Here we say that an ideal $\J\subset\L$ is \emph{prime} if the complement $\L\setminus\J$ is a submonoid of $\L$, and we say that $\J$ is downward closed if for any $\la,\mu\in\L$ with $\mu\le\la$ and $\la\in\J$ we have $\mu\in\J$.\par
In particular, when working over an algebraically closed field (of any characteristic), we recover the classical result on the normality of orbit closures of reductive monoids (see for example \cite[Theorem 3 (5)]{Vinberg1995} in the characteristic zero case). 

\subsection{Related work}
Besides the work of Renner, Vinberg and Rittatore on reductive monoids over a field that we already discussed, let us mention the work \cite{Zhu2020}, \cite{Hove}, \cite{Nath} that constructs integral models of the Vinberg monoid without using the canonical bases, and the work \cite{BaoSongAffineEmbeddings} that constructs integral models for affine embeddings of symmetric spaces. See \cite[\S1.3]{CJ1} for further discussion of their relationship with our construction. On the other hand, the Vinberg monoid is closely related to the wonderful compactification of adjoint groups and we mention the work of Li \cite{LiShang} that construct integral models of the group compactification by algebro-geometric methods. As for the classification problem, let us also mention the works of Huruguen \cite{Huruguen} and Wedhorn \cite{Wedhorn} that classify spherical embeddings over arbitrary fields, generalizing the theory of \cite{Luna-Vust} and \cite{Knop-LunaVust} over algebraically closed fields. It would be interesting to see if the technique of this paper can be applied in the setting of spherical embeddings and generalize these works to arbitrary base schemes. 

\subsection{Organization of the article}
In \S\ref{sec:reductive-monoid-construction} we study basic properties of reductive monoids, prove the existence part of the classification theorem and the generalization of Steinberg's theorem on adjoint quotients. Then in \S\ref{sec:classification} we establish the uniqueness part and hence finish the proof of the classification theorem. \S\ref{sec:orbit} is devoted to the study of $G\times G$ orbits in a reductive monoid. Finally in \S\ref{sec:very-flat} we study the important class of very flat reductive monoids, especially the Vinberg monoid, and compare our construction with other approaches in the literature. In Appendix \ref{sec:appendix-monoid} we review some basic facts on (abstract) commutative monoids and monoid algebras following the book \cite{Ogus-Log}.

\subsection{Notations and conventions}\label{sec:notations}
We will retain the notations on based root data and reductive group schemes from \cite[\S2.1]{CJ1}. A based root datum is denoted by $\Psi=(\Delta, Y, X,\{\alpha_i,i\in\Delta\}, \{\alpha_i^\vee,i \in\Delta\})$ where $X$ (resp. $Y$) is the weight (resp. coweight) lattice, $\Delta$ is an index set, $\{\alpha_i,i\in\Delta\}$ (resp. $\{\alpha_i^\vee,i \in\Delta\}$) is the set of simple roots (resp. simple coroots). On $X$ we have the partial order $\le$ determined by $\Delta$ and $X^+\subset X$ is the submonoid of dominant weights. From $\Psi$ we construct a pinned split reductive group scheme $(G,B,T)$ over any base scheme $S$, where $B\subset G$ is a Borel subgroup and $T\subset B$ is a split maximal torus. Then we have $X=X^*(T)$ and $Y=X_*(T)$. Other notations on reductive groups are the same as \emph{loc. cit.} \S2.1.2. For instance, $Z$ denotes the maximal torus in the center $Z_G$ of $G$, $G_{\der}$ denotes the derived group of $G$, $T_\der=T\cap G_\der$, $Z_\der\coloneqq Z\cap G_{\der}$ and $G_{\ab}=G/G_{\der}\cong Z/Z_\der$ etc. Also we have the character lattices
\[X_\ab\coloneqq X^*(G_\ab)\subset X_Z\coloneqq X^*(Z), \quad X_{\der}=X^*(T_\der)=X/X_\ab.\]
We also have a natural projection $X\to X_Z$ that is compatible with the natural inclusions $X_\ab\into X$ and $X_\ab\into X_Z$. Moreover, the projections from $X$ to $X_Z$ and $X_\der$ induce an embedding $X\subset X_Z\oplus X_\der$ whose image consists of pairs of characters that agree on $Z_\der$.\par 
All commutative rings are assumed to be unital and will usually be denoted by $\k$. Fields are usually denoted by $k$. For a commutative ring $\k$ and a scheme $\mathcal{X}$ defined over $\k$, we let $\k[\mathcal{X}]$ denote the $\k$-algebra of global sections of the structure sheaf of $\mathcal{X}$. More generally, for a scheme $\mathcal{X}$ over a general base scheme $S$ with structure morphism $f\colon\mathcal{X}\to S$, we let $\O_S[\mathcal{X}]\coloneqq f_*\O_{\mathcal{X}}$, which is a sheaf of $\O_S$-algebras.\par
We will frequently use the notion of \emph{scheme-theoretic image}, for which we refer to \cite[\href{https://stacks.math.columbia.edu/tag/01R5}{Section 01R5}]{Stacks} for the precise definition. Let us recall that for a morphism between schemes $f\colon \mathcal{X}\to\mathcal{Y}$, its scheme-theoretic image is by definition a \emph{closed} subscheme of the target scheme $\mathcal{Y}$, and if $\mathcal{X}$ is reduced then it is the reduced induced scheme structure on the set-theoretic closure $\overline{f(\mathcal{X})}$ (see \cite[\href{https://stacks.math.columbia.edu/tag/056B}{Lemma 056B}]{Stacks}). If $f$ is an immersion, we also say \emph{scheme-theoretic closure, or schematic closure} instead of scheme-theoretic image.
\subsection*{Acknowledgments}
This paper is updated from part of the first version of arXiv:2607.00322, which was done without using AI tools. During the preparation of the current version, ChatGPT 6 Astra was used for proofreading and text improvements. Both author were funded by the Hua Loo-Keng Center for Mathematical Sciences, AMSS, CAS. The research of J.C. was also supported by National Key R\&D Program of China No.2023YFA1009701, National Natural Science Foundation of China (Grant No.12288201, No.12231001), CAS Project for Young Scientists in Basic Research (Grant No.YSBR-033).

\section{Reductive monoid schemes: constructions}\label{sec:reductive-monoid-construction}
In this section we first discuss some generalities on affine monoid schemes, and then prove the existence part of the main Theorem \ref{theo:classification-intro} by constructing the reductive monoid associated to any weight monoid in $X^+$. In \S\ref{subsec:adjoint-quotient} we prove the generalization of Steinberg's theorem for adjoint quotients of the reductive monoids we constructed. In \S\ref{sec:classification} we complete the classification (Theorem \ref{theo:classification-intro}), so the adjoint quotient description applies to all reductive monoids and Theorem \ref{theo:adjoint-quotient-intro} will follow.
\subsection{Definitions and basic properties}
Recall the definition of monoid schemes from Definition \ref{def:monoid-scheme}. 
\begin{lem}\label{lem:quotient-flat}
    Let $\M$ be an affine flat monoid scheme of finite presentation over a scheme $S$. Suppose that the unit group $G=\M^\times$ is (represented by) a scheme that is flat over $S$ and the structure morphism $\M\to S$ has integral geometric fibers. Let $j\colon G=\M^\times\to\M$ be the natural morphism (i.e. the action map of $G$ at the unit element of $\M$). Then the following statements hold:
    \begin{enumerate}
        \item[(i)] $j$ is an open embedding;
        \item[(ii)] the natural homomorphism of $\O_S$-modules $\O_\M\to j_*\O_G$ is universally injective (i.e. injective after base change along any morphism $S'\to S$);
        \item[(iii)] the quotient $j_*\O_G/\O_\M$ is a flat $\O_S$-module. 
    \end{enumerate}
\end{lem}
\begin{proof}
    (i) First we note that the formation of the unit group $\M^\times$ commutes with arbitrary base change by definition. Since $G$ is flat over $S$, to show that $j$ is an open embedding we may pass to geometric fibers over $S$ thanks to the fiberwise criterion \cite[Corollaire 17.9.5]{EGAIV16a23}. Then we conclude by \cite[II, \S2, Corollaire 3.6]{DemazureGabriel} or \cite[Theorem 1]{RittatoreMonoidArticle}. \par 
    (ii) The injectivity is clear after base change to geometric fibers over $S$ since the fibers are integral. Then we conclude by \cite[11.9.17]{EGAIV8a15}.\par 
    (iii) We may assume that $S=\Spec(R)$ is affine. Then $\M=\Spec(A)$, $G=\Spec(B)$ are both affine and $A$ is a subring of $B$ by (ii). For any ideal $I\subset R$, the induced map $A/IA\to B/IB$ is injective by (ii). Therefore $IA=A\cap IB$ (intersection inside $B$) and we deduce that $B/A$ is a flat $R$-module by \cite[Chapitre I, \S6 Corollaire, p.33]{BourbakiAC}.
\end{proof}

\begin{prop}\label{prop:monoid-scheme-equiv-condition}
    Let $\M\to S$ be a morphism of schemes that is affine, flat, and of finite presentation, and all geometric fibers are integral. Let $G$ be a flat group scheme over $S$. Then the following are equivalent:
    \begin{enumerate}
        \item[(i)] $\M$ is a monoid scheme whose unit group is isomorphic to $G$;
        \item[(ii)] The group scheme $G\times_S G$ acts on the $S$-scheme $\M$ and there is a $G\times_S G$-equivariant open embedding of $S$-schemes $G\hookrightarrow \M$.
    \end{enumerate}
\end{prop}
\begin{proof}
    (i)$\Rightarrow$(ii): Clearly $G\times_S G$ acts on $\M$ by left and right multiplication. By Lemma \ref{lem:quotient-flat}, the natural $G\times_S G$-equivariant embedding $G\into\M$ is open and hence (ii) follows.\par 
    (ii)$\Rightarrow$(i): By gluing and affineness over $S$, we can assume that $S=\Spec C$, $G=\Spec B$, $\M=\Spec A$ are affine. From the assumption that $\M\to S$ has integral geometric fibers, using the same argument in the proof of Lemma \ref{lem:quotient-flat}, we deduce that the natural homomorphism $A\to B$ is injective and the quotient $B/A$ is a flat $C$-module. 
    Restricting the $G\times_S G$-action to the two factors we get morphisms $G\times_S \M\to \M$ and $\M\times_S G \to \M$, which together induce a morphism  $\Omega\coloneqq(G\times_S \M)\cup_{G\times_S G}(\M\times_S G)\to \M$. Here $G\times_S\M$ and $\M\times_S G$ are both open subschemes of $\M\times_S\M$ by assumption and $\Omega$ is their union.\par 
    We have the following commutative diagram of $C$-modules in which the lines are short exact sequences
    \[\xymatrix{0\ar[r] & A\otimes _C A\ar[r] \ar[d] & B\otimes _C A \ar[r]\ar[d] & B/A\otimes _C A \ar[r]\ar[d] & 0\\
          0\ar[r] & A\otimes_C B\ar[r] & B\otimes_C B \ar[r] &  B/A\otimes_C B \ar[r] & 0}\]
    Since $B/A$ is flat, the right vertical arrow is injective and we deduce that 
    \[A\otimes_CA=(B\otimes_CA)\times_{B\otimes_CB}(A\otimes_CB)\]
    where the fiber product is in the category of $C$-algebras. Then we deduce that 
    \[\Gamma(\Omega,\O_\Omega)=\Gamma(\M\times_S\M,\O_{\M\times_S\M}).\] 
    Consequently the morphism $\Omega\to\M$ defined above extends to a morphism $\M\times_S\M\to\M$ that defines a monoid scheme structure on $\M$. It remains to show that $G=\M^\times$. For any $S$-scheme $S'$, it is clear that $G(S')\subset\M^\times(S')$. Conversely, for any element $x\in\M^\times(S')$ corresponding to a morphism $x\colon S'\to\M$ over $S$, thanks to \cite[Proposition 1]{RittatoreMonoidArticle} we know that the morphism $x$ sends any point $s\in S'$ into $G\subset\M$. Therefore the morphism $x$ factors through the open subscheme $G\subset\M$ and we conclude that $G=\M^\times$.
\end{proof}

\begin{defi}\label{def:red-monoid}
Let $G$ be a reductive group scheme over $S$. We say that an $S$-scheme $\M$ is a \emph{reductive monoid with unit group $G$}, or simply a \emph{reductive $G$-monoid} if 
    \begin{enumerate}
       \item[(i)] it is a monoid scheme over $S$ and its unit group is isomorphic to $G$ as group schemes over $S$;
       \item[(ii)] the structure morphism $\M\to S$ is affine, flat, of finite presentation, and has integral normal geometric fibers.
    \end{enumerate}
A \emph{rigidified reductive $G$-monoid} over $S$ is a pair $(\M,\iota)$ where $\M$ is a reductive monoid over $S$ and $\iota\colon G\cong\M^\times$ is an isomorphism of $S$-group schemes. A morphism between two rigidified reductive $G$-monoids $(\M,\iota)$ and $(\M',\iota')$ is a morphism of $S$-monoid schemes $\M\to\M'$ that is compatible with the rigidifications $\iota$ and $\iota'$ on the unit groups. Such a morphism induces an $\O_S$-algebra homomorphism $\O_S[\M']\to \O_S[\M]$ that commutes with their homomorphisms to $\O_S[G]$ (induced by the rigidifications), which are embeddings of subalgebras by Lemma \ref{lem:quotient-flat}. Hence such a morphism is unique if it exists. Thus the isomorphism classes of rigidified $G$-monoids form a poset. 
\end{defi}

\begin{rqe}\label{rqe:equivalence reductive monoid}
    This notion of reductive monoid scheme is preserved by base change. By Proposition \ref{prop:monoid-scheme-equiv-condition}, one can replace condition (i) in the above definition by the following condition:
    \begin{enumerate}
        \item[(i')] The group scheme $G\times_S G$ acts on the $S$-scheme $\M$ and there is a $G\times_S G$-equivariant open embedding of $S$-schemes $G\hookrightarrow \M$.
    \end{enumerate}
    Moreover, such an open embedding gives a rigidification $\iota\colon G\cong\M^\times$.
\end{rqe}

\begin{cor}
    Let $\M$ be a reductive monoid with unit group $G$. Then $\M$ is commutative if and only if $G$ is a torus.
\end{cor}
\begin{proof}
    The direct implication is clear since the reductive group scheme $G$ is commutative if and only if it is a torus.\par 
    Conversely suppose that $G$ is a torus. By Lemma \ref{lem:quotient-flat} the natural homomorphism $\O_\M\to j_*\O_G$ is universally injective and hence the induced homomorphism $\O_\M\otimes\O_\M\to j_*\O_G\otimes j_*\O_G$ is injective. So if $\O_G$ is cocommutative, then $\O_\M$ is also cocommutative. In other words, the commutativity of $G$ implies the commutativity of $\M$.
\end{proof}
\begin{cor}
    If $G$ is a semisimple group scheme over $S$, then any reductive $G$-monoid is isomorphic to $G$ itself.
\end{cor}
\begin{proof}
    Let $\M$ be a reductive $G$-monoid. After choosing a rigidification we get an open embedding $G\into\M$ by Lemma \ref{lem:quotient-flat}. Since its restriction to any geometric fiber over $S$ is an isomorphism by \cite[Corollary 2]{RittatoreMonoidArticle}, it is a surjective open embedding and hence an isomorphism. 
\end{proof}

\begin{defi}\label{def:MH AM As} 
    Let $\M$ be a reductive monoid with unit group $G$ and let $G_{\der}$ be the derived group of $G$. The \emph{abelianisation} of $\M$ is defined to be the GIT quotient 
    \[A_{\M}\coloneqq\M//(G_{\der}\times G_{\der})=\Spec(\O_S[\M]^{G_{\der}\times G_{\der}}).\]
    By construction we have a natural morphism of $S$-schemes $\alpha_\M\colon\M\to A_{\M}$, called the \emph{abelianisation map}. 
\end{defi}

\begin{defi}\label{def:very flat}
    A reductive monoid $\M$ with unit group $G$ is said to be \emph{very flat} if the abelianisation map $\alpha_\M\colon\M\to A_{\M}$ is flat and all its geometric fibers are integral.
\end{defi}

\begin{defi}
    A monoid $\M$ over $S$ with unit group $G$ \emph{has a zero} if there exists a section $0_\M\colon S\to \M$ of the structure morphism $f\colon\M\to S$ such that the following diagrams commute:
    \[\xymatrix@C=2cm{
    \M\ar[r]^{0_\M\times\mathrm{Id}_\M}\ar[d]_f & \M\times_S\M\ar[d]^{\pi_\M}\\
    S\ar[r]^{0_\M} & \M,
    }\quad\quad\quad 
    \xymatrix@C=2cm{
    \M\ar[r]^{\mathrm{Id}_\M\times0_\M}\ar[d]_f & \M\times_S\M\ar[d]^{\pi_\M}\\
    S\ar[r]^{0_\M} & \M
    }\]
    where $\pi_\M$ is the product morphism for $\M$. More concretely, this means that for any $S$-scheme $S'$ and any element $x\in \M(S')$, we have $x\cdot0_{\M,S'}=0_{\M,S'}\cdot x=0_{\M,S'}$ in the monoid $\M(S')$, where $0_{\M,S'}\colon S'\to\M$ is the composition of the structure morphism $S'\to S$ with $0_\M$. If a zero $0_\M$ exists, then it is necessarily unique and we say that $0_\M$ is \emph{the zero} of $\M$. 
\end{defi}
For example all the matrix monoids have a zero (given by the zero matrices), while on the other hand nontrivial group schemes do not have a zero. We will see in Corollary \ref{cor:monoid-decomposition} that any reductive monoid scheme decomposes as an almost direct product of a reductive group and a reductive monoid with zero. 

\subsection{Construction of reductive monoid schemes}
We fix a based root datum $\Psi=(\Delta,Y,X,...)$ and let $(G,B,T)$ be the associated split reductive group scheme over some base scheme $S$ as in \S\ref{sec:notations}.

\subsubsection{Weight monoids and cones}\label{sec:cones}
Recall that the simple roots define a partial order $\le$ on the weight lattice $X$.
\begin{defi}\label{def:downward-closed}
    Let $P\subset X_\QQ$ be a subset. A subset $Q\subset P$ is said to be \emph{downward closed in $P$} if for any elements $\la,\mu\in P$ such that $\mu\le\la$ and $\la\in Q$, we have $\mu\in Q$.
\end{defi}
\begin{defi}\label{def:weight-monoid}
    A \emph{weight monoid} in $X^+$ is a submonoid $\L\subset X^+$ that satisfies the following conditions:
    \begin{itemize}
        \item $\L$ is finitely generated, saturated and $\L^{\gp}=X$ (see Definition \ref{def:monoid-properties} for these notions);
        \item $\L$ is downward closed in $X^+$.
    \end{itemize}
\end{defi}
As a basic example, $X^+$ itself is a weight monoid.\par 
Let $\L\subset X^+$ be a weight monoid. Let $\QQ_{\ge0}\L$ be the rational convex polyhedral cone in $X_{\QQ}\coloneqq X\otimes_\ZZ\QQ$ generated by $\L$. In particular, when $\L=X^+$ we get the Weyl chamber which we simply denote by $C\coloneqq \QQ_{\ge0}X^+$.\par 
Among the supporting hyperplanes of facets of $\QQ_{\ge0}\L$ in $X_\QQ$, we pick out those that are not root hyperplanes and denote them by $H_1,\dots,H_s$. Let $H_1^+,\dots,H_s^+$ be the corresponding half-spaces in $X_\QQ$ containing $\QQ_{\ge0}\L$. We consider the following rational convex polyhedral cones in $X_{\QQ}$:
\[K_\L\coloneqq\bigcap_{i=1}^sH_i^+,\quad \tilde{K}_{\L}\coloneqq\bigcap_{w\in W}w(K_\L)=\bigcap_{w\in W, 1\leq i\leq s}w(H_i^+).\]
When $s=0$ we have $\L=X^+$ and $K_\L=\tilde{K}_\L=X_\QQ$. 
\begin{lem}\label{lem:K-L-properties}
    With notations as above, the following properties hold:
    \begin{enumerate}
    \item The cone $K_\L$ is the largest among rational convex polyhedral cones $K'\subset X_\QQ$ satisfying $K'\cap X^+=\L$. In particular, $\L=K_\L\cap X^+$.
    \item The cone $K_\L$ contains all the negative roots.
    \item $\tilde{K}_\L\cap C=K_\L\cap C=\QQ_{\ge0}\L$.
    \end{enumerate}
\end{lem}
\begin{proof}
    (1) By definition we have $K_\L\cap C=\QQ_{\ge0}\L$. Intersecting both sides with $X$ we get that $K_\L\cap X^+=\QQ_{\ge0}\L\cap X=\L$. Let $K'\subset X_\QQ$ be a rational convex polyhedral cone such that $K'\cap X^+=\L$. By definition $K'$ is the intersection of finitely many half spaces of $X_\QQ$. The condition $K'\cap X^+=\L$ is equivalent to $K'\cap C=\QQ_{\ge0}\L$, and implies that the half spaces $H_i^+$, $1\le i\le s$ defining $\QQ_{\ge0}\L$ (and distinct from half-spaces defined by the roots) occur among the half spaces defining $K'$. Therefore $K'\subset K_\L$ and this proves the maximal property of $K_\L$. \par
    (2) It suffices to show that $K_\L$ contains all the negative simple roots. Suppose on the contrary that there exists a simple root $\a$ and $1\leq i\leq s$ such that $-\a\notin H_i^+$. Let $\ell_i\in X_\QQ^*$ be a linear functional that is non-negative on $H_i^+$ and vanishes on $H_i$. Since $H_i$ is a defining hyperplane of $\QQ_{\ge0}\L$ and is not a root hyperplane, we can find an element $x\in \QQ_{\ge0}\L\cap H_i$ that does not lie on any root hyperplane. After multiplying by a positive scalar, we may assume that $x\in\L\cap H_i$ and is strictly dominant, i.e. lies in the interior of the Weyl chamber. Let $y\coloneqq 2x-\a$. Since $x\in X^+$ is strictly dominant, we have $\langle\alpha_j^\vee,x\rangle\ge1$ and hence $\langle\alpha_j^\vee,y\rangle\ge0$ for all $j\in\Delta$ (use that $\langle\alpha_j^\vee,\alpha\rangle\le0$ if $\alpha_j\ne\alpha$ and $\langle\alpha^\vee,\alpha\rangle=2$). This shows that $y\in X^+$. Since $y\leq 2x$ and $\L$ is downward closed in $X^+$, we deduce that $y\in\L\subset H_i^+$ and hence $\ell_i(y)\geq 0$. On the other hand, recall that $x\in H_i$ and $-\alpha\notin H_i^+$ by assumption, so we have $\ell_i(y)=\ell_i(2x)-\ell_i(\a)=-\ell_i(\a)<0$ and this is a contradiction.\par 
    (3) The inclusion $\subseteq$ is clear. For the converse, let $x\in K_\L\cap C$ and $w\in W$. We want to prove $w\cdot x\in K_\L$. Since $x\in C=\QQ_{\ge0}X^+$, we have $x-w(x)\in\QQ_{\ge0}^\Delta$. Then by (2) we get $w(x)-x\in K_\L$ and so $w\cdot x\in K_\L$.
\end{proof}

\begin{cor}\label{cor:WL}
    Let $\L\subset X^+$ be a weight monoid and let $W\cdot \L\coloneqq\bigcup_{w\in W}w(\L)$. Then we have: 
    \begin{enumerate}
        \item $W\cdot \QQ_{\ge0}\L=\tilde{K}_\L$ and $W\cdot\L=\tilde{K}_\L\cap X$,
        \item $W\cdot\L$ is a finitely generated saturated submonoid of $X$ and generates it as an abelian group. In particular, $\tilde{K}_\L$ is a cone of full dimension in $X_\QQ$. 
    \end{enumerate}
\end{cor}
\begin{proof}
    (1) From Lemma \ref{lem:K-L-properties}(3) and the fact that $\tilde{K}_\L$ is $W$-stable we get
    \[W\cdot \QQ_{\ge0}\L=W\cdot(\tilde{K}_\L\cap C)=\tilde{K}_\L\cap W\cdot C=\tilde{K}_\L.\]
    After taking the intersection of both sides with $X$ we get the second equality.\par
    (2) follows from (1) and the fact that $\tilde{K}_\L$ is a convex polyhedral cone in $X_\QQ$ that spans $X_\QQ$. 
\end{proof}

\begin{cor}\label{cor:W-stable-cone}
    The map $\L\mapsto\QQ_{\ge0}W\cdot\L=W\cdot\QQ_{\ge0}\L$ defines a bijection between weight monoids in $X^+$ and $W$-stable convex polyhedral cones in $X_\QQ$ that span $X_\QQ$, with the inverse map given by $K\mapsto K\cap X^+$. 
\end{cor}
\begin{proof}
    According to Corollary \ref{cor:WL}, for any weight monoid $\L\subset X^+$ we have $\QQ_{\ge0}W\cdot\L=\tilde{K}_\L$, which is a convex polyhedral cone by definition. So the map is well-defined. Let $K\subset X_\QQ$ be a $W$-stable convex polyhedral cone that spans $X_\QQ$. The intersection $\L_K\coloneqq K\cap X^+$ is clearly a saturated submonoid of $X^+$ that generates $X$ and it remains to show that it is downward closed. Let $\la\in\L_K$ and $\mu\in X^+$ with $\mu\le\la$. For any $w\in W$ we have $w(\mu)\le\mu\le\la$ and therefore $\mu$ lies in the intersection $\bigcap\limits_{w\in W}w(\la-\QQ_{\ge0}^\Delta)$, which is well-known to be equal to the convex hull $\mathrm{conv}(W\cdot\la)$ of the $W$-orbit of $\la$. Since $K$ is $W$-stable, we have $\mathrm{conv}(W\cdot\la)\subset K$ and so $\mu\in K\cap X^+=\L_K$. This shows that $\L_K$ is downward closed in $X^+$ and so the map $K\mapsto\L_K$ is well-defined. It is clear that the two maps are inverse to each other.
\end{proof}
The following lemma will be used in the proof of Proposition \ref{prop:very-flat-characterisation}.
\begin{lem}\label{lem:L-X-der+surj}
    Let $\L\subset X^+$ be a weight monoid. Then the natural projection $X\to X/X_{\ab}=X_{\der}$ restricts to a surjection from $\L$ to $X_{\der}^+$. 
\end{lem}
\begin{proof}
    Let $p\colon X\epic X_{\der}$ be the natural projection. Take an element $\mu\in X_{\der}^+$. Since $p$ restricts to a surjection $X^+\epic X_{\der}^+$, there exists $\tilde{\mu}\in X^+$ with $p(\tilde{\mu})=\mu$. If $s=0$, then $\L=X^+$ and the statement is trivial. Assume from now on that $s>0$ and so $\L\ne X^+$. 
    For each $1\le i\le s$, let $l_i\in X^*_\QQ$ be a nonzero linear functional that is non-negative on $K_\L$ and vanishes on the hyperplane $H_i$. Since $\tilde{K}_\L=W\cdot\QQ_{\ge0}\L$ is a cone of full dimension in $X_\QQ$ by Corollary \ref{cor:WL}, we can find a point in its interior. After taking the average of its $W$-orbit and clearing denominators, we get an element $\nu\in X_\ab$ such that $l_i(\nu)>0$ for all $1\le i\le s$. Then we can choose $n\in\ZZ_{>0}$ sufficiently large such that $n\nu+\tilde{\mu}\in K_\L\cap X^+=\L$ and $p(n\nu+\tilde{\mu})=\mu$.
\end{proof}

\subsubsection{Reductive monoid associated to a weight monoid}\label{sec:RM-associated-to-WM}
Let us first review some constructions and results from \cite{CJ1}. Suppose for the moment that the base $S=\Spec(\k)$ is affine, where $\k$ is a commutative ring. The coordinate ring of $G$ is denoted by ${}_\k\bfO=\k[G]$ and the \emph{dual canonical basis} $\dot{\bfB}^*$ is a $\k$-basis of ${}_\k\bfO$ that admits a partition
\[\dot{\bfB}^*=\bigsqcup_{\la\in X^+}\dot{\bfB}[\la]^*\]
into finite subsets $\dot{\bfB}[\la]^*$. We think of this partition as a replacement of the Peter-Weyl decomposition (in characteristic $0$) over general base. For any subset $P\subset X^+$, we denote 
\[\dot{\bfB}[P]^*\coloneqq\bigsqcup_{\la\in P}\dot{\bfB}[\la]^*\]
and we let ${}_\k\bfO(P)\subset{}_\k\bfO$ be the $\k$-submodule spanned by elements in $\dot{\bfB}[P]^*$ (see \cite[Definition 4.1]{CJ1}). Also for any $\mu\in X^+$, we let ${}_\k\bfO_{\le\mu}=\k[G]_{\le\mu}\subset\k[G]$ be the $\k$-submodule spanned by elements in $\bigsqcup_{\la\in X^+,\la\le\mu}\dot{\bfB}[\la]^*$. \par 
For any weight $\la\in X$, there is a distinguished element $1_\la^*\in\dot{\bfB}^*$ that will play important roles in the following. The element $1_\la^*$ lies in $\dot{\bfB}[\la_+]^*$ where $\la_+\in X^+$ is in the $W$-orbit of $\la$. In particular, when $\la\in X^+$ we have $1_\la^*\in\dot{\bfB}[\la]^*$. Moreover, we have $1_\la^*\cdot 1_\mu^*=1_{\la+\mu}^*$ for all $\la,\mu\in X^+$ and altogether the elements $\{1_\la^*,\la\in X^+\}$ span the subalgebra $\k[G]^{U^-\times U}$ of invariants under the left and right regular representation of $U^-\times U$. Here $U\subset B$ is the unipotent radical and $U^-\subset B^-$ is the unipotent radical of the opposite Borel $B^-$. In fact let us recall the following more general result:
\begin{prop}\label{prop:finite-invariant-review}
    Let $\L\subset X^+$ be a weight monoid and $\k$ be a commutative ring. The following statements hold.
    \begin{enumerate}
        \item The $\k$-module ${}_\k\bfO(\L)$ is a sub-bialgebra of ${}_\k\bfO=\k[G]$ and is finitely presented as $\k$-algebra. 
        \item For any $\la,\mu\in X^+$ we have ${}_\k\bfO_{\le\la}\cdot{}_\k\bfO_{\le\mu}\subset{}_\k\bfO_{\le\la+\mu}$ and $1_\la^*\cdot 1_\mu^*=1_{\la+\mu}^*$.
        \item There is an isomorphism of $\k$-algebras
            \[\k[\L]\xrightarrow{\sim}{}_\k\bfO(\L)^{U^-\times U},\quad e^\la\mapsto 1_\la^*\]
            where $\k[\L]$ is the monoid algebra with basis $\{e^\la,\la\in\L\}$ and the right hand side is the invariants under the left and right regular representation of $U^-\times U$.
        \item If $\k$ is a noetherian normal integral domain, then ${}_\k\bfO(\L)$ is a noetherian normal integral domain.
    \end{enumerate}
\end{prop}
\begin{proof}
    (1) By \cite[Proposition 4.5, Proposition 4.2]{CJ1}, the assumption on $\L$ implies that the $\k$-algebra ${}_\k\bfO(\L)$, which is free as a $\k$-module, is a sub-bialgebra of $\k[G]$. Since $\L$ is finitely generated, the $\ZZ$-algebra ${}_\ZZ\bfO(\L)$ is finitely generated by \emph{loc. cit.} Corollary 4.13 and hence finitely presented (since $\ZZ$ is a noetherian ring). Thus the base change ${}_\k\bfO(\L)={}_\ZZ\bfO(\L)\otimes_\ZZ\k$ is a finitely presented $\k$-algebra.\par 
    (2) is \cite[Lemma 3.9]{CJ1}, (3) follows from \cite[Corollary 5.11]{CJ1} and (4) follows from \cite[Theorem 4.16]{CJ1}.
\end{proof}

\begin{theo-def}\label{theo:M(L)-def}
    For any weight monoid $\L\subset X^+$ and any base scheme $S$, the scheme 
    \[\M(\L)_S\coloneqq\Spec({}_\ZZ\mathbf{O}(\L)\otimes_\ZZ\O_S)\] 
    is a reductive monoid scheme over $S$ and there exists an isomorphism of $S$-group schemes 
    \[\iota_\L\colon G\cong\M(\L)_S^\times.\] 
    If $S=\Spec(\k)$ is an affine scheme, we also write $\M(\L)_\k=\M(\L)_S$. The subscripts ``$S$" or ``$\k$" will often be omitted when they are clear from the context. 
\end{theo-def}
\begin{proof}
    We may assume $S=\Spec(\k)$ is affine. For simplicity denote $\M\coloneqq\Spec({}_\k\mathbf{O}(\L))$. From Proposition \ref{prop:finite-invariant-review}(1) we deduce that $\M$ is an affine flat monoid over $\k$ of finite presentation and the natural homomorphism $G\to\M$ corresponding to the embedding ${}_k\bfO(\L)\subset\k[G]$ is $G\times G$-equivariant. Then by Proposition \ref{prop:finite-invariant-review}(4), all geometric fibers of $\M\to S$ are integral and normal. According to Remark \ref{rqe:equivalence reductive monoid}, it remains to show that the morphism $G\to\M$ induced by the inclusion ${}_\k\bfO(\L)\subset\k[G]$ is a $G\times G$-equivariant open embedding (then it will induce the isomorphism of group schemes $\iota\colon G\cong\M^\times$). Since $G$ and $\M$ are flat and of finite presentation over $\Spec(\k)$, it suffices to check it at the level of geometric fibers by \cite[Corollaire 17.9.5]{EGAIV16a23}. So we may and do assume that $\k=k$ is an algebraically closed field.\par  
    Let $\lambda\in X^+$ and $f\in k[G]_{\leq\lambda}$. Since $\L$ generates $X$ as an abelian group, there exists $\mu,\mu'\in\L$ such that $\lambda=\mu-\mu'$. From Proposition \ref{prop:finite-invariant-review}(2) and the fact that $\L$ is downward closed, we deduce that $1_{\mu'}^*\cdot f\in k[G]_{\leq \mu}\subset{}_k\bfO(\L)$ and hence $f\in\mathrm{Frac}({}_k\bfO(\L))$. This shows that $k[G]_{\leq \lambda}\subset\mathrm{Frac}({}_k\bfO(\L))$ for any $\lambda \in X^+$ and so $k[G]\subset\mathrm{Frac}({}_k\bfO(\L))$. From this we get that the natural morphism $G\to\M$ is birational. In particular $\dim G=\dim \M$. On the other hand, since it is a morphism of $k$-monoids, it factors through a morphism of $k$-groups $\iota\colon G\to\M^\times$. By Lemma \ref{lem:quotient-flat}, the unit group $\M^\times$ is an affine open subvariety of $\M$ and hence normal (in fact it is smooth over $k$). Since $\iota(G)$ is a closed subgroup of $\M^\times$ by \cite[Chapter I, Corollary 1.4]{Borel}, birationality forces $\iota$ to be surjective. Since $G$ and $\M^\times$ have the same dimension, the kernel of $\iota$ is discrete and hence $\iota$ is finite by \cite[VI-B, Corollaire 1.4.1]{SGA3-1}. Thus we have shown that $\iota$ is a finite birational morphism to the normal $k$-variety $\M^\times$ and then we conclude by Zariski's Main Theorem that $\iota$ is an isomorphism. This finishes the proof.
\end{proof}

\begin{lem-def}\label{lem:M_H}
    Let $\M=\M(\L)_S$ be the rigidified reductive $G$-monoid over a scheme $S$ associated to a weight monoid $\L\subset X^+$ and let $H\subset G$ be a closed subgroup scheme. We define $\M_H$ to be the schematic closure of $H$ in $\M$. Suppose moreover that $H$ and $\M_H$ are both flat over $S$. Then $\M_H$ is a submonoid scheme of $\M$ and it is commutative if $H$ is commutative. 
\end{lem-def}
\begin{proof}
    We may assume that $S=\Spec(\k)$ is affine, since the formation of scheme-theoretic image commutes with any flat base change (see  \cite[\href{https://stacks.math.columbia.edu/tag/081I}{Lemma 081I}]{Stacks}). Let $I\subset\k[G]$ be the Hopf ideal such that $\k[H]=\k[G]/I$ and let $I_\M\coloneqq I\cap\k[\M]$. Then we have $\k[\M_H]=\k[\M]/I_\M$ by \cite[\href{https://stacks.math.columbia.edu/tag/056A}{Example 056A}]{Stacks}. If $H$ and $\M_H$ are both flat over $S$, then the natural map $\k[\M_H]\otimes_\k\k[\M_H]\to\k[H]\otimes_\k\k[H]$ is injective. 
    Consider the following commutative diagram
    \[\xymatrix{
    \k[\M]\ar[r]^{\Delta_{\k[\M]}\qquad}\ar[d] & \k[\M]\otimes_\k\k[\M]\ar[d]\\
    \k[H]\ar[r]^{\Delta_{\k[H]}\qquad} & \k[H]\otimes_\k\k[H]
    }\]
    in which the horizontal maps are the coproducts. By the discussions above, the images of the left and right vertical maps are $\k[\M_H]$ and $\k[\M_H]\otimes_\k\k[\M_H]$, respectively, with the latter embedded by the injection above. Therefore the coproduct on $\k[\M]$ induces a coproduct on $\k[\M_H]$. Also the counit on $\k[\M]$ factors through the counit for $\k[H]$ and hence induces a counit for $\k[\M_H]$. Thus $\M_H$ is a submonoid scheme of $\M$. Moreover, if $H$ is commutative then the coproduct $\Delta_{\k[H]}$, and hence the induced coproduct on $\k[\M_H]$, is cocommutative. So $\M_H$ is also commutative.
\end{proof}

\begin{theo}\label{theo:properties AM MT} 
Let $S$ be a scheme and let $\L\subset X^+$ be a weight monoid that defines the reductive monoid $\M\coloneqq\M(\L)_S$ with abelianisation map $\alpha_\M\colon\M\to A_{\M}$.
\begin{enumerate}
    \item The schematic closure $\M_T$ is the rigidified reductive $T$-monoid over $S$ associated to the weight monoid $W\cdot\L\subset X^*(T)$. In particular, the formation of $\M_T$ commutes with arbitrary base change and we have:
    \[\M_T\cong\Spec(\O_S[W\cdot\L]).\]
    \item The restriction of the abelianisation map to $\M_{Z}$, the schematic closure of $Z$ in $\M$, induces an isomorphism $\M_{Z}//Z_{\der}\cong A_{\M}$ and realizes $A_\M$ as the rigidified reductive $G_{\ab}\cong Z/Z_{\der}$-monoid over $S$ associated to the weight monoid $\L_{\ab}=\L\cap X_{\ab}$. In particular we have 
    \[A_\M\cong\Spec(\O_S[\L_{ab}])\] 
    and hence the formation of the abelianisation $A_\M$ commutes with arbitrary base change.
    \item We have natural isomorphisms of group schemes 
    \[G\cong \M\times_{A_{\M}}G_{\ab},\qquad G_{\der}\cong\M\times_{A_{\M}}S\]
    where the map $S\to A_{\M}$ in the second fiber product corresponds to the unit element of $A_\M(S)^\times=G_{ab}(S)$.
\end{enumerate}
\end{theo}
\begin{proof}
    (1) Since the formation of scheme-theoretic image commutes with any flat base change (see  \cite[\href{https://stacks.math.columbia.edu/tag/081I}{Lemma 081I}]{Stacks}), we can assume that $S=\Spec(\k)$ is affine.\par
    First we show that $\M_T=\Spec(\O_S[W\cdot\L])$. It suffices to prove that the image of the composition $\k[\M]\to\k[G]\to\k[T]\simeq \k[X]$ (the group algebra of $X$) equals the subalgebra $\k[W\cdot \L]\subset\k[X]$. In terms of the canonical basis of $\k[\M]={}_\k\bfO(\L)$, the composition sends $1_\la^*$ to $e^\la$ for any $\la\in X$, and sends any other member of the canonical basis to $0$. By \cite[Corollary 2.14]{CJ1}, we have $1_{\la}^*\in \dot{\bfB}[\mathcal{L}]^*$ if and only if $\la\in W\cdot\L$. This implies the desired result.\par
    Therefore $\M_T$ is an affine monoid scheme that is flat and finitely presented over $S=\Spec(\k)$ with integral geometric fibers. From \cite[\href{https://stacks.math.columbia.edu/tag/01RG}{Lemma 01RG}]{Stacks} we know that $T$ is an open subscheme of $\M_T$. Thus $\M_T$ is a monoid scheme with unit group $T$ by Proposition \ref{prop:monoid-scheme-equiv-condition}. Moreover, the monoid $W\cdot\L$ is saturated in $X=X^*(T)$ since $\L$ is saturated in $X^+$, and so the geometric $S$-fibers of $\M_T$ are normal. We conclude that $\M_T$ is a reductive monoid over $S$ with unit group $T$ and weight monoid $W\cdot\L$.\par 
    (2) We can again assume that the base $S=\Spec(\k)$ is affine. Let $p_Z\colon X\to X_Z$ be the natural projection. Then $p_Z(\L)$ is a submonoid of $X_Z$ and from the proof of Lemma \ref{lem:M_H} we see that $\M_Z\cong\Spec(\k[p_Z(\L)])$ where $\k[p_Z(\L)]$ is the monoid algebra\footnote{Note that in general $p_Z(\L)$ is not saturated in $X_Z$ so the geometric fibers of $\M_Z$ over $S$ may not be normal.}. \par 
    Recall that we have natural embeddings $X_\ab\into X_Z$ and $X_\ab\into X$ so that $p_Z$ restricts to the identity on $X_\ab$. We claim that $p_Z(\L)\cap X_\ab=\L_\ab\coloneqq\L\cap X_\ab$. The reverse inclusion ``$\supset$" is immediate. Conversely, take $\la\in\L$ such that $\chi\coloneqq p_Z(\la)\in X_\ab$. Then we have 
    \[\la-\chi\in\ker(p_Z)\cap X^+\subset \QQ_{\ge0}^\Delta\]
    and hence there exists an integer $N\in\ZZ_{>0}$ such that $N\chi\le N\la$. Since $N\la\in\L$, $N\chi\in X_\ab\subset X^+$ and $\L$ is downward closed, we have $N\chi\in\L$ and therefore $\chi\in\L\cap X_\ab=\L_\ab$ because $\L$ is saturated. This proves the claim and consequently we get
     \[\M_{Z}//Z_{\der}=\Spec(\k[p_Z(\L)]^{Z_{\der}})=\Spec(\k[p_Z(\L)\cap X_\ab])=\Spec(\k[\L_{ab}]).\]
    On the other hand, from \cite[Lemma 5.2]{CJ1} and Proposition \ref{prop:finite-invariant-review}(3) we deduce that
    \[\k[A_{\M}]=({}_\k\bfO(\L))^{G_{\der}\times G_{\der}}=({}_\k\bfO(\L))^{U^-T_{\der}\times UT_{\der}}\cong\k[\L_{ab}].\]
    In particular, the formation of $A_\M$ commutes with arbitrary base change. Combining the equalities above, we conclude that $\M_Z//Z_\der\cong A_\M$.\par 
   
    (3) The natural homomorphism $f\colon G\to\M\times_{A_\M}G_{ab}$ commutes with the natural maps to $G_\ab$ and is an open embedding because its composition with the open embedding $\M\times_{A_\M}G_{ab}\into\M$ (obtained by base change of the open embedding $G_\ab\into A_\M$) is the natural morphism $G\into\M$ given by the rigidification and is open by Lemma \ref{lem:quotient-flat}. On geometric fibers (over $S$) the morphism $f$ is an isomorphism by \cite[Theorem 3]{Vinberg1995} in characteristic $0$ and \cite[Theorem 12]{RittatoreThese} in general. Therefore $f$ is a surjective open embedding and hence an isomorphism over $G_\ab$. Taking fiber over the unit section of $G_\ab$ we get the isomorphism $G_{\der}\cong\M\times_{A_\M}S$.
\end{proof}

\subsubsection{Embedding in matrix monoids}
The simplest reductive monoids are the matrix monoids, i.e. the multiplicative monoids of matrix algebras. After imposing some assumptions on the base schemes, we will show that all reductive monoids can be realized concretely inside a product of matrix monoids. \par
For any dominant weight $\la\in X^+$, recall that the Weyl module ${}_\ZZ\La_\la$ is a free $\ZZ$-module of finite rank and is a representation of $G_\ZZ$ (see \cite[Proposition 5.3]{CJ1}). For any scheme $S$, let $\La_{\la,S}$ be the finite free $\O_S$-module obtained as the base change of ${}_\ZZ\La_\la$. When $S=\Spec(\k)$ is an affine scheme, we have $\La_{\la,\Spec(\k)}={}_\k\La_\la$, in the notation we introduced before. Let $\underline{\mathrm{End}}(\La_{\la,S})$ be the $S$-scheme of endomorphisms of $\La_{\la,S}$ and let $\underline{\mathrm{GL}}(\La_{\la,S})$ be its open subscheme of invertible endomorphisms. Then $\underline{\mathrm{End}}(\La_{\la,S})$ is a monoid scheme whose unit group scheme is $\underline{\mathrm{GL}}(\La_{\la,S})$. Let $\rho_{\la,S}\colon G_S\to\underline{\mathrm{GL}}(\La_{\la,S})$ be the group scheme homomorphism defined by the representation. 

\begin{prop}\label{prop:lifting hatU module to hatUM module}
    For any $\la\in\L$, the composition of $\rho_\la$ with the open embedding $\underline{\mathrm{GL}}(\La_{\la,S})\subset\underline{\mathrm{End}}(\La_{\la,S})$ extends uniquely to a homomorphism of monoid schemes over $S$: 
    \[\M(\L)_S\to\underline{\mathrm{End}}(\La_{\la,S}).\]
\end{prop}
\begin{proof}
    Since everything is obtained by base change from $\ZZ$, we may assume that $S=\Spec(\ZZ)$. Let $\tilde{\rho}_\la\colon G_\ZZ\to\underline{\mathrm{End}}({}_\ZZ\La_{\la})$ be the composition of $\rho_\la$ with the natural open embedding and let \[\tilde{\rho}_\la^*\colon\mathrm{Sym}({}_\ZZ\La_{\la}\otimes{}_\ZZ\La_\la^*)\to{}_\ZZ\bfO\]
    be the corresponding homomorphism between coordinate rings, where ${}_\ZZ\La_\la^*\coloneqq\Hom_\ZZ({}_\ZZ\La_\la,\ZZ)$. It suffices to show that the image of $\tilde{\rho}_\la^*$ lies in the subring ${}_\ZZ\bfO(\L)$ if $\la\in\L$. Since the map $\tilde{\rho}_\la^*$ is given by taking matrix coefficients, this is equivalent to showing that for any $\mu\in X^+\setminus\L$, any canonical basis element $b\in\dot{\bfB}[\mu]$ acts by $0$ on ${}_\ZZ\La_\la$. Suppose this is not the case, then $\mu\le\la$ by \cite[Lemma 2.13]{CJ1}. But this contradicts the fact that $\L$ is downward closed and we are done. 
\end{proof}
\begin{theo}\label{theo:embed-G-in-matrix}
    Let $\lambda_1, \dots, \lambda_r\in X^+\setminus\{0\}$ be elements that generate the abelian group $X$. Then the natural action morphism 
    \begin{equation}\label{eq:embedding into highest weights representations}
        G\to \prod_{1\leq i\leq r}\underline{\mathrm{End}}(\Lambda_{\lambda_i, S}) 
    \end{equation} 
    is an immersion of schemes.
\end{theo}
\begin{proof}
    We may assume that $S=\Spec(\ZZ)$. Since (\ref{eq:embedding into highest weights representations}) factorizes as 
    \[\rho\colon G\to \prod_{1\leq i\leq r}\underline{\mathrm{GL}}({}_{\ZZ}\Lambda_{\lambda_i})\]
    followed by the open embedding 
    \[\prod_{1\leq i\leq r}\underline{\mathrm{GL}}({}_{\ZZ}\Lambda_{\lambda_i}) \to \prod_{1\leq i\leq r}\underline{\mathrm{End}}({}_{\ZZ}\Lambda_{\lambda_i}),\] 
    it suffices to prove that the former map $\rho$ is a closed embedding. Since $G$ is a reductive group scheme and $\prod\limits_{1\leq i\leq r}\underline{\mathrm{GL}}({}_{\ZZ}\Lambda_{\lambda_i})$ is a finitely presented and separated group scheme, according to \cite[XVI, 1.5(a)]{SGA3VIIIaXVIII} or \cite[Theorem 5.3.5]{ConradReductiveGroupScheme}, it suffices to prove that $\rho$ is a monomorphism, or equivalently, that its scheme-theoretic kernel $K\coloneqq\mathrm{Ker} \rho$ is trivial. We prove that the unit section $\Spec\ZZ\hookrightarrow K$ is a surjective open immersion.\par
    First we show that the unit section $\Spec\ZZ\hookrightarrow K$ is an open immersion. It suffices to show that the intersection $(U^-\times T \times U)\times_G K $ of $K$ with the open big cell $U^-\times T \times U \hookrightarrow G$ coincides with the unit section. Let $\k$ be a commutative ring. We will use the distribution algebra ${}_\k\hat{\bfU}(H)$ associated to any algebraic group $H$ over $\k$, see \cite[\S5.1 and Definition 3.4]{CJ1}. Let $\xi=u^-tu^+ \in ((U^-\times T \times U)\times_GK)(\k)$ be an element. Viewing $u^-\in U^-(\k)$ as an element of ${}_\k\hat{\mathbf{U}}^-$ (see \emph{loc. cit.} equation (3.1)-(3.3)), we can write
    \[u^-=\sum_{b\in \mathbf{B},\lambda \in X}x_b\cdot b^{-}1_{\lambda}\]
    where $x_b\in\k$, $b\in\mathbf{B}$ and also $x_1=1$. For any dominant $\lambda$, we let $\eta_{\lambda}\in{}_{\ZZ}\Lambda_{\lambda}$ be a highest weight vector of weight $\lambda$. Then for any $i$ we have 
    $$\eta_{\lambda_i}=\xi\cdot \eta_{\lambda_i}=u^-t\cdot (u^+\cdot \eta_{\lambda_i})=u^-t\cdot \eta_{\lambda_i}=\lambda_i(t)u^-\cdot \eta_{\lambda_i}=\sum_{b\in \mathbf{B}(\lambda_i)}\lambda_i(t)x_bb^-1_{\lambda_i}\cdot \eta_{\lambda_i}.$$ 
    Since the elements $\{b^-1_{\lambda_i}\cdot \eta_{\la_i},b\in\bfB(\la_i)\}$ form a basis of ${}_\ZZ\La_{\la_i}$ and $x_1=1$, we have $\lambda_i(t)=1$ and $x_b=0$ for any $b\in \bfB(\lambda_i)\setminus\Set{1}$. By assumption, the elements $\{\lambda_i, 1\leq i\leq r\}$ generate $X$, so they separate the elements of $T(\k)$ and we have $t=1$. Moreover, by \cite[Proposition 3.6]{CJ1}, we also deduce that $u^-=1$. Therefore, we have $\xi=u^+\in U(\k)$. Since $K$ is normal in $G$ and $U$ is conjugate to $U^-$, there is a $G(\k)$-conjugate of $\xi$ that lies in $U^-(\k)\cap K(\k)$. Applying the arguments above to this element we conclude that $\xi=1$. Hence $(U^-\times T \times U)\times_GK $ coincides with the unit section.\par
    It remains to prove the surjectivity of the unit section $\Spec(\ZZ)\hookrightarrow K$. Since $K$ is of finite type over $\Spec(\ZZ)$, it suffices to prove that the group of $k$-points $K(k)$ is trivial for any algebraically closed field $k$. Since $G(k)$ is covered by the conjugates of ${B}(k)$ and admits $K(k)$ as a normal subgroup, it suffices to show that $B(k)\cap K(k)$ is trivial. But this follows from the previous arguments.
\end{proof}

\begin{rqe}
    If either the characteristic of the base field $k$ is not equal to $2$ or the group $G$ has no direct factor $\mathrm{SO}_{2n+1}$ for any $n\ge1$, then any normal unipotent subgroup of $G_k$ is automatically trivial by \cite{Vasiu}. In such cases, we do not need \cite[Proposition 3.6]{CJ1} in the proof above.  
\end{rqe}
\begin{prop}\label{prop:closure-in-matrix-algebra}
    If the base scheme $S$ is integral and normal, the monoid scheme $\M(\L)_{S}$ is the normalization of the scheme-theoretic image of (\ref{eq:embedding into highest weights representations}), in which $\lambda_1,\dotsc,\la_r$ are any set of nonzero generators of $\L$. 
\end{prop}
\begin{proof}
    We may assume that $S=\Spec(\k)$ is affine, where $\k$ is a normal integral domain. To simplify notations, in the following proof we will omit the subscripts $()_{\k}$. Let $\overline{G}\hookrightarrow\prod_{1\leq i\leq r}\underline{\mathrm{End}}(\Lambda_{\lambda_i})$ be the scheme-theoretic image of (\ref{eq:embedding into highest weights representations}), and let $\widetilde{G}\to \overline{G}$ be its normalization. Note that the $G\times G$-action on $G$ extends to uniquely to $\overline{G}$ and its normalization $\widetilde{G}$. Since (\ref{eq:embedding into highest weights representations}) is a quasicompact immersion of schemes, the natural map $G\hookrightarrow \overline{G}$ is a scheme-theoretically dominant open immersion. Since $G$ is normal, this map also factorizes through the normalization as a scheme-theoretic dominant open immersion ${G}\hookrightarrow \widetilde{G}$. Besides, according to Proposition \ref{prop:lifting hatU module to hatUM module}, the immersion (\ref{eq:embedding into highest weights representations}) factors as $G\hookrightarrow \M(\L)\to \prod_{1\leq i\leq r}\underline{\mathrm{End}}(\Lambda_{\lambda_i})$. Since the first map is scheme-theoretically dominant, the second one factors through $\overline{G}\hookrightarrow \prod_{1\leq i\leq r}\underline{\mathrm{End}}(\Lambda_{\lambda_i})$ and even through $\widetilde{G}\to \prod_{1\leq i\leq r}\underline{\mathrm{End}}(\Lambda_{\lambda_i})$ since $\M(\L)$ is normal by \cite[Lemma 4.14]{CJ1}. To summarize, we have a commutative diagram of schemes with $G\times G$-action in which all morphisms are $G\times G$-equivariant:
    \[\xymatrix{& \M(\L)\ar[d]\ar[rdd]\\
    & \widetilde{G} \ar[d]\ar[dr]\\
    G\ar[r]\ar[ur]\ar[uur]& \overline{G}\ar[r] & \prod\limits_{1\leq i\leq r}\underline{\mathrm{End}}(\Lambda_{\lambda_i}). 
    }\]
    For any $1\le i\le r$, the normalized matrix coefficient of the highest weight space of ${}_\k\La_{\la_i}$ (taking value $1$ at the identity) defines a function in $\k[G]^{U^-\times U}$ that must be equal to $1_{\la_i}^*$ by Proposition \ref{prop:finite-invariant-review}(3). In particular we have $1_{\la_i}^*\in\k[\widetilde{G}]^{U^-\times U}$ for all $1\le i\le r$ and since the set $\{\la_1,\dotsc,\la_r\}$ generates the monoid $\L$ we get that $\k[\L]\subset\k[\widetilde{G}]^{U^-\times U}$. On the other hand, since $\k[\widetilde{G}]\subset{}_\k\mathbf{O}(\L)$, we deduce from Proposition \ref{prop:finite-invariant-review}(3) that
    \[\k[\widetilde{G}]^{U^-\times U}\subset{}_\k\mathbf{O}(\L)^{U^-\times U}=\k[\L].\]
    Thus we conclude that $\k[\widetilde{G}]^{U^-\times U}={}_\k\mathbf{O}(\L)^{U^-\times U}$. Consequently the smallest $G\times G$-invariant subalgebra of ${}_\k\bfO(\L)$ containing ${}_\k\mathbf{O}(\L)^{U^-\times U}$ lies inside $\k[\widetilde{G}]$. Since ${}_\k\mathbf{O}(\L)$ is integral over the former by \cite[Theorem 6]{Kallen2014} (which generalizes \cite[Theorem 5]{Grosshans1992}), we deduce that ${}_\k\mathbf{O}(\L)$ is integral over $\k[\widetilde{G}]$. Since $\k[\widetilde{G}]$ is integrally closed (by normality of $\widetilde{G}$), we conclude that ${}_\k\bfO(\L)=\k[\widetilde{G}]$ and so the map $\M(\L)\to\widetilde{G}$ is an isomorphism.
\end{proof}

\subsection{Adjoint quotients}\label{subsec:adjoint-quotient}
We retain the notations from the previous subsection. Let $\L\subset X^+$ be a weight monoid and let $\M\coloneqq\M(\L)$ be the associated reductive monoid scheme over $S$ with unit group $G$ (cf. Definition \ref{theo:M(L)-def}). In this section we study the GIT quotient of $\M$ under the adjoint action of $G$:
\[\M//\Ad(G)\coloneqq\Spec(\O_S[\M]^{\Ad(G)}).\]
Steinberg \cite{Steinberg-regular}  established many foundational results on the GIT quotient $G//\Ad(G)$ for semisimple groups over a field. Later Lee \cite{LeeAdjointQuotient} extended many of Steinberg's results on $G//\Ad(G)$ to split reductive groups over a general base scheme. Our study is based on these works and generalizes some ingredients therein. Since the formation of GIT quotients commutes with any flat base change, we may assume that $S=\Spec(\k)$ is affine. \par 
Let us first recall some results in \cite{LeeAdjointQuotient} that can be easily adapted to our context.  Arguing as in \cite[Lemma 3.2, 3.3, 3.4]{LeeAdjointQuotient}, we have the following two lemmas:
\begin{lem}\label{lem:basis for invariants VT}
    For any weight $\mu\in X$ corresponding to an element $e^\mu\in\ZZ[X]$, we denote:
    $$\mathrm{Sym}(e^{\mu})\coloneqq\sum_{w\in W/W_{\mu}}e^{w\cdot \mu}\in \ZZ[X]^W$$ 
    where $W_{\mu}$ is the stabilizer of $\mu$ under the action of the Weyl group $W$.
    \begin{enumerate}
        \item The set $\Set{\mathrm{Sym} (e^{\lambda})| \lambda \in \L}$ forms a $\ZZ$-basis of the ring of $W$-invariants $\ZZ[W\cdot \L]^W$. In particular, the formation of the GIT quotient $\M_T//W$ commutes with arbitrary base change. 
        \item For any family of integers $(a_{\mu}^{\lambda})\in \ZZ^{X\times X}$, the set 
        $$\Set{\mathrm{Sym}(e^{\lambda})+\sum_{\mu<\lambda,\mu\in X^+ }a_{\mu}^{\lambda}\mathrm{Sym}(e^{\mu})| \lambda \in \L}$$ 
        forms a $\ZZ$-basis of $\ZZ[W\cdot \L]^W$.
    \end{enumerate}
\end{lem}

Recall that for any $\lambda \in X^+$, we have the Weyl module $\La_{\la,S}$ and the associated representation $\rho_{\la,S}\colon G\to\underline{\mathrm{GL}}(\La_{\la,S})$. The \emph{character for $\La_{\la,S}$} is the element $\chi_\la\in\O_S[G]^{\Ad(G)}$ defined as the pull-back of the trace function along $\rho_{\la,S}$. We can deduce as in \cite[Corollary 3.6]{LeeAdjointQuotient} the following result:
\begin{lem}\label{lem:explicit form restriction chi}
    Assume that the base scheme $S=\Spec(\k)$ is affine. For any $\lambda\in\L$, the restriction of the character $\chi_{\lambda}$ along the natural embedding $T\hookrightarrow G$, seen as an element of $\k[T]\cong\k[X]$ (the group ring of $X$), has the form 
    \[\mathrm{Sym}(e^{\lambda})+\sum_{\mu<\lambda,\mu\in\L}a_\mu\cdot\mathrm{Sym}(e^{\mu}),\quad a_\mu\in\k,\ \forall\mu.\] 
    The same holds for any product $\chi_{\lambda_1}\chi_{\lambda_2}...\chi_{\lambda_m}$ in which $\lambda_1,\dotsc,\lambda_m\in\L$ satisfy $\lambda=\lambda_1+\dotsm+\lambda_m$. 
\end{lem}

\begin{theo}[\cite{LeeAdjointQuotient}]
    Let $G$ be a split reductive group scheme over a general base scheme $S$. Let $T\subset G$ be a split maximal torus and let $W=N_G(T)/T$ be the Weyl group. Then the natural embedding $T\into G$ induces an isomorphism $T//W\cong G//\Ad(G)$.
\end{theo}
\begin{proof}
    First we assume $S=\Spec(k)$ where $k$ is an algebraically closed field. When $G$ is semisimple, this is \cite[Corollary 6.4]{Steinberg-regular} (see also \cite[Theorem 2.1]{LeeAdjointQuotient}). In general we have an isomorphism $G\cong(Z\times G_{\der})/Z_{\der}$ that induces isomorphisms 
    \[G//\Ad(G)\cong (Z\times G_{\der}//\Ad(G_{\der}))//Z_{\der},\quad T//W\cong (Z\times T_{\der}//W)//Z_{\der}.\]
    Then the isomorphism $T_{\der}//W\cong G_{\der}//\Ad(G_{\der})$ induces the isomorphism $T//W\cong G//\Ad(G)$. \par 
    Then the argument in the proof of \cite[Lemma 4.2]{LeeAdjointQuotient} shows that the formation of the GIT quotients commutes with arbitrary base change and hence it remains to prove the theorem in the case where $S=\Spec(\ZZ)$. This case is \cite[Theorem 3.1]{LeeAdjointQuotient}.
\end{proof}
We are now ready to prove the following:
\begin{theo}\label{theo:Chevalley iso} 
    Let $T\subset G$ be a split maximal torus and let $\M_T$ be the schematic closure of $T$ in $\M$. Then the natural embedding $\M_T\hookrightarrow\M$ induces an isomorphism 
    $$\M_T//W\cong\M//\Ad(G)$$
    and the formation of these GIT quotients commutes with arbitrary base change. In the case where the monoid $\L$ is freely generated by $\lambda_1,\dots,\lambda_r\in X^+$, the characters $\chi_{\la_i}$ induce an isomorphism \[\M//\Ad(G)\cong\AA^r_S.\]
\end{theo}
\begin{proof}
    We may assume that the base scheme $S=\Spec(\k)$ is affine. Recall that $\k[\M]={}_\k\bfO(\L)$ is a direct summand of $\k[G]$ (as $\k$-modules) by construction. Then we have $\k[\M]^{\Ad(G)}=\k[G]^{\Ad(G)}\cap\k[\M]$ and the Steinberg isomorphism $\k[G]^{\Ad(G)}\xrightarrow{\sim}\k[T]^W$ induces an injective homomorphism
    \[\iota\colon\k[\M]^{\Ad(G)}\into\k[\M_T]^W=\k[W\cdot\L]^W.\]
    For any $\la\in\L$, the character $\chi_\la$ extends to a regular function on $\M$ by Proposition \ref{prop:lifting hatU module to hatUM module}. In other words, we have $\chi_\la\in\k[\M]^{\Ad(G)}$ for all $\la\in\L$. Then it follows from Lemma \ref{lem:basis for invariants VT} and Lemma \ref{lem:explicit form restriction chi} that $\iota$ is surjective and hence an isomorphism.\par
    By Lemma \ref{lem:basis for invariants VT}, the formation of $\M_T//W$ commutes with arbitrary base change. Then from the isomorphism $\iota$ above we deduce that the formation of the adjoint quotient $\M//\Ad(G)$ also commutes with arbitrary base change. \par
    Now assume that $\L$ is freely generated by the elements $\lambda_1,\dots,\lambda_r\in X^+$. We want to prove that the class functions $\chi_{i}\coloneqq\chi_{\lambda_i}\colon\M//\Ad(G)\to \AA^1_S$ for $1\leq i\leq r$ give rise to an isomorphism $\M//\Ad(G)\simeq \AA^{r}_S$. It suffices to prove that the restriction through $\M_T//W\xrightarrow{\sim} \M//\Ad(G)$ of all the monomials that can be produced with this family is a $\k$-basis for the ring of regular functions $\k[\M_T//W]\simeq\k[W\cdot \L]^W$. But according to Lemma \ref{lem:explicit form restriction chi}, for any nonnegative integers $n_1,\dots,n_r$, the class function $\prod_{1\leq i\leq r}\chi_{i}^{n_i}$ restricts to an element of $\mathrm{Sym}(e^{\lambda})+\bigoplus_{\mu<\lambda,\mu\in X^+}\k\cdot\mathrm{Sym}(e^{\mu})$ where $\lambda=\sum_{1\leq i\leq r}n_i\lambda_i$. Since the set $\{\lambda_1,\dotsc,\lambda_r\}$ is a basis of the free monoid $\L$, these elements form a $\k$-basis of $\k[\M_T]^W=\k[W\cdot\L]^W$ by Lemma \ref{lem:basis for invariants VT}.
\end{proof}

\section{Reductive monoid schemes: classification}\label{sec:classification}
In this section we state and prove our main theorem on the classification of reductive monoids. We fix a based root datum $\Psi$ and keep the notations from \S\ref{sec:notations}.

\subsection{The weight monoid of a rigidified reductive monoid}
Throughout this subsection we let $G$ be a pinned split reductive group over a field $k$ with based root datum $\Psi$. Our goal is to associate a weight monoid $\L(\M)$ (in the sense of Definition \ref{def:weight-monoid}) to any rigidified reductive $G$-monoid $(\M,\iota)$.
\begin{defi}
    Let $V$ be a $G\times G$-module. Restricting the $G\times G$-action to $T\times T$, we get an $X\oplus X$-grading on $V$. For any weights $\la,\mu\in X$, let $V_{(\la,\mu)}$ be the corresponding homogeneous component and let
    \[V^{(\la,\mu)}\coloneqq V_{(\la,\mu)}\cap V^{U^-\times U}.\]
    In other words, $V^{(\la,\mu)}$ is the $B^-\times B$-eigenspace of $V$ of weight $(\la,\mu)$, or equivalently, the $(\la,\mu)$-component of the $X\times X$-graded $\k$-module $V^{U^-\times U}$. We set
    \[\L(V)\coloneqq\Set{\mu \in X^+| V^{(-\mu,\mu)}\neq 0}=\Set{\mu \in X^+| V_{(-\mu,\mu)}\cap V^{U^-\times U}\neq0}.\] 
    In particular, for a rigidified reductive $G$-monoid $(\M,\iota)$, the coordinate ring $k[\M]$ is a $G\times G$-submodule of $k[G]$ and we set
    \[\L(\M)\coloneqq\L(k[\M]).\]
\end{defi}

\begin{lem}\label{lem:B-eigenspace}
    Let $W\subset k[G]$ be a $G\times G$-submodule. Then we have 
    \[\L(W)=\{\la\in X^+\mid 1_\la^*\in W\}.\]
    In particular, for any rigidified reductive $G$-monoid $(\M,\iota)$ we have 
     \[\L(\M)=\{\la\in X^+\mid 1_\la^*\in k[\M]\}.\]
\end{lem}
\begin{proof}
    By Proposition \ref{prop:finite-invariant-review}(3), for any $\la\in X^+$ the space $k[G]^{(-\la,\la)}$ is one-dimensional and spanned by the function $1_\la^*$. Then we have $W^{(-\la,\la)}\ne0$ if and only if $W^{(-\la,\la)}=k[G]^{(-\la,\la)}$, if and only if $1_\la^*\in W$.
\end{proof}

Our goal is to show that for any rigidified reductive $G$-monoid, the set $\L(\M)$ is a weight monoid. For this we first review the construction of invariant valuations on the function field $k(G)$ from one-parameter subgroups, following \cite{Luna-Vust} and \cite{RittatoreMonoidArticle}.\par
Let $\la\colon\GG_m\to G$ be a one-parameter subgroup. Let $L(\la)$ be the centralizer of $\la$ and let 
\[U(\la)\coloneqq\{g\in G\mid\lim_{t\to0}\la(t)g\la(t)^{-1}=1\}.\]
Then $P(\la)\coloneqq U(\la)L(\la)$ is a parabolic subgroup of $G$. Let
\[U(\la)^-\coloneqq\{g\in G\mid\lim_{t\to0}\la(t)^{-1}g\la(t)=1\}\]
and let $P(\la)^-\coloneqq L(\la)U(\la)^-$ be the opposite parabolic subgroup. Consider the morphism 
\[a_\la\colon G\times G\times\GG_m\to G\]
defined by $a_\la(g_1,g_2,t)=g_1^{-1}\la(t)g_2$. It induces an embedding between the fields of rational functions 
\[a_\la^*\colon k(G)\into k(G\times G\times\GG_m)\subset k(G\times G)((t)).\] 
Let $v_\la$ be the discrete valuation\footnote{We use the convention that the trivial valuation (sending any nonzero element to $0$) is also regarded as a discrete valuation. When $\la=0$ is the trivial one-parameter subgroup, then $v_0$ is the trivial valuation on $k(G)$.} on $k(G)$ defined as the composition of $a_\la^*$ with the $t$-adic valuation on $k(G\times G)((t))$. Let $\O_\la\coloneqq\{f\in k(G)\mid v_\la(f)\ge0\}$ be the valuation ring of $v_\la$ and let $\mathfrak{m}_\la\subset\O_\la$ be its maximal ideal. Since the $t$-adic valuation is invariant under left translation by $G\times G$, the valuation $v_\la$ on $k(G)$ is $G\times G$-invariant. In particular, the valuation $v_\la$ only depends on the $G$-conjugacy class of $\la$.\par  
For any rational function $f\in k(G)$ we can write
\[a_\la^*(f)=\sum_{n\in\ZZ} f_n(g_1,g_2)t^n\in k(G\times G)((t))\]
where $f_n(g_1,g_2)\in k(G\times G)$ and it vanishes for all but finitely many $n<0$. If $f\ne0$, then the valuation $v_\la(f)$ equals the least $n\in\ZZ$ such that $f_n(g_1,g_2)\ne0$.\par
We view $a_\la^*(f)$ as a rational function on $G\times G\times\AA^1$. Let $\Omega_f$ be the maximal open subset of $G\times G\times\AA^1$ on which $a_\la^*(f)$ is defined and let $\Omega_0\coloneqq\Omega_f\cap (G\times G\times\{0\})$. Then $\Omega_0$ is nonempty if and only if $f\in\O_\la$, and in this case $f_0\in k[\Omega_0]\subset k(G\times G)$. More precisely, a pair $(g_1,g_2)\in G\times G$ lies in $\Omega_0$ if and only if $(g_1,g_2,t)\in\Omega_f$ for all $t$ in an open neighborhood of $0$. Then we can write
\[f_0(g_1,g_2)=\lim_{t\to0}f(g_1^{-1}\la(t)g_2).\]
This means that for fixed $(g_1,g_2)\in\Omega_0$, the expression $f(g_1^{-1}\la(t)g_2)$, viewed as a rational function in the variable $t$, is regular in a neighborhood of $0$ and its value at $0$ equals $f_0(g_1,g_2)$.\par
Define the subgroup 
\[G(\la)\coloneqq\Delta(L(\la))(U(\la)^-\times U(\la))\subset G\times G\]
where $\Delta(L(\la))$ is the image of $L(\la)$ under the diagonal embedding $\Delta\colon G\to G\times G$. 

\begin{lem}
    Suppose $f\in\O_\la$. Then $\Omega_0$ is stable under left multiplication by $G(\la)$ and $f_0$ is left $G(\la)$-invariant. 
\end{lem}
\begin{proof}
    For any $g_1,g_2\in G$ and $u\in U(\la)$, we have
    \[g_1^{-1}\la(t)ug_2=(\la(t)u^{-1}\la(t)^{-1}g_1)^{-1}\la(t)g_2,\quad\forall t\in k^\times.\]
    Since $u^{-1}\in U(\la)$, the morphism $f_{\la,u}\colon\GG_m\to U(\la)$ defined by the formula $f_{\la,u}(t)\coloneqq \la(t)u^{-1}\la(t)^{-1}$ extends to a morphism $\AA^1\to U(\la)$ with $f_{\la,u}(0)=1$. 
    Define two automorphisms $\varphi_u,\psi_u$ of the scheme $G\times G\times\AA^1$ by
    \[\varphi_u(g_1,g_2,t)=(g_1,ug_2,t),\quad\psi_u(g_1,g_2,t)=(f_{\la,u}(t)g_1,g_2,t),\quad\forall g_1,g_2\in G.\]
    Then the identity above implies that $\varphi_u^*a_\la^*(f)=\psi_u^*a_\la^*(f)$ and hence $\varphi_u^{-1}(\Omega_f)=\psi_u^{-1}(\Omega_f)$. Evaluating at $t=0$ and noting that $\psi_u$ restricts to the identity on the closed subscheme $G\times G\times\{0\}\subset G\times G\times\AA^1$, we see that $\Omega_0$ is stable under left multiplication by $(1,u)$ and that $f_0(g_1,g_2)=f_0(g_1,ug_2)$ for all $(g_1,g_2)\in\Omega_0$.\par
    By a similar argument, using instead the following identity:
    \[g_1^{-1}u^{-1}\la(t)g_2=g_1^{-1}\la(t)\cdot\la(t)^{-1}u^{-1}\la(t)g_2,\quad\forall t\in k^\times, u\in U(\la)^-,\]
    we deduce that for any $u^-\in U(\la)^-$, the set $\Omega_0$ is stable under left multiplication by $(u^-,1)$ and $f_0(g_1,g_2)=f_0(u^-g_1,g_2)$ for all $(g_1,g_2)\in\Omega_0$.\par 
    Finally it is clear that $\Omega_0$ is stable under left multiplication by $\Delta(L(\la))$ and $f_0$ is left invariant under  $\Delta(L(\la))$.
\end{proof}

Now we study valuations of elements in the multiplicative group of nonzero rational semi-invariants $k(G)^{(B^-\times B)}=k[B^-B]^{(B^-\times B)}$ consisting of eigen-functions of $B^-\times B$. For any $\chi\in X$, let $\zeta_\chi$ be the corresponding eigenfunction in $k(G)$ of weight $(-\chi,\chi)$, normalized such that $\zeta_\chi(1)=1$. \par 
Let $\la^\dagger\in X_*(T)$ be the \emph{anti-dominant} representative in the $G$-conjugacy class of $\la$. 
\begin{lem}\label{lem:val-coweight}
    There exists a positive rational number $r\in\QQ_{>0}$ such that $v_\la(\zeta_\chi)=r\langle\la^\dagger,\chi\rangle$ for all $\chi\in X$. 
\end{lem}
\begin{proof}
    If $\la=0$ then $v_\la$ is the trivial valuation and we can take $r=1$. Now assume that $\la\ne0$. After $G$-conjugation (which does not change the valuation $v_\la$), we may and do assume that $\la=\la^\dagger$ is \emph{anti-dominant}. Then we claim that $G(\la)(B^-\times B)$ is an open subset of $G\times G$. It is equivalent to show that the $G(\la)$ orbit of the base point on the flag variety $G/B^-\times G/B$ (for $G\times G$) is open. Since $\la$ is anti-dominant, the $P(\la)^-\times P(\la)$-orbit of the base point in $G/B^-\times G/B$:
    \[U(\la)^-L(\la)B^-/B^-\times U(\la)L(\la)B/B\]
    is open and it is a $U(\la)^-\times U(\la)$-torsor over the flag variety
    \[L(\la)/(L(\la)\cap B^-)\times L(\la)/(L(\la)\cap B)\]
    for $L(\la)\times L(\la)$. Then the $G(\la)$-orbit is the inverse image of the diagonal $L(\la)$-orbit and hence open.\par 
    Now we take $\chi\in X$ such that $f\coloneqq\zeta_\chi$ belongs to the valuation ring $\O_\la$. Recall the open subset $\Omega_0\subset G\times G$ consisting of $(g_1,g_2)\in G\times G$ such that $a_\la^*(f)$ is regular at $(g_1,g_2,0)$. We have shown that it is stable under left multiplication by $G(\la)$. Since $f$ is an eigenfunction for $B^-\times B$, the set $\Omega_0$ is also stable under right multiplication by $B^-\times B$.
    If $(1,1)\notin \Omega_0$, then the open set $G(\la)\cdot(B^-\times B)$ is in the complement of the open set $\Omega_0$, which implies that $\Omega_0=\varnothing$, contradicting the assumption that $f\in\O_\la$. So we have $(1,1)\in\Omega_0$ and hence $\zeta_\chi(\la(t))=t^{\langle\la,\chi\rangle}\in k[[t]]$. Therefore $\langle\la,\chi\rangle\ge0$. If moreover $v_\la(f)>0$, then $f_0=0$ and $\lim\limits_{t\to0}\zeta_\chi(\la(t))=f_0(1,1)=0$ so we have $\langle\la,\chi\rangle>0$. On the other hand, if $v_\la(f)=0$, then $f_0\ne0$. Since $f_0$ is an eigenfunction under right $B^-\times B$-action and left $G(\la)$-invariant, we must have $\lim_{t\to0}\zeta_\chi(\la(t))=f_0(1,1)\ne0$ (otherwise $f_0$ would vanish on the open subset $G(\la)\cdot(B^-\times B)\subset G\times G$), so $\langle\la,\chi\rangle=0$. \par 
    Conversely if $\zeta_\chi\notin\O_\la$, then $\zeta_\chi^{-1}=\zeta_{-\chi}\in\mathfrak{m}_\la$ and by what we just proved we get $\langle\la,\chi\rangle<0$. Thus we conclude that $v_\la(\zeta_\chi)\ge0$ if and only if $\langle\la,\chi\rangle\ge0$. Since the functions $\chi\mapsto v_\la(\zeta_\chi)$ and $\chi\mapsto \langle\la,\chi\rangle$ are both $\ZZ$-valued and $\ZZ$-linear, they differ by a positive rational scalar. 
\end{proof}
The following result is a consequence of \cite[Proposition 9]{RittatoreMonoidArticle}, which identifies the valuation cone of $G$ with the anti-dominant Weyl chamber in $X_*(T)_\QQ$. We give an alternative and straightforward proof.
\begin{lem}\label{lem:ideal-downward-closed}
    Suppose that the field $k$ is algebraically closed. Let $(\M,\iota)$ be a rigidified reductive $G$-monoid over $k$ and let $\p\subset k[\M]$ be a $G\times G$-stable prime ideal. There exists a $G\times G$-invariant discrete valuation $v$ on the rational function field $k(G)$ satisfying the following conditions:
    \begin{enumerate}
        \item[(i)] $\p=\{f\in k[\M]\mid v(f)>0\}$;
        \item[(ii)] for any $\la_1,\la_2\in X^+$ with $\la_1\ge\la_2$ we have 
        $v(1_{\la_1}^*)\le v(1_{\la_2}^*)$.
    \end{enumerate}
    In particular, combining (ii) with Lemma \ref{lem:B-eigenspace} we get that $\L(\p)\coloneqq\{\la\in X^+\mid\p^{(-\la,\la)}\ne0\}$ is downward closed in $\L(\M)$. 
\end{lem}
\begin{proof}
    If $\p=\{0\}$ we take $v$ to be the trivial valuation, then the conditions are obviously satisfied and $\L(\{0\})=\varnothing$. From now on assume that $\p$ is nonzero. Let $Y=\Spec(k[\M]/\p)$ be the closed $G\times G$-stable subscheme of $\M$ defined by $\p$. Since there are only finitely many $G\times G$-orbits on $\M$ by \cite[Corollary 2.2]{Knop-LunaVust}, $Y$ contains an open $G\times G$-orbit. First we claim that there is a coweight $\mu\in X_*(T)$ that extends to a morphism $\AA^1\to\M$ (which we still denote by $\mu$) such that $\mu(0)$ lies in the open $G\times G$-orbit of $Y$. Choose a point $y$ in the open $G\times G$-orbit in $Y$. By Noether normalization there is a finite morphism $\M\to\AA^d$ sending $y$ to the origin, where $d=\dim(\M)$. We can find a morphism $\eta\colon\Spec(k[[s]])\to\AA^d$ such that the generic point lies in the complement of the image of the boundary $\M\setminus G$ and the special point is sent to the origin. Let $C_\eta$ be the normalization of the base change of $\M\to\AA^d$ along $\eta$. Then all the generic points of $C_\eta$ maps to the dense open subscheme $G$. After choosing a closed point $y'\in C_\eta$ that maps to $y$ and taking the completion of the local ring of $C_\eta$ at $y'$, we get a morphism $\Spec(k[[t]])\to\M$ whose generic point lands in the dense open $G$ and whose special point maps to $y$. In other words, we have found an element $g\in\M(k[[t]])\cap G(k((t)))$ such that $g(0)=y$. By the Cartan decomposition, there exists $\mu\in X_*(T)$ such that $g\in G(k[[t]])\mu(t)G(k[[t]])$. Then $\mu$ extends to a morphism $\AA^1\to\M$ such that $\mu(0)$ lies in the $G\times G$-orbit of $y$, i.e. the open $G\times G$-orbit in $Y$. This proves the claim.\par
    Then for any $f\in k[\M]$ we have 
    \[f(g_1^{-1}\mu(t)g_2)=\sum_{n\in\NN}f_n(g_1,g_2)t^n\in k[G\times G][[t]]\]
    and in particular $v_\mu(f)\ge0$. 
    Letting $t$ tend to $0$, we obtain
    \[f(g_1^{-1}\mu(0)g_2)=f_0(g_1,g_2)\in k[G\times G].\]
    Note that $f\in\p$ if and only if $f$ vanishes on the $G\times G$-orbit of $\mu(0)$. This is equivalent to the condition $f_0=0$, i.e. $v_\mu(f)>0$. Thus (i) holds with $v=v_\mu$ and then (ii) follows from Lemma \ref{lem:val-coweight}. 
\end{proof}
The following Proposition is proved in \cite{BaoSongAffineEmbeddings} for the more general setting of embeddings of symmetric spaces. We give an independent proof in our setting.
\begin{prop}\label{prop:LM is closed saturated}
    Let $(\M,\iota)$ be a rigidified reductive $G$-monoid over $k$. Then $\L(\M)$ is a weight monoid in $X^+$ (see Definition \ref{def:weight-monoid}).
\end{prop}
\begin{proof}
    Let $k(G)$ be the common fraction field of $k[\M]$ and $k[G]$. For $\mu',\mu''\in X^+$, the fraction $1_{\mu'}^*/1_{\mu''}^*\in k(G)$ depends only on the difference $\lambda=\mu'-\mu''$ by Proposition \ref{prop:finite-invariant-review}(2) and we denote it by $\zeta_{\lambda}$. In particular, if $\la\in X^+$ we have $\zeta_\la=1_\la^*$. \par
    By Lemma \ref{lem:B-eigenspace} we have $k[\M]^{U^-\times U}\cong k[\L(\M)]$. Combined with \cite[Theorem 9]{Grosshans1992} we get that $k[\L(\M)]$ is a finitely generated $k$-algebra and hence $\L(\M)$ is a finitely generated monoid.\par
    Next we show that $\L(\M)$ is saturated. Let $\mu\in X^+$ and $n\in\ZZ_{\ge1}$ be such that $n\mu\in\L(\M)$. Then we have $1_{n\mu}^*=(1_\mu^*)^n\in k[\M]$ by Lemma \ref{lem:B-eigenspace} and Proposition \ref{prop:finite-invariant-review}(2). Since $k[\M]$ is integrally closed by assumption, we get that $1_\mu^*\in k[\M]$ and hence $\mu\in\L(\M)$.\par   
    Now let us show that $\L(\M)$ is downward closed. If $\M=G$, then $\L(\M)=\L(G)=X^+$ and this is obvious. So we may assume that $\M\ne G$. Let $\bar{k}$ be an algebraic closure of $k$. Then by Lemma \ref{lem:B-eigenspace} we have $\L(\M)=\L(\M_{\bar{k}})$. So after replacing $k$ by $\bar{k}$ we may assume that $k$ is algebraically closed. Suppose $\lambda\in X^+$ and $\mu\in\L(\M)$ satisfy $\lambda\leq\mu$. According to \cite[II, \S2, Corollaire 3.6]{DemazureGabriel}, the unit group $G$ is a principal affine open subset of $\M$. Thus the complement $D\coloneqq\M\setminus G$ is of pure codimension $1$ by \cite[21.12.7]{EGAIV16a23}. Let $D_1,\dots, D_s$ be the irreducible components of $D$ equipped with their reduced subscheme structure and let $v_1,\dots , v_s$ be the associated $G\times G$-invariant valuations $k(G)\to\ZZ$. Since $\M$ is normal, for any $f\in k[G]$, one has $f\in k[\M]$ if and only if $v_i(f)\geq 0$ for any $1\leq i\leq s$. Then we conclude by Lemma \ref{lem:ideal-downward-closed} that $\L(\M)$ is downward closed.\par
    Finally, we show that $\L(\M)$ generates the abelian group $X$. It suffices to prove that for any $\mu \in X$, there exists $\mu',\mu''\in \L(\M)$ such that $\mu=\mu'-\mu''$. Consider the nonzero ideal 
    \[I_\mu\coloneqq\{f\in k[\M]\mid f\zeta_\mu\in k[\M]\}.\]
    It is $B^-\times B$-stable because $\zeta_\mu$ is an eigen-function of $B^-\times B$. By \cite[Expos\'e XVII, Proposition 3.2]{SGA3VIIIaXVIII} we have $I_\mu^{U^-\times U}\neq 0$ and hence there exists $\nu\in \L(\M)$ such that $\zeta_{\nu}\in I_\mu$. Then we have $\zeta_{\mu+\nu}\in k[\M]$ and so $\mu+\nu\in\L(\M)$. Thus $\mu=(\mu+\nu)-\nu$ is the desired decomposition.
\end{proof}

\subsection{The classification theorem}
\begin{theo}\label{theo:monoid-classification}
Let $G$ be a pinned split reductive group scheme over a nonempty scheme $S$. 
\begin{enumerate}
    \item For every rigidified reductive $G$-monoid over $S$, the map that sends any point $s\in S$ to $\L(\M_s)$ is locally constant under the Zariski topology of $S$. In particular when $S$ is connected, $\L(\M_s)$ is independent of the choice of the point $s\in S$ and we denote it simply by $\L(\M)$ and call it the \emph{weight monoid of $(\M,\iota)$}.
    \item Suppose that $S$ is connected. Then there is a natural isomorphism of posets:
    \begin{align*}
        &\Set{\text{Weight monoids in }X^+}\simeq\Set{\text{Isomorphism classes of rigidified reductive $G$-monoids over }S}
    \end{align*}
    given by the maps $\L\mapsto(\M(\L)_S,\iota_\L)$ and
    $(\M,\iota)\mapsto\L(\M)$. Here the partial order on weight monoids is defined by reverse inclusion and the partial order on the right hand side is defined by $(\M,\iota)\le(\M',\iota')$ if there exists a morphism of monoid schemes $\M\to\M'$ over $S$ that commutes with the isomorphisms $\iota$ and $\iota'$ on unit groups. 
\end{enumerate}
\end{theo}
\begin{proof}
    By Theorem \ref{theo:M(L)-def}, the assignment $\L\mapsto\M(\L)_S$ preserves the partial orders. By Proposition \ref{prop:finite-invariant-review}(3) we have
    $\L(\M(\L)_{\bar{s}})=\L$ for any geometric point $\bar{s}\in S$. In other words, (1) holds for the reductive monoids $\M(\L)_S$.\par 
    Since reductive monoids over $S$ are finitely presented, to prove the theorem we may and do assume that $S=\Spec(\k)$ is a connected noetherian affine scheme. Let $(\M,\iota)$ be a reductive $G$-monoid over $S$. To finish the proof, it remains to do the following: First prove (1) for $\M$ so that $\L(\M)$ is well-defined; then show that the rigidified monoids $\M$ and $\M(\L(\M))$ are isomorphic. We will do this in several steps.\par  
    \emph{Step 1.} Suppose that $\k$ is a field, so (1) is obvious. Let $\L\coloneqq\L(\M)$ and $\M'\coloneqq\M(\L)$. Then $\k[\M]$ and $\k[\M']$ are both normal integral domains contained in $\k[G]$ (the former by assumption, the latter by Theorem \ref{theo:M(L)-def}). Moreover we have $\k[\M]^{U^-\times U}=\k[\M']^{U^-\times U}$ by Proposition \ref{prop:finite-invariant-review}(3) and Lemma \ref{lem:B-eigenspace}. By \cite[Theorem 5]{Grosshans1992} we get that $\k[\M]=\k[\M']$ since they both coincide with the integral closure of the smallest $G_\k\times G_\k$-stable subalgebra of $\k[G]$ containing $\k[\M]^{U^-\times U}$.\par
    \emph{Step 2.} Suppose that $\k$ is a noetherian integral domain with fraction field $K$. Let $\L\coloneqq\L(\M_K)$ and $\M'\coloneqq\M(\L)_S$. Since $\k[G]/\k[\M]$ is a flat $\k$-module by Lemma \ref{lem:quotient-flat}, the natural homomorphism 
    \[\k[G]/\k[\M]\to K[G]/K[\M]\]
    is injective and hence $\k[\M]=\k[G]\cap K[\M]$ (intersection inside $K[G]$). Similarly we have $\k[\M']=\k[G]\cap K[\M']$. By Step 1 we have $K[\M]=K[\M']$ and so we conclude that $\k[\M]=\k[\M']$. Then $\M\cong\M'$ and so (1) holds for $\M$, $\L(\M)=\L$ and $\M(\L(\M))=\M'\cong\M$. \par 
    \emph{Step 3.} Suppose that $\k$ is a reduced noetherian ring. Let $K=\prod\limits_{i=1}^mK_i$ be the ring of total fractions of $\k$ where $K_i$'s are residue fields at the (finitely many) minimal prime ideals. By Step 2, the assignment $\bar{s}\mapsto\L(\M_{\bar{s}})$ is constant on each irreducible component of $S=\Spec(\k)$ and since $S$ is connected, it is constant on $S$. So (1) holds for $\M$. In particular, we have $\L(\M_{K_i})=\L(\M)$ for each $1\le i\le m$. Let $\M'\coloneqq\M(\L(\M))_S$. By Step 1 we have
    \[K[\M]=\prod_{i=1}^m K_i[\M]=\prod_{i=1}^m K_i[\M']=K[\M'].\]
    Since $\k$ is a subring of $K$, by the same argument in Step 2 we get that $\k[\M]=\k[\M']$ and the theorem holds in this case.\par 
    \emph{Step 4.} Finally we treat the case of general noetherian base ring $\k$ such that $S=\Spec(\k)$ is connected. Let $\mathfrak{N}$ be the nilradical of $\k$ and let $\k_{\mathrm{red}}\coloneqq\k/\mathfrak{N}$. Clearly, statement (1) for $S=\Spec(\k)$ follows from statement (1) for $\Spec(\k_{\mathrm{red}})$ which is proved in Step 3. So $\L(\M)$ is well-defined and we let $\M'\coloneqq\M(\L(\M))_S$. By Step 3 we have $\k_{\mathrm{red}}[\M]=\k_{\mathrm{red}}[\M']$ and it follows that
    \[\k[\M]\subset\k[\M']+\mathfrak{N}\cdot\k[G].\]
    Moreover, since $\k[G]/\k[\M]$ is $\k$-flat by Lemma \ref{lem:quotient-flat}, the conditions (i),(ii) and (iii) of \cite[Proposition 5.12]{CJ1}
    are all satisfied (with $Q=\L$, $P=X^+$, $N=\k[\M]$, $I=\mathfrak{N}$) and we deduce that $\k[\M]=\k[\M']$. This finishes the proof by the remark at the beginning. 
\end{proof}

\section{A study of orbit closures}\label{sec:orbit}
In this section we study $G\times G$-orbits in a rigidified reductive $G$-monoid $(\M,\iota)$ over a connected base scheme. It suffices to do this when the base is a connected affine scheme.  Over an algebraically closed field of characteristic $\ne2$, some of our results are also proved in \cite{BaoSongAffineEmbeddings} in the more general setting of affine embeddings of symmetric varieties.\par 
Throughout this section we let $\Psi$ be a based root datum and retain the notations from \S\ref{sec:notations} and let $G$ be a pinned split reductive group over a nonempty and connected scheme $S$ with based root datum $\Psi$.

\subsection{Classification of orbit closures}
Let $\M$ be a reductive $G$-monoid over $S$ with weight monoid $\L=\L(\M)$, which is a saturated downward closed submonoid of $X^+$ that generates the group $X$. We will use the notations from \S\ref{sec:RM-associated-to-WM}.
\begin{lem}\label{lem:M(L/J)}
Let $\k$ be a commutative ring and let $\J$ be a downward closed prime ideal of $\L$. Then ${}_\k\bfO(\J)$ is a finitely generated $G\times G$-stable ideal of ${}_\k\bfO(\L)$. Moreover, the affine scheme 
\[\M(\L/\J)\coloneqq\Spec({}_\k\bfO(\L)/{}_\k\bfO(\J))\]
is flat over $\Spec(\k)$ and has integral normal geometric fibers.
\end{lem}
\begin{proof}
    Since $\J$ is downward closed, the subset ${}_\k\bfO(\J)\subset{}_\k\bfO(\L)$ is a sub-coalgebra by \cite[Proposition 4.2]{CJ1} and hence also a $G_{\k}\times G_{\k}$-submodule. Since $\J$ is an ideal of $\L$, for any $\lambda\in \L$ and $\mu\in \J$ we have \[{}_\k\bfO_{\leq\lambda}\cdot{}_\k\bfO_{\leq \mu}\subset{}_\k\bfO_{\leq \lambda+\mu}\subset{}_\k\bfO(\J)\]
     by Proposition \ref{prop:finite-invariant-review}(2). Therefore ${}_\k\bfO(\J)$ is an ideal of ${}_\k\bfO(\L)$. The quotient ring ${}_\k\bfO(\L)/{}_\k\bfO(\J)$ is a free $\k$-module with basis $\{b^*,b\in\bigcup\limits_{\la\in\L\setminus\J}\dot{\bfB}[\la]\}$ and in particular flat over $\k$. To show that the ideal ${}_\k\bfO(\J)$ is finitely generated, we may assume that $\k=\ZZ$. Then ${}_\ZZ\bfO(\L)$ is a finitely generated $\ZZ$-algebra by Proposition \ref{prop:finite-invariant-review}(1) and hence noetherian, so ${}_\ZZ\bfO(\J)$ is a finitely generated ideal.\par 
     It remains to show that when $\k=k$ is an algebraically closed field, the $k$-algebra $k[\M(\L/\J)]={}_k\bfO(\L)/{}_k\bfO(\J)$ is a normal integral domain. The images of the dual canonical basis elements $\bigcup_{\la\in\L\setminus\J}\dot{\bfB}[\la]^*$ form a $k$-basis for $k[\M(\L/\J)]$. Using this basis and the partial order $\le$ on $X^+$ we obtain a filtration $\{k[\M(\L/\J)]_{\le\la},\la\in\L\setminus\J\}$ whose associated graded algebra $\gr(k[\M(\L/\J)])$ is isomorphic to the $\L\setminus\J$-graded algebra $\gr({}_k\bfO(\L\setminus\J))$ in \cite[Definition 4.4]{CJ1}. By \emph{loc. cit.} Theorem 4.16 we see that $\gr(k[\M(\L/\J)])$ is a normal integral domain. Then from \emph{loc. cit.} Lemma 4.15 and Lemma 4.7 we deduce that $k[\M(\L/\J)]$ is a normal integral domain.
\end{proof}
\begin{rqe}
    In the previous setting, the closed $G\times G$-stable subscheme $\M(\L/\J)\hookrightarrow \M(\L)$ is a sub-semigroup of $\M(\L)$.
\end{rqe}
\begin{lem}\label{lem:GxG invariant prime ideal}
    Let $\k$ be a commutative ring such that $\Spec(\k)$ is connected. Let $\mathfrak{p}\subset\k[\M]={}_\k\bfO(\L)$ be a $G_\k\times G_\k$-stable finitely generated ideal such that the structure morphism $\Spec({}_\k\bfO(\L)/\mathfrak{p})\to\Spec(\k)$ is flat with integral geometric fibers. 
    \begin{enumerate}
        \item For any point $s\in\Spec(\k)$ with residue field $\kappa(s)$, the set $\L(\p_s)=\{\la\in X^+\mid\p_s^{(-\la,\la)}\ne0\}$ is a saturated downward closed prime ideal of $\L$.
        \item The map $s\mapsto\L(\p_s)$ is constant on $S$. Let $\J$ be its value. Then we have $\p={}_\k\bfO(\J)$.
    \end{enumerate}
\end{lem}
\begin{proof}
    (1) We may assume that $\k=k$ is an algebraically closed field (an algebraic closure of $\kappa(s)$). Then $\p$ is a prime ideal of $k[\M]$ by assumption and we deduce that $\L(\p)$ is a prime ideal of $\L$ by Lemma \ref{lem:B-eigenspace}. Suppose $n\la\in\L(\mathfrak{p})$ for some $\la\in X^+$ and $n\in\ZZ_{\ge1}$, then $1_{n\la}^*=(1_\la^*)^n\in\p$ and hence $1_\la^*\in\mathfrak{p}$ since $\p$ is a prime ideal. Thus $\la\in\L(\mathfrak{p})$ and we see that $\L(\p)$ is saturated. Finally $\L(\p)$ is downward closed thanks to Lemma \ref{lem:ideal-downward-closed}.\par
    (2) The strategy is similar to the proof of Theorem \ref{theo:monoid-classification} and proceeds in several steps.\par 
    \emph{Step 1.} Suppose that $\k=k$ is a field. Let $\J\coloneqq\L(\p)$. Then we have $\p^{U^-\times U}={}_k\bfO(\J)^{U^-\times U}$. By Lemma \ref{lem:M(L/J)}, ${}_k\bfO(\J)$ is a prime ideal of $k[\M]={}_k\bfO(\L)$. Then the module $Q\coloneqq \p/(\p\cap{}_k\bfO(\J))$ embeds in the integral domain $k[\M]/{}_k\bfO(\J)$. If $Q\ne0$, then there exists a nonzero element $0\ne x\in Q^{U^-\times U}$ by \cite[Expos\'e XVII, Proposition 3.2]{SGA3VIIIaXVIII}. Apply \cite[Lemma 5.9]{CJ1} to the augmented $k$-algebra $k\oplus\p$ and its ideal $\p\cap{}_k\bfO(\J)$, we see that there exists some power $x^N$ that lifts to $y\in\p^{U^-\times U}={}_k\bfO(\J)^{U^-\times U}$. But then since $y\in{}_k\bfO(\J)^{U^-\times U}$ its image in $Q$ is $0$ and so $x^N=0$, contradicting the assumption that $x\ne0$ and $k[\M]/{}_k\bfO(\J)$ is an integral domain. Thus we conclude that $\p\subset{}_k\bfO(\J)$. By a similar argument with the role of $\p$ and ${}_k\bfO(\J)$ switched we can prove the opposite inclusion and conclude that $\p={}_k\bfO(\J)$.\par 
    \emph{Step 2.} Suppose that $\k$ is a noetherian integral domain. Let $K$ be the fraction field of $\k$ and $\overline{K}$ be an algebraic closure of $K$. Let $\J\coloneqq\L(\p_{\overline{K}})$. By Step 1 we have $\p\otimes_\k\overline{K}={}_{\overline{K}}\bfO(\J)$.
    Since $\k[\M]/\p$ is a flat $\k$-module by assumption, the natural map $\k[\M]/\p\to\overline{K}[\M]/(\p\otimes_\k\overline{K})$ is injective and hence $\p=\k[\M]\cap(\p\otimes_\k\overline{K})$ (intersection inside $\overline{K}[\M]$). Similarly we have ${}_\k\bfO(\J)=\k[\M]\cap{}_{\overline{K}}\bfO(\J)$ and we conclude that $\p={}_\k\bfO(\J)$. In particular we deduce that $\L(\p_s)=\J$ for all $s\in S$. \par
    \emph{Step 3.} Next we assume that $\k$ is a noetherian and reduced ring such that $\Spec(\k)$ is connected. Let $K$ be the ring of total fractions of $\k$. Then $K=\prod\limits_{i=1}^m K_i$ is a finite product of fields and $\k$ is a subring of $K$. Let $\overline{K}_i$ be an algebraic closure of $K_i$ for each $i$ and let $\overline{K}\coloneqq\prod_{i=1}^m\overline{K}_i$. By Step 2, the map $s\mapsto\L(\p_s)$ is constant on each irreducible component of $S$ and hence constant on the connected scheme $S$. Let $\J$ be its value. In particular, we have $\L(\p\otimes_\k K_i)=\J$ for all $i$. Then by Step 1 we get that $\p\otimes_\k\overline{K}={}_{\overline{K}}\bfO(\J)$. By the same argument in Step 2 we deduce that $\p={}_\k\bfO(\J)$. \par
    \emph{Step 4.} Finally we treat the case of general base ring. Since $\M$ is finitely presented over $\k$ and the ideal $\p$ is finitely generated, we may and do assume that $\k$ is noetherian. Let $\mathfrak{N}\subset\k$ be the nil-radical and let $\k_{\mathrm{red}}=\k/\mathfrak{N}$ be the reduced quotient ring. The constancy of the map $s\mapsto\L(\p_s)$ on $S$ follows from the constancy of this map on $\Spec(\k_{\mathrm{red}})$ which is proved in Step 3. Let $\J$ be its constant value. By Step 3 we have 
    $\p\otimes_\k\k_{\mathrm{red}}={}_{\k_{\mathrm{red}}}\bfO(\J)$. Then we deduce that 
    \[\p\subset{}_\k\bfO(\J)+\mathfrak{N}\cdot\k[\M]\quad\text{ and }\quad{}_\k\bfO(\J)\subset\p+\mathfrak{N}\cdot\k[\M].\]
    Moreover $\k[\M]/\p$ is $\k$-flat by assumption, so conditions (i), (ii) and (iii) of \cite[Proposition 5.12]{CJ1} are satisfied (with $Q=\J$, $P=\L$, $N=\p$ and $I=\mathfrak{N}$) and we deduce that $\p={}_\k\bfO(\J)$.
\end{proof}

\begin{theo}\label{theo:parametrization GxG orbits}
    Let $(\M,\iota)$ be a rigidified reductive $G$-monoid over a nonempty connected scheme $S$. Then the map $\J\mapsto \M(\L/\J)$ defines an order-reversing bijection between the following posets:
    \begin{enumerate}
        \item[(i)] downward closed prime ideals of $\L(\M)$;
        \item[(ii)] $G\times G$-stable closed subschemes of $\M$ that are flat, of finite presentation, and with integral geometric fibers over $S$;
    \end{enumerate}
    and the inverse map sends a closed subscheme  $\mathcal{Y}\subset\M$ to $\L(\p_{\mathcal{Y}_s})$, where $s\colon\Spec(k)\to S$ is a geometric point and $\p_{\mathcal{Y}_s}\subset k[\M_s]$ is the ideal of $\mathcal{Y}_s\subset\M_s$. Moreover, the geometric fibers over $S$ of the closed subschemes in (ii) are normal. 
\end{theo}
\begin{proof}
    One easily reduces this to the case where $S$ is affine. Then the result follows by combining Lemma \ref{lem:M(L/J)} and Lemma \ref{lem:GxG invariant prime ideal}.
\end{proof}

\subsection{Idempotents}
In this subsection we study some idempotents in each $G\times G$-orbit of a rigidified reductive $G$-monoid $(\M,\iota)$. Let $\L=\L(\M)$ be its weight monoid.
\begin{defi}
    For any prime ideal $\I\subset W\cdot \L$, we define an idempotent $e_\I\in\M_T(S)\subset\M(S)$ corresponding to the homomorphism of $\O_S$-algebras 
    \[e_\I\colon \O_S[W\cdot\L]\to\O_S,\quad e^\la\mapsto\begin{cases}1&\text{ if }\la\in W\cdot\L\setminus\I,\\ 0&\text{ if }\la\in\I.\end{cases}\]
    Its image under the abelianisation map $\M\to A_{\M}$ is the idempotent $e_{\I}^{ab}\in A_{\M}(S)$ corresponding to the $\O_S$-algebra homomorphism
    \[e_{\I}^{ab}\colon \O_S[\L_{\ab}]\to\O_S,\quad e^\la\mapsto\begin{cases}
        1& \text{ if }\la\in\L_{\ab}\setminus\I,\\
        0& \text{ if }\la\in\I\cap\L_{\ab}.
    \end{cases}\]
\end{defi}

\begin{lem}\label{lem:restriction prime ideal of WL}
For any prime ideal $\I\subset W\cdot \L$, the following statements hold:
\begin{enumerate}
     \item The intersection $\I\cap\L$ is downward closed in $\L$.
     \item The intersection $(\bigcap_{w\in W}w\cdot\I)\cap\L$ is a downward closed prime ideal of $\L$.
     \item If $\I$ is downward closed in $W\cdot \L$, then $(\bigcap_{w\in W}w\cdot \I)\cap\L=\I\cap\L$.
\end{enumerate}
Moreover, for any two prime ideals $\I,\I'$ of $W\cdot\L$ such that 
\[(\bigcap_{w\in W}w\cdot(\I))\cap\L=(\bigcap_{w\in W}w\cdot(\I'))\cap\L,\] 
there exists $w\in W$ such that $\I=w(\I')$.
\end{lem}
\begin{proof} 
    We will freely use notations from \S\ref{sec:cones}. Let $F$ be the face of the cone $\tilde{K}\coloneqq W\cdot \QQ_{\ge0}\L$ such that $W\cdot\L\setminus F=\I$.\par
    (1) We need to show that $F\cap \L$ is upward closed in $\L$. Since $F$ is an intersection of some hyperplanes among the $W$-orbits of $H_1,\dotsc,H_s$ and the intersections of any upward closed subsets of $\L$ remain upward closed, we may assume that $F=w(H)\cap\tilde{K}$ where $H=H_i$ for some $1\le i\le s$ and $w\in W$.\par
    Let $x<y$ with $x\in w(H)\cap \L$ and $y\in\L$. We want to show that $ y\in w(H)$. Since $\L$ is downward closed in $X^+$, we may assume that $y$ is minimal among the dominant characters greater than (but not equal to) $x$. Then by \cite[Corollary 2.7]{Stembridge} we have $y=x+\a$ where $\a$ is a positive root. Let $m\coloneqq\langle\alpha^\vee, x\rangle$, a non-negative integer. Since $w^{-1}(x)\in H$ by assumption, we have
    \[(2+m)w^{-1}(x)=(1+m)w^{-1}(x+\a)+w^{-1}(x-(1+m)\a)\in H.\]
    Since $\L\subset\tilde{K}$, we deduce that the two summands $(1+m)w^{-1}(x+\a)$ and $w^{-1}s_{\a}(x+\a)=w^{-1}(x-(1+m)\a)$ both belong to $\tilde{K}$. Then by convexity of $\tilde{K}$, both summands belong to $H$ and hence $w^{-1}(y)=w^{-1}(x+\a)\in H$.\par
    (2) It is clear that the intersection is an ideal and it is downward closed by (1). To show that it is a prime ideal, it remains to verify that its complement $\bigcup_{w\in W}(w(F)\cap\L)$ is stable under addition. Take any $w,w'\in W$ and $x\in wF\cap \L$, $y\in w'F\cap \L$. Then we have $w^{-1}\cdot x+w'^{-1}\cdot y\in F\cap W\cdot\L$ and so there exists $v\in W$ such that $v\cdot (w^{-1}\cdot x+w'^{-1}\cdot y)\in vF\cap\L$. Since $v\cdot (w^{-1}\cdot x+w'^{-1}\cdot y)\leq x+y$ and $x+y\in \L$, we get $x+y\in vF$ by (1). \par
    (3) By assumption $F\cap W\cdot\L$ is upward closed in $W\cdot \L$. For any $w\in W$ and $x\in \L$ such that $w^{-1}\cdot x\in F$, we have $w^{-1}\cdot x\leq x$, so that $ x\in F$ by the upward closed property.
    This shows that $wF\cap\L\subset F$ for any $w\in W$ and the result follows.\par 
    It remains to prove the last statement. Let $F'$ be the face of $\tilde{K}$ such that $W\cdot\L\setminus F'=\I'$. By assumption we have $\bigcup_{w\in W}w(F)\cap\L=\bigcup_{w\in W}w(F')\cap\L$. To finish the proof it is enough to show that there exists $w\in W$ such that 
    \[F\cap W\cdot\L\subset w(F')\cap W\cdot\L.\]
    Indeed, this implies that $\I\supset w(\I')$ and by symmetry we get that $\I\supset w(\I')\supset w'(\I)$ for some $w'\in W$. Then we must have $\I=w'(\I)$ and hence $I=w(\I')$. \par 
    Suppose the inclusion above does not hold. Then for any $w\in W$, there exists an element $x_w\in F\cap W\cdot\L\setminus w(F')\subset\tilde{K}$. Then $x\coloneqq\sum_{w\in W}x_w$ belongs to $F\cap W\cdot\L$ and so there exists $v\in W$ such that $x\in F\cap v(\L)$. Then we have $v^{-1}(x)\in v^{-1}(F)\cap\L$ and our assumption implies that $v^{-1}(x)\in \bigcup_{w\in W}w(F')\cap\L$ and hence $x\in u(F')$ for some $u\in W$. Since $u(F')$ is a face of $\tilde{K}$, the convexity of $\tilde{K}$ implies that $x_u\in u(F')$, which is a contradiction and we are done.
\end{proof} 

\begin{lem}\label{lem:some properties WL and J}
    Let $\J$ be a downward closed ideal of $\L$. Then the downward closure \[\overline{\J}^{\leq}\coloneqq\{\la\in W\cdot\L\mid \la\le\mu\text{ for some }\mu\in\J\}\] 
    of $\J$ in $W\cdot\L$ is an ideal of $W\cdot \L$ and $\overline{\J}^{\leq}\cap\L=\J$.
\end{lem}
\begin{proof}
    Take any element $w\cdot x\in\overline{\J}^{\leq}$ with $w\in W$ and $x\in\L$. Then there exists $y\in\J$ such that $w\cdot x\leq y$. For any $w'\in W$ and $x'\in\L$ we have $w\cdot x+w'\cdot x'\leq y+w'\cdot x'\leq y+x'$. Since $\J$ is an ideal of $\L$ we have $y+x'\in\J$ and hence $w\cdot x+w'\cdot x'\in \overline{\J}^{\leq}$. This shows that $\overline{\J}^{\leq}$ is an ideal of $W\cdot\L$. \par 
    To show the last equality, we suppose moreover that $w\cdot x\in\L$. Since $w\cdot x\leq y$, $y\in\J$ and $\J$ is downward closed in $\L$, we get that $w\cdot x\in \J$. Therefore $\overline{\J}^{\leq}\cap\L=\J$.
\end{proof}

\begin{prop}\label{prop:bijection closed prime L WL}
    Any $W$-orbit of prime ideals in $W\cdot \L$ contains a downward closed prime ideal, and we have a bijection 
    $$\Set{\text{prime ideals of } W\cdot \L}/W \xrightarrow[]{\sim} \Set{\text{downward closed prime ideals of } \L}$$
    given by  $$\I\mapsto(\bigcap_{w\in W}w\cdot\I)\cap\L.$$
\end{prop}
\begin{proof}
    According to Lemma \ref{lem:restriction prime ideal of WL}, the map is well defined and injective. It remains to show surjectivity. Let $\J\subset\L$ be a downward closed prime ideal of $\L$. In particular $\J$ is a proper ideal by Definition \ref{def:monoid-ideal}. Let $\Omega_\J$ be the set of downward closed ideals $\I\subset W\cdot\L$ such that $\I\cap\L=\J$. We first note that $\Omega_\J$ is nonempty since it contains the downward closure $\overline{\J}^{\leq}$ by Lemma \ref{lem:some properties WL and J}. Since arbitrary unions of ideals in $\Omega_\J$ also lie in $\Omega_\J$, the set $\Omega_\J$ has a unique maximal element $\I$.\par
    We claim that $\I$ is a prime ideal of $W\cdot\L$. Suppose this is not the case. Then there exists $x,y\in W\cdot\L\setminus\I$ such that $x+y\in\I$. Consider the downward closures $\overline{\langle x\rangle}^{\leq},\overline{\langle y\rangle}^{\leq}\subset W\cdot\L$ of the principal ideals 
    \[\langle x\rangle\coloneqq x+W\cdot\L,\quad \langle y\rangle\coloneqq y+W\cdot\L.\]
    Let $\I_x\coloneqq\I\cup\overline{\langle x\rangle}^{\leq}$ and $\I_y\coloneqq\I\cup\overline{\langle y\rangle}^{\leq}$. Suppose there exists $\la_x\in\I_x\cap\L\setminus\J$ and $\la_y\in\I_y\cap\L\setminus\J$. Then there exists $a,b\in W\cdot\L$ such that 
    \[\la_x+\la_y\le x+y+a+b\in\I\]
    and since $\I$ is downward closed, we get that $\la_x+\la_y\in\I\cap\L=\J$. This contradicts the fact that $\J$ is a prime ideal. Therefore we have either $\I_x\cap\L=\J$ or $\I_y\cap\L=\J$, contradicting the maximality of $\I$. So we conclude that $\I$ is a prime ideal of $W\cdot\L$.
\end{proof}

\begin{cor}\label{cor:orbit closure}
  Let $\I$ be a prime ideal of $W\cdot \L$. Then the $G\times G$-orbit of the section $e_{\I}\in\M(S)$ is representable by a locally closed subscheme of $\M$ and its schematic closure equals $\M(\L/\J)$ where
  \[\J\coloneqq(\bigcap_{w\in W}w\cdot \I)\cap \L.\] 
  In particular, the formation of this schematic closure commutes with arbitrary base change.
\end{cor}
\begin{proof}
    For any downward closed prime ideal $\J\subset \L$, the formation of $\M(\L/\J)$ commutes with arbitrary base change by definition. So the last assertion follows from the remaining ones and we may assume that $S=\Spec(\k)$ is affine.\par
    For any $w\in W$, we choose a lift $\dot{w}\in N_G(T)(\k)\subset G(\k)$. Since the automorphism of $T$ induced by $w$ extends uniquely to $\M_T$, we have by definition $e_{w(\I)}=\dot{w}e_{\I}\dot{w}^{-1}$. In particular $e_{\I}$ and $e_{w(\I)}$ belong to the same $G\times G$-orbit. So by Proposition \ref{prop:bijection closed prime L WL} we may assume that $\I$ is downward closed in $W\cdot\L$. Then the intersection
    \[\J\coloneqq(\bigcap_{w\in W}w\cdot \I)\cap \L=\I\cap \L\]
    is a downward closed prime ideal of $\L$, where the equality follows from Lemma \ref{lem:restriction prime ideal of WL}. Since $\I$ is downward closed and $\J\subset X^+$, we have $W\cdot \J\subset \I$ and hence the ideal ${}_\k\bfO(\J)\subset\k[\M]={}_\k\bfO(\L)$ lies inside the kernel of the homomorphism $e_{\I}\colon {}_\k\bfO(\L)\to\k$ corresponding to $e_\I$. This shows that $e_{\I}\in\M(\L/\J)(\k)$ and the orbit map for $e_{\I}$ factors through a morphism 
    \[\mu_{\I}\colon G\times G\to \M(\L/\J),\] 
    whose image (as fppf sheaf on $S$) is identified with the orbit $\mathrm{Orb}_{\I}$ of $e_{\I}$ (defined as an fppf quotient of $G\times G$). 
    It remains to show that the natural morphism $\mathrm{Orb}_{\I}\to\M(\L/\J)$ is an open embedding and its schematic closure (or equivalently, the scheme-theoretic image of $\mu_\I$) equals $\M(\L/\J)$.\par
    Let us first assume that $S=\Spec(k)$ where $k$ is an algebraically closed field. Let $Z$ be the scheme-theoretic image of $\mu_{\I}$. According to Theorem \ref{theo:parametrization GxG orbits}, there exists a downward closed prime ideal $\J'\subset\L$ such that $Z=\M(\L/\J')$. Since $Z\subset\M(\L/\J)$ we have $\J\subset\J'$. For any $\la\in\J'$, from the fact that $e_\I\in Z$ we deduce that $1_{\lambda}^*$ lies in the kernel of $e_{\I}\colon{}_k\bfO(\L)\to k$ and so $\la\in\I\cap\L=\J$. Therefore we have $\J=\J'$ and $Z=\M(\L/\J)$. This shows that $\mu_\I$ is scheme-theoretically dominant. Combined with the fact that $\M(\L/\J)$ is reduced, we deduce that there is a nonempty open subset $U\subset G\times G$ such that the restriction of $\mu_\I$ to $U$ is flat by the generic flatness theorem \cite[Theorem 6.9.1]{EGAIV2a7}. Since $\mu_\I$ is $G\times G$-equivariant, by translation we see that $\mu_\I$ is flat.\par 
    In general, for any geometric point $s\in S$, since the formation of $\M(\L/\J)$ commutes with arbitrary base change, the base change $\mu_{\I,s}$ of the map $\mu_\I$ to $s$ is scheme-theoretically dominant and flat by what we just proved. Since $G$ and $\M(\L/\J)$ are of finite presentation and flat over $S$, the map $\mu_{\I}$ is flat by the fiberwise criteria for flatness. Then we deduce that the natural map $\mathrm{Orb}_{\I}\to\M(\L/\J)$ is an open embedding by \cite[Theorem 17.9.1]{EGAIV16a23}. On the other hand, according to \cite[11.10.9]{EGAIV8a15} (which is a consequence of \cite[Proposition 11.9.17]{EGAIV8a15}) the natural map $\mu_\I$ is scheme-theoretically dominant. Thus $\M(\L/\J)$ is the schematic closure of $\mathrm{Orb}_{\I}$ by \cite[11.10.3(iv)]{EGAIV8a15}.
\end{proof}

\begin{cor}\label{cor:orbit decomposition over algebraically closed field}
    Suppose $S=\Spec(k)$ where $k$ is an algebraically closed field. Then all the $G\times G$-orbit closures in $\M$ are normal $k$-varieties. Moreover, we have a decomposition:
    \[\M(k)=\bigsqcup_{[\I]\in \{\text{prime ideals of } W\cdot \L\}/W}G(k)e_{\I}G(k).\]
\end{cor}
\begin{proof}
By Theorem \ref{theo:parametrization GxG orbits}, the $G\times G$ orbit closures in $\M$ are the schemes $\M(\L/\J)$, where $\J$ is any downward closed prime ideal of $\L$. Lemma \ref{lem:M(L/J)} gives then the first assertion. By Proposition \ref{prop:bijection closed prime L WL}, we have $\J=(\bigcap_{w\in W}w\cdot \I)\cap \L$ for $\I$ in a unique $W$-orbit of prime ideals of $W\cdot \L$. From Corollary \ref{cor:orbit closure}, we deduce that the idempotents $e_{\I}$, for $\I$ ranging over a set of representatives of such $W$-orbits, form a representative set for the $G\times G$-orbits of $\M$. That implies the desired decomposition on $k$-points.
\end{proof}

\subsection{The minimal orbit}
Let $X_{\ab}\coloneqq X^*(G_{\ab})$, $X_{\der}\coloneqq X^*(T_{\der})$ and $X_{\der}^+$ be the image of $X^+$ in $X_{\der}$. 
Then we have 
\[X_{\ab}=\{\la\in X,\langle\alpha_i^\vee,\la\rangle=0,\forall i\in \Delta\}=X^+\cap(-X^+).\] 
In fact, we have a monoid decomposition $X^+_\QQ=X_{\der,\QQ}^+\oplus X_{\ab,\QQ}$.\par 
Let $\L^*\coloneqq\L\cap(-\L)$ be the subgroup of invertible elements in $\L$. Then $\L^*\subset X_{\ab}$ is a saturated sub-lattice (since $\L$ is saturated in $X^+$). In particular we have $\L^*=\L_{\ab}^*$ where $\L_{\ab}\coloneqq X_{\ab}\cap\L$. Choose a sub-lattice $X_{\ab}^{(2)}\subset X_{\ab}$ such that $X_{\ab}=\L^*\oplus X_{\ab}^{(2)}$.
\begin{lem}\label{lem:decomposition I1 I2}
    There is a disjoint union decomposition of the set of simple roots $\Delta=\Delta_1\sqcup\Delta_2$ such that $\QQ_{\ge0}\L\cap X_{\der,\QQ}=\QQ^{\Delta_1}\cap X_{\der,\QQ}^+$ and $\langle\alpha_i^\vee,\alpha_j\rangle=0$ for all $i\in\Delta_1$ and $j\in\Delta_2$.
\end{lem}
\begin{proof}
    Recall that $\QQ_{\ge0}\L=K\cap X_\QQ^+$ where $K\subset X_\QQ$ is a convex cone such that $-\alpha_i\in K$ for all $i\in\Delta$. Then we have 
    \[-X_{\der,\QQ}^+=-\left(\bigoplus_{i\in\Delta}\QQ_{\ge0}\omega_i\right)\subset-\left(\bigoplus_{i\in\Delta}\QQ_{\ge0}\alpha_i\right)\subset K\]
    where $\omega_i$ are the fundamental weights. On the other hand we have 
    \[\QQ_{\ge0}\L\cap X_{\der,\QQ}=K\cap X_{\der,\QQ}^+\subset X_{\der,\QQ}^+=\bigoplus_{i\in\Delta}\QQ_{\ge0}\omega_i.\]
    For any element $x=\sum_{i\in\Delta}c_i\omega_i\in\QQ_{\ge0}\L\cap X_{\der,\QQ}$, from the first inclusion above we deduce that for any $i\in\Delta$ such that $c_i>0$,
    \[\omega_i=c_i^{-1}(x-\sum_{j\in\Delta,j\ne i}c_j\omega_j)\in K\cap X_{\der,\QQ}^+=\QQ_{\ge0}\L\cap X_{\der,\QQ}.\]
    So if $\QQ_{\ge0}\L\cap X_{\der,\QQ}$ contains an interior point of some face of $X_{\der,\QQ}^+$, it contains the whole face. Therefore $\QQ_{\ge0}\L\cap X_{\der,\QQ}=K\cap X_{\der,\QQ}^+$ is a face of $X_{\der,\QQ}^+$.\par 
    Now suppose that $\omega_i\in K$ for some $i\in\Delta$ and write $\omega_i=\sum_{j\in\Delta}c_{ij}\alpha_j$ where the coefficients $c_{ij}$ are entries of an inverse Cartan matrix. So $c_{ij}>0$ if $j$ is in the same connected component of $i$ in the Dynkin diagram and equals $0$ otherwise. Then for any $j\in\Delta$ in the same connected component of $i$ in the Dynkin diagram, we have  
    \[\alpha_j=c_{ij}^{-1}(\omega_i-\sum_{k\in\Delta,k\ne j}c_{ik}\alpha_k)\in K.\]
    Thus we have $\QQ^{\Delta(i)}\subset K$ where $\Delta(i)\subset\Delta$ is the connected component of $i$, and hence $\omega_j\in K\cap X_{\der,\QQ}^+$ for all $j\in\Delta(i)$. As a conclusion, there is a union of connected components $\Delta_1\subset\Delta$ such that 
    \[\QQ_{\ge0}\L\cap X_{\der,\QQ}=K\cap X_{\der,\QQ}^+=\QQ^{\Delta_1}\cap X_{\der,\QQ}^+.\]
\end{proof}
Define the following sub-lattices of $X$:
\[X_1\coloneqq(\L^*_\QQ\oplus\QQ^{\Delta_1})\cap X,\quad X_2\coloneqq(X_{\ab,\QQ}^{(2)}\oplus\QQ^{\Delta_2})\cap X.\]
Then we have $X_{1,\QQ}\oplus X_{2,\QQ}=X_\QQ$. For $i=1,2$, we have the dominant weight monoids $X_i^+\coloneqq X_i\cap X^+$ and the associated convex polyhedral cones $X_{i,\QQ}^+\coloneqq X_{i,\QQ}\cap X^+_\QQ$. We also have a monoid decomposition $X_\QQ^+=X_{1,\QQ}^+\oplus X_{2,\QQ}^+$. We remark that $X_1$ is canonically associated to $\L$ while the definition of $X_2$ depends on the choice of the splitting $X_{\ab}=\L^*\oplus X_{\ab}^{(2)}$.
\begin{lem}\label{lem:basic properties X_1}
    With notations as above, we have the following equalities:
    \begin{enumerate}
        \item $X_{1}^+ \cap \L_{\ab}=\L_{\ab}\cap(-\L_{\ab})=\L^*$,
        \item $W\cdot X_1^+=X_1$,
        \item $\QQ_{\ge0}\L\cap X_{1,\QQ}=X_{1,\QQ}^+$ and $\L\cap X_1=X_1^+$
    \end{enumerate}
\end{lem}
\begin{proof}
    (1) is clear from the definition of $X_1$. We have a decomposition $W=W_1\times W_2$ corresponding to the decomposition $\Delta=\Delta_1\sqcup\Delta_2$ in Lemma \ref{lem:decomposition I1 I2}, in which $W_1$ (resp. $W_2$) is the subgroup generated by the simple reflections associated to the simple roots in $\Delta_1$ (resp. $\Delta_2$). In particular, $W_2$ acts trivially on $X_1$ and hence $X_1$ is $W$-invariant. This implies (2).\par
    Lastly, we have $\QQ_{\ge0}\L\cap X_{1,\QQ}\subset X_\QQ^+\cap X_{1,\QQ}=X_{1,\QQ}^+$, and on the other hand
    \[X_{1,\QQ}^+=(\L^*_\QQ\oplus\QQ^{\Delta_1})\cap X_\QQ^+=\L^*_\QQ\oplus(\QQ^{\Delta_1}\cap X_\QQ^+)\subset\QQ_{\ge0}\L\cap X_{1,\QQ}\]
    where the middle equality follows from the fact that $\L^*\subset X^+$ and the last inclusion follows from the definition of $\Delta_1$. Thus (3) follows.
\end{proof}

\begin{prop}\label{prop:max-closed-prime-ideal}
    With notations as above, the set $X_1^+$ is a submonoid of $\L$ and the complement $\L\setminus X_1^+$ is the maximal proper downward closed prime ideal of $\L$.
\end{prop}
\begin{proof}
    By the discussion above we have $X_1^+\subset X_{1,\QQ}^+=\QQ_{\ge0}\L\cap X_{1,\QQ}\subset\QQ_{\ge0}\L$. Since $\L$ is saturated in $X^+$, we deduce that $X_1^+\subset\L$.\par 
    Next we claim that 
    \[\QQ_{\ge0}\L=X_{1,\QQ}^+\oplus(\QQ_{\ge0}\L\cap X_{2,\QQ}^+)=X_{1,\QQ}^+\oplus(\QQ_{\ge0}\L\cap X_{2,\QQ}).\]
    Indeed, the second equality is clear since $\QQ_{\ge0}\L\subset X_\QQ^+$. Now suppose that an element $x\in\QQ_{\ge0}\L$ decomposes as $x=x_1+x_2$ where $x_1\in X_{1,\QQ}^+$ and $x_2\in X_{2,\QQ}^+$. Then $x_1\in\QQ_{\ge0}\L$ by the inclusion $X_{1,\QQ}^+\subset\QQ_{\ge0}\L\cap X_{1,\QQ}$ we just proved. To show the first equality it remains to show that $x_2\in\QQ_{\ge0}\L\cap X_{2,\QQ}^+$. After adding an element in $\L^*_\QQ$, we may assume that $x_1\in\QQ^{\Delta_1}\cap X_\QQ^+\subset\QQ_{\ge0}^{\Delta_1}$. Since $\QQ_{\ge0}\L$ is downward closed, we get that $x_2=x-x_1\in\QQ_{\ge0}\L\cap X_{2,\QQ}^+$. This proves the claim.\par
    Since $\L^*=\L\cap(-\L)\subset X_1^+$, the cone $\QQ_{\ge0}\L\cap X_{2,\QQ}$ is strictly convex. Then from the decomposition above we deduce that $\L\setminus X_1^+$ is an ideal in $\L$, which is moreover a prime ideal since $X_1^+$ is a submonoid of $\L$.\par 
    Next we show that $\L\setminus X_1^+$ is downward closed. This is equivalent to showing that $X_1^+$ is upward closed in $\L$. Suppose $\la_1\in X_1^+$ and $\la\in\L$ with $\la_1\le\la$. After multiplying by some positive integer, we may assume that $\la_1\in\L^*\oplus\ZZ^{\Delta_1}$. Then there exists $\la_0\in\L^*$ with $\la_0\le\la_1\le\la$. Since $-\la_0\in\L$, we have $\la-\la_0\in\L\cap X_{\der,\QQ}\subset\QQ^{\Delta_1}$ and hence $\la\in X_{1,\QQ}\cap\L\subset X_1^+$. \par 
    Finally, let $\J\subset\L$ be a proper downward closed prime ideal and we need to show that $\J\subset\L\setminus X_1^+$, or equivalently that $X_1^+\subset\L\setminus\J$. Since $\J\ne\L$ we must have $\L^*\subset\L\setminus\J$. (Otherwise there would exist $\la\in\L^*\cap\J$, then $-\la\in\L$ and so $0=\la+(-\la)\in\J$ and hence $\L=0+\L\subset\J$, a contradiction.) Since $\J$ is downward closed, the complement $\L\setminus\J$ is upward closed and hence 
    $(\L^*+\NN^\Delta)\cap\L\subset\L\setminus\J$. Clearly $\L\setminus\J$ is saturated in $\L$ (since $\J$ is an ideal of $\L$), so we have $(\L^*+\QQ_{\ge0}^{\Delta})\cap\L\subset\L\setminus\J$. On the other hand, we have $X_1^+\subset(\L^*_\QQ+\QQ_{\ge0}^\Delta)\cap\L$ and so for any $\la\in X_1^+$ there exists an integer $N>0$ such that $N\la\in(\L^*+\NN^\Delta)\cap\L\subset\L\setminus\J$ which implies that $\la\in\L\setminus\J$. So we conclude that $X_1^+\subset\L\setminus\J$. 
\end{proof}

According to the previous Proposition, we can form the following $G\times G$-stable subscheme of $\M$: 
\[\O_{\min}\coloneqq\M(\L/(\L\setminus X_1^+))=\Spec({}_\ZZ\bfO(X_1^+)\otimes_\ZZ\O_S).\]

\begin{lem}\label{lem:Imax}
    The set $\I_{\max}\coloneqq W\cdot \L\setminus X_1$ is the largest proper prime ideal of $W\cdot \L$. It is $W$-invariant, downward closed, and equals the downward closure $\overline{\L\setminus X_1^+}^{\leq}$ of $\L\setminus X_1^+$ inside $W\cdot \L$ and satisfies $\I_{\max}\cap \L=\L\setminus X^+_1$.
\end{lem}
\begin{proof}
    According to Lemma \ref{lem:restriction prime ideal of WL} and Proposition \ref{prop:bijection closed prime L WL}, there exists a proper downward closed prime ideal $\I'$ of $W\cdot \L$ such that 
    \[\L\setminus X_1^+=(\bigcap_{w\in W}w\cdot \I')\cap \L=\I'\cap\L.\] 
    On the other hand, let $\I_0\subset W\cdot\L$ be the maximal proper prime ideal. Then for any $w\in W$, the union $\I_0\cup w(\I_0)$ is also a proper prime ideal of $W\cdot\L$ and so $\I_0\cup w(\I_0)\subset\I_0$. This shows that $\I_0$ is $W$-stable. Then for any element $u(x)\in\I_0\subset W\cdot\L$ with $u\in W$ and $x\in\L$, we have $x\in u^{-1}(\I_0)\cap\L=(\bigcap_{w\in W}w(\I_0))\cap\L$. Since the latter set is a proper downward closed prime ideal of $\L$ by Lemma \ref{lem:restriction prime ideal of WL}, we must have $x\in  \L\setminus X^+_1$ by Proposition \ref{prop:max-closed-prime-ideal}. Then by the above equality we have $x\in\I'$. Since $\I'$ is downward closed and $u(x)\le x$ (because $x\in\L\subset X^+$), we get $u(x)\in\I'$. This shows that $\I_0\subset\I'$ and by the maximality of $\I_0$ we get that $\I_0=\I'$. In summary, we have shown that $\I'$ is the maximal proper prime ideal of $W\cdot\L$ and it is $W$-stable and downward closed.\par 
    Let $\I''\coloneqq\overline{\L\setminus X_1^+}^{\leq}$ be the downward closure of $\L\setminus X_1^+$ inside $W\cdot\L$. All that remains is to prove the three inclusions 
    \[\I_{\max}\subset \I''\subset\I'\subset \I_{\max}.\]
    The middle inclusion is clear by the definition of $\I'$. Let $w\in W$, $x\in \L$. Recall that $X_1\cap \L=X_1^+$ and $X_1=W\cdot X_1^+$ by Lemma \ref{lem:basic properties X_1}. If $w\cdot x\in \I_{\max}$, then $x\in \L\setminus X_1^+$. Hence $w\cdot x\in \I''$ because $w\cdot x\leq x$. This proves the first inclusion. On the other hand, if $w\cdot x\in \I'$, then since $\I'$ is $W$-stable we get $x\in\I'\cap\L=\L\setminus X_1^+$. Therefore $w(x)\in W\cdot\L\setminus X_1=\I_{\max}$. This proves the last inclusion and we are done. 
\end{proof}
\begin{defi}
    The \emph{minimal idempotent} of $\M$ is the element 
    \[e_{\min}\coloneqq e_{\I_{\max}}\in\O_{\min}(S)\subset\M(S)\] 
    associated to $\I_{\max}=W\cdot\L\setminus X_1$.
\end{defi}

\begin{cor}\label{cor:min-orbit bis}
    The scheme $\O_{\min}$ represents the orbit of the minimal idempotent $e_{\min}$ and is the smallest closed $G\times G$-stable subscheme of $\M$ that is flat and of finite presentation over $S$ with integral geometric fibers. If $S=\Spec(k)$ where $k$ is an algebraically closed field, it is thus the minimal $G\times G$ orbit in $\M$.
\end{cor}
\begin{proof}
    According to Theorem \ref{theo:parametrization GxG orbits} and Proposition \ref{prop:max-closed-prime-ideal}, the second assertion concerning the minimum property is clear and the third follows from the two previous ones. Let $\O_{\min}'$ be the orbit of the section $e_{\min}$. According to Corollary \ref{cor:orbit closure} and the previous Lemma \ref{lem:Imax}, $\O_{\min}'$ is representable and we have a universal scheme-theoretically dominant open immersion $\O_{\min}'\hookrightarrow \O_{\min}$. It suffices to prove it is surjective. Since the formation of $\O_{\min}'$ and $\O_{\min}$ commute with any base change, we can assume that $S=\Spec(k)$ where $k$ is an algebraically closed field. Then the result follows by the minimality of $\O_{\min}$. 
\end{proof}

Let $Y_1=\mathrm{Hom}_\ZZ(X_1,\ZZ)$. Then we have a homomorphism $Y\to Y_1$ that induces an embedding $\Delta_1\into Y_1$, $i\mapsto\alpha_i^\vee$. So we get a based root datum $\Psi_1\coloneqq(\Delta_1,Y_1,X_1,\{\alpha_i,i\in\Delta_1\},\{\alpha_i^\vee,i\in\Delta_1\})$. 
\begin{cor}\label{cor:min-orbit}
    The scheme $\O_{\min}$ is a split reductive group scheme over $S$ with based root datum $\Psi_1$. Its unit section is given by $e_{\min}$ and the natural embedding $\O_{\min}\into\M$ is a homomorphism of semigroup schemes. 
\end{cor}
\begin{proof}
    We may assume that $S=\Spec(\k)$ is affine. We will use the modified quantized enveloping algebras studied in \cite{CJ1}. Recall the embedding of the (non-unital) $\QQ(v)$-algebras $\iota\colon\dot{\bfU}_{\Psi_1}\into\dot{\bfU}$ from \cite[Equation (3.6)]{CJ1}. According to \cite[Proposition 3.13]{CJ1}, $\iota$ induces bijections $\dot{\bfB}_{\Psi_1}[\lambda]\xrightarrow[]{\sim}\iota(\dot{\bfB}_{\Psi_1}[\la])=\dot{\bfB}[\la]$ for all $\la\in X_1^+$. Thus the dual of $\iota$ induces a bijection 
    \[\iota^*\colon{}_\k\bfO(\L)/{}_\k\bfO(\L\setminus X_1^+)={}_\k\bfO(X_1^+)\xrightarrow[]{\sim}{}_\k\bfO_{\Psi_1}\] 
    which is compatible with the comultiplication of ${}_\k\bfO(\L)/{}_\k\bfO(\L\setminus X_1^+)$ (induced from that of ${}_\k\bfO(\L)$) at the source and the one of the Hopf algebra ${}_\k\bfO_{\Psi_1}$ at the target. Moreover, by definition, $\iota$ fixes the elements $1_{\lambda}$ for $\lambda \in X_1$. Hence pre-composing by $\iota^*$ transforms the co-unit ${}_\k\bfO_{\Psi_1}\to\k$ of ${}_\k\bfO_{\Psi_1}$ into the map ${}_\k\bfO(\L)/{}_\k\bfO(\L\setminus X_1^+)\to\k$ corresponding to $e_{\min}$, as these two ring morphisms are informally given by the formula $$\varphi\mapsto \sum_{\lambda \in X_1}\varphi(1_{\lambda}).$$
    Therefore, it remains to show that $\iota^*$ is a ring homomorphism. For this we need to show that for any $a,b,c\in\dot{\bfB}_{\Psi_1}$, we have $\langle\iota^*(\iota(a)^*\iota(b)^*),c\rangle=\langle a^*b^*,c\rangle$, or equivalently
    \[\langle\iota(a)^*\otimes\iota(b)^*,\Delta(\iota(c))\rangle=\langle a^*\otimes b^*,\Delta_{\Psi_1}(c)\rangle.\]
    Recall the projection $p\colon\dot{\bfU}\to\dot{\bfU}_{\Psi_1}$. By \cite[Lemma 3.12]{CJ1} we have $\Delta_{\Psi_1}(c)=(p\otimes p)\circ\Delta(\iota(c))$ and hence the right hand side above equals the coefficient of $\iota(a)\otimes\iota(b)$ in the expression of $\Delta(\iota(c))$ in terms of the basis $\dot{\bfB}\otimes\dot{\bfB}$, which is exactly the left hand side. This finishes the proof.
\end{proof}

\begin{lem}\label{lem:Stembridge}
    Let $P\subset X^+$ be a downward closed subset. Then the following are equivalent:
    \begin{enumerate}
        \item[(i)] $P$ does not contain any root;
        \item[(ii)] $P\cap\NN^\Delta\subset\{0\}$;
        \item[(iii)] $P\setminus\{0\}$ is downward closed.
    \end{enumerate}
\end{lem}
\begin{proof}
    (i)$\Leftrightarrow$(ii): The implication ``$\Leftarrow$" is obvious. Now assume (i) and suppose that there exists $\lambda\in P\cap\NN^\Delta$ and $\la\ne0$. Then in particular $0<\la$. The nonempty finite partially ordered set $\set{\mu\in X^+|0<\mu\leq \lambda}$ has a minimal element which must be a root by \cite[Corollary 2.7]{Stembridge} and also must lie in $P\subset X^+$ by the fact that $P$ is downward closed. But this contradicts (i).\par
    The equivalence (ii)$\Leftrightarrow$(iii) is clear by definition.
\end{proof}

\begin{prop}\label{prop:characterization zero of reductive monoid} 
Let $\M$ be a reductive monoid scheme over a nonempty and connected scheme $S$ with unit group $G$ and weight monoid $\L=\L(\M)$. The following conditions are equivalent:
\begin{enumerate}
    \item[(i)] the monoid scheme $\M$ has a zero;
    \item[(ii)] The minimal orbit coincides with the minimal idempotent, i.e. the $S$-point $e_{\min}\in\O_{\min}(S)$ identifies the base scheme $S$ with $\O_{\min}$;
    \item[(iii)] $\L\setminus\{0\}$ is a downward closed ideal of $\L$;
    \item[(iii')] $\L\cap(-\L)=\{0\}$ and $\L\cap\NN^\Delta=\{0\}$;
    \item[(iv)] $W\cdot\L\cap(-W\cdot\L)=\{0\}$;
    \item[(v)] the monoid scheme $\M_T$ has a zero.
\end{enumerate}
If these conditions are satisfied, then $e_{\min}$ is the zero of $\M$ and its image under the abelianisation map is the zero for $A_{\M}$.
\end{prop}
\begin{proof}
    The equivalence (iii)$\Leftrightarrow$(iii') follows from Lemma \ref{lem:Stembridge}. To prove the other implications we may assume that $S=\Spec(\k)$ is affine.\par 
    (i)$\Rightarrow$(ii): The zero of $\M$ is a $G\times G$-stable section of the structure morphism and hence coincides with $\O_{\min}$ and also $e_{\min}$. \par 
    (ii)$\Rightarrow$(iii): By the definition of $\O_{\min}$, condition (ii) implies that $X_1^+=\{0\}$ and then we conclude by Proposition \ref{prop:max-closed-prime-ideal}.\par 
    (iii)$\Rightarrow$(ii): The proper downward closed ideal $\L\setminus\{0\}\subset\L$ is the maximal one. So the minimal orbit is $\M(\L/(\L\setminus\{0\}))=S$.\par 
    (ii)$\Rightarrow$(i): Consider the morphism $f\colon\M\to\M$ defined by left multiplication by $e_{\min}$. Since $e_{\min}$ identifies $S$ with the minimal orbit $\O_{\min}$, we have $\M^\times\subset f^{-1}(e_{\min})$. Since $G$ is schematically dense in $\M$ by Lemma \ref{lem:quotient-flat}, we deduce that $e_{\min}\cdot\M=\O_{\min}$. Similarly we have $\M\cdot e_{\min}=\O_{\min}$ and therefore $e_{\min}$ is a zero for $\M$.\par 
    (iii')$\Rightarrow$(iv): Suppose that $W\cdot \L \cap (-W\cdot \L)\ne\{0\}$. Then there exists $\la,\mu\in\L\setminus\{0\}$ and $w_1,w_2\in W$ such that $w_1(\la)=-w_2(\mu)$. Let $w\coloneqq w_1^{-1}w_2\in W$. Then we have $w(\mu)=-\la$ and hence $\mu-w(\mu)=\mu+\la\in\L$. On the other hand we have $\mu-w(\mu)\in\ZZ^\Delta$ and hence $\mu-w(\mu)\in\L\cap\ZZ^\Delta=\L\cap\NN^\Delta=\{0\}$. Thus we have $\mu=w(\mu)=-\la$ so that $0\ne\mu\in\L\cap(-\L)$, contradicting (iii').\par
    (iv)$\Rightarrow$(iii'): Clearly (iv) implies that $\L\cap(-\L)=0$. Suppose $\L\cap\NN^\Delta\ne\{0\}$, then by Lemma \ref{lem:Stembridge} there exists a root $\alpha\in\L$. Let $s_\alpha\in W$ be the associated reflection. Then we have $s_\alpha(\alpha)=-\alpha\in W\cdot\L$ which means that $\alpha\in W\cdot\L\cap(-W\cdot\L)$ and contradicts (iv).\par
    (iv)$\Rightarrow$(v): By Theorem \ref{theo:properties AM MT} we have $\M_T\cong\Spec(\k[W\cdot\L])$ where $\k[W\cdot\L]$ is the monoid algebra of $W\cdot\L$. Condition (iv) implies that $W\cdot\L\setminus\{0\}$ is a prime ideal of $W\cdot\L$ and spans an ideal of $\k[W\cdot\L]$ that defines an element $m\in\M_T(\k)$. Recall that the product on $\M_T$ corresponds to the co-product $\k[W\cdot\L]\to\k[W\cdot\L]\otimes_\k\k[W\cdot\L]$ sending $e^\chi\mapsto e^\chi\otimes e^\chi$ for all $\chi\in W\cdot\L$. From this we easily deduce that $m$ is a zero for $\M_T$.\par 
    The implication (v)$\Rightarrow$(iv) is a special case of the implication (i)$\Rightarrow$(iii') which has already been proved.\par 
    It remains to show the last statement. We already noticed during the proof of (i)$\Rightarrow$(ii) that $e_{\min}$ is the zero of $\M$. By definition we have $e_{\min}=e_{W\cdot \L\setminus\set{0}}\in\M_T(S)\subset\M(S)$ and also, 
    \[\L_{\ab}\cap(\L\setminus\{0\})=\L_{\ab}\cap(W\cdot\L\setminus\{0\})=\L_{\ab}\setminus\{0\}.\]
    Thus the image of $e_{\min}$ in the abelianisation $A_{\M}=\Spec(\k[\L_{\ab}])$ is defined by the ideal $\k[\L_{\ab}\setminus\{0\}]$ and so is the zero of $A_\M$.
\end{proof}

\begin{cor}\label{cor:monoid-decomposition}
    Let $\M$ be a reductive monoid over a connected scheme $S$ with unit group $G$. Then there exists:
    \begin{itemize}
        \item split reductive groups $\widetilde{G}_1$, $\widetilde{G}_2$ equipped with a central isogeny $\widetilde{G}_1\times\widetilde{G}_2\to G$ with kernel $\widetilde{Z}$;
        \item a weight monoid $\L_2\subset \widetilde{X}_2^+$, where $\widetilde{X}_2^+$ is the dominant weight monoid for the pinning of $\widetilde{G}_2$ induced by the pinning of $G$;
    \end{itemize}
    such that we have an isomorphism of monoid schemes
    \[\M\cong(\widetilde{G}_1\times\M(\L_2))/\widetilde{Z}\]
    where $\M(\L_2)$ is the reductive $\widetilde{G}_2$-monoid attached to $\L_2$ (as in Definition \ref{theo:M(L)-def}) and it has a zero.
\end{cor}
\begin{proof}
    For $k=1,2$, let $\widetilde{X}_k$ be the image of $X$ under the projection $X_\QQ\to X_{k,\QQ}$. From $\widetilde{X}_k$ we get root datum with simple roots $\Delta_k$ and the associated reductive group scheme $\widetilde{G}_k$. Moreover, since $X\subset\widetilde{X}_1\oplus\widetilde{X}_2$ is a sub-lattice of finite index, we have a central isogeny $\widetilde{G}_1\times\widetilde{G}_2\to G$ whose kernel is $\widetilde{Z}\coloneqq\Spec_S(\O_S[(\widetilde{X}_1\oplus\widetilde{X}_2)/X])$. Let $\L_2\coloneqq\QQ_{\ge0}\L\cap\widetilde{X}_2$. Then $\L_2\subset \widetilde{X}_2^+$ is a weight monoid and we have $\L_2\cap(-\L_2)=\{0\}$ and $\L_2\cap\NN^{\Delta_2}=\{0\}$. Therefore the reductive $\widetilde{G}_2$-monoid $\M(\L_2)$ has a zero by Proposition \ref{prop:characterization zero of reductive monoid} and we have $\M=\M(\L)\cong(\widetilde{G}_1\times\M(\L_2))/\widetilde{Z}$ where the center acts diagonally. 
\end{proof}

\section{Very flat reductive monoids}\label{sec:very-flat}
In this section we study the important class of very flat reductive monoids. We will first give a characterization of them in terms of the associated weight monoid, and then introduce the Vinberg monoid and its universal property. Finally we will compare our construction of the Vinberg monoid with other definitions in the literature.\par 
Let $G$ be a pinned split reductive group over a nonempty and connected base scheme $S$ and we will freely use notations from \S\ref{sec:notations}.

\subsection{Characterization}
Let $\L\subset X^+$ be a weight monoid (see Definition \ref{def:weight-monoid}). Let $S$ be a base scheme and $\M=\M(\L)_S$ be the associated rigidified reductive $G$-monoid over $S$.\par 
Recall that $\L_{\ab}\coloneqq\L\cap X_{\ab}$ is a submonoid of $\L$ and $\L_{\ab}^*\coloneqq\L_{\ab}\cap(-\L_{\ab})$ is the maximal subgroup of $\L_{\ab}$. Apply Definition \ref{def:preorder-monoid-set} to $\L$, viewed as an $\L_{\ab}$-set, we get a pre-order $\preceq_{ab}$ and an equivalence relation $\sim_{ab}$ on $\L$ described as follows: for any $\la_1,\la_2\in\L$, we have
\begin{itemize}
    \item $\la_1\preceq_{ab}\la_2$ if $\la_2-\la_1\in\L_{ab}$;
    \item $\la_1\sim_{ab}\la_2$ if $\la_1\preceq_{ab}\la_2$ and $\la_2\preceq_{ab}\la_1$, or equivalently $\la_1-\la_2\in\L_{ab}^*$.
\end{itemize}
Let $\mathfrak{M}=\L_{\min}$ denote the set of minimal elements of $\L$ with respect to $\preceq_{\ab}$. Then $\mathfrak{M}$ is a union of $\L_{\ab}^*$ orbits. Moreover, we have $\L_{\ab}^*\subset\fM$. Indeed, for $\mu\in\L_{\ab}^*$ and $\la\in\L$ with $\la\preceq_{\ab}\mu$, we have $\nu\coloneqq\mu-\la\in\L_{\ab}$ so that $\la=\mu-\nu\in\L\cap X_{\ab}=\L_{\ab}$. On the other hand, we have $-\la=\nu-\mu\in\L_{\ab}$ and hence $\la\in\L_{\ab}^*$. Thus $\la\sim_{\ab}\mu$ which means that $\mu\in\fM$. Consequently, we have a natural inclusion
\begin{equation}\label{eq:L_ab-inclusion}
    \L_{\ab}^*\subset(\fM-\fM)\cap(\L_{\ab}-\L_{\ab})
\end{equation}
and an equality
\[\L=\fM+\L_{\ab}.\]
Indeed, we can choose a linear function $h\colon\L/\L^*\to\NN$ which is positive on every nonzero element. Then for any $\la\in\L$ and $\chi\in\L_\ab\setminus\L_\ab^*$ we have $0\le h(\la-\chi)<h(\la)$. Thus there exists $\chi\in\L_\ab\setminus\L_\ab^*$ such that $\la-\chi\in\fM$ (otherwise we would find an infinite sequence of elements in $\L$ whose $h$-value is strictly decreasing). This proves the above decomposition.\par 
We highlight two extreme cases: 
\begin{itemize}
    \item If $\L_{\ab}^{*}=\L_{\ab}\cap(-\L_{\ab})=\{0\}$, then the preorder $\preceq_{\ab}$ becomes a partial order and the equivalence relation $\sim_{\ab}$ becomes equality.
    \item If $X_{\ab}\subset\L$ so that $\L_{\ab}=X_{\ab}$, then the preorder $\preceq_{\ab}$ coincides with the equivalence relation $\sim_{\ab}$. In this case we have $\fM=\L$.
\end{itemize}

\begin{lem}\label{lem:abelization-flat}
    The abelianisation map $\M\to A_\M$ is flat if and only if $\L$ is a free $\L_{\ab}$-set (cf. Definition \ref{def:M-set}). In this case, the $\O_S[A_\M]=\O_S[\L_{\ab}]$-module $\O_S[\M]$ is free and for any set of representatives $\fM_1\subset\fM$ of the quotient $\fM/\L_{\ab}^*$, the set $\bigcup\limits_{\la\in\fM_1}\dot{\bfB}[\la]^*$ forms an $\O_S[A_\M]$-basis of $\O_S[\M]$.
\end{lem}
\begin{proof}
    The ``if" direction is clear. Now assume that the abelianisation map is flat. By Proposition \ref{prop:finite-invariant-review}(3), the $U^-\times U$-invariant subalgebra is described as $\bfO(\L)_S^{U^-\times U}=\bigoplus_{\la\in\L}\O_S\cdot 1_\la^*$ and is identified with the monoid algebra $\O_S[\L]$. By Proposition \ref{prop:finite-invariant-review}(2) it is an $\O_S[\L_\ab]$-submodule of $\O_S[\M]=\bfO(\L)_S$. Define the $\O_S$-linear projection $\bfO(\L)_S\to \bfO(\L)_S^{U^-\times U}$ that keeps the dual canonical basis elements $\{1_\la^*,\la\in\L\}$ and kills the other basis elements. By \cite[Lemma 3.10]{CJ1} this projection is $\O_S[\L_\ab]$-linear. Therefore $\bfO(\L)_S^{U^-\times U}\cong\O_S[\L]$ is a direct summand of $\O_S[\M]$ as an $\O_S[\L_{\ab}]$-module. Since $\O_S[\M]$ is $\O_S[\L_{\ab}]$-flat by assumption, the algebra $\O_S[\L]$ is also a flat $\O_S[\L_{\ab}]$-algebra. After choosing a point of $S$ and passing to its residue field, we see that there exists a field $k$ such that $k[\L]$ is a flat $k[\L_{\ab}]$-module. Clearly $\L$ and $\L_{\ab}$ are fine monoids and the inclusion $\L_{\ab}\into\L$ is a local homomorphism since $\L_{\ab}=\L\cap X_{\ab}$. By Lemma \ref{lem:free-set-flat} we conclude that $\L$ is a free $\L_{\ab}$-set and any choice of representatives $\fM_1$ of $\fM/\L_{\ab}^*$ forms an $\L_{\ab}$-basis of $\L$. Then the last statement follows from \cite[Lemma 3.10]{CJ1}.
\end{proof}

\begin{prop}\label{prop:very-flat-characterisation}
    With notations as above, the following conditions are equivalent:
    \begin{enumerate}
        \item $\M$ is very flat in the sense of Definition \ref{def:very flat}.
        \item $\L$ is a free $\L_{\ab}$-set (cf. Definition \ref{def:M-set}) and $\fM$ is a submonoid of $\L$.
        \item There exists a homomorphism $\theta\colon Z\to T_{\der}$ that restricts to the identity on $Z_{\der}$ such that $\L=\L_{\ab}\oplus\fM_1$ where
        \[\fM_1\coloneqq\{(\theta^*(\la),\la)\in X^+\subset X_Z\oplus X_{\der}^+\mid\la\in X_{\der}^+\}\]
        is a submonoid of $\L$.
    \end{enumerate}
    If these conditions are satisfied, then $\O_S[\M]$ is a free $\O_S[\L_{\ab}]$-module with basis: 
    \[\bigsqcup_{\la\in X^+_{\der}}\dot{\bfB}[(\theta^*(\la),\la)]^*.\]
\end{prop}
\begin{proof}
    (1)$\Rightarrow$(2): By Lemma \ref{lem:abelization-flat} it remains to show that $\fM$ is a submonoid of $\L$. After choosing a geometric point on $S$ we may assume that $S=\Spec(k)$ where $k$ is an algebraically closed field. Let $e_0\in A_\M(k)$ be the idempotent corresponding to the $k$-algebra homomorphism $k[\L_{\ab}]\to k$ defined by 
    \[e^\chi\mapsto\begin{cases}
        0, &\text{ if }\chi\in\L_{\ab}\setminus\L_{\ab}^*\\
        1,&\text{ if }\chi\in\L_{\ab}^*.
    \end{cases}\]
    and let $\M_{e_0}$ be the fiber of the abelianisation map over $e_0$. The coordinate ring of $\M_{e_0}$ equals $k[\M_{e_0}]=k[\M]/I_{e_0}$, where
    \[I_{e_0}\coloneqq \sum_{\chi\in\L_{\ab}\setminus\L_{\ab}^*}k[\M]\cdot 1_\chi^*+\sum_{\chi\in\L_{\ab}^*}k[\M]\cdot(1_\chi^*-1)\]
    is the ideal generated by the elements $\{1_\chi^*,\;\chi\in\L_{\ab}\setminus\L_{\ab}^*\}$ and $\{1_\chi^*-1,\;\chi\in\L_{\ab}^*\}$. Then for any $\la\in\fM$, the image of $1_\la^*\in{}_k\bfO(\L)$ in $k[\M_{e_0}]$ is nonzero. By assumption the scheme $\M_{e_0}$ is integral and hence its coordinate ring $k[\M_{e_0}]$ is an integral domain. Then we deduce that $\fM\subset\L$ is a submonoid. \par  
    (2)$\Rightarrow$(3): By condition (2) and Lemma \ref{lem:L-X-der+surj}, the natural homomorphism $X\to X_{\der}$ restricts to a surjection of free abelian groups 
    \[p\colon\fM^{\gp}=\fM-\fM\to X_{\der}\]
    which further induces a surjection $\fM\epic X_{\der}^+$ and an isomorphism $\fM^{\gp}/\L_{\ab}^*\cong X_{\der}$. Choose a section $q\colon X_{\der}\to\fM^{\gp}$ of $p$ and let $\fM_1\coloneqq q(X_{\der}^+)$. We first claim that $\fM_1\subset\fM$. For each $\la\in X_\der^+$, choose $m\in\fM$ with $p(m)=\la$. Then we have $p(q(\la)-m)=\la-p(m)=0$ and hence $q(\la)-m\in\L_\ab^*$. So $q(\la)\in\fM$ and the claim is proved. Then the submonoid $\fM_1\subset\fM$ is a set of representatives of the quotient $\fM/\L_{\ab}^*$. Moreover, there exists a homomorphism of group schemes $\theta\colon Z\to T_{\der}$ such that $q(\la)=(\theta^*(\la),\la)\in X\subset X_Z\oplus X_{\der}$ for all $\la\in X_{\der}$. In particular, we have $\theta^*(\la)|_{Z_{\der}}=\la|_{Z_{\der}}$ and hence $\theta$ restricts to the identity on $Z_{\der}$.\par 
    (3)$\Rightarrow$(1): By Lemma \ref{lem:abelization-flat} we know that the abelianisation map $\pi\colon\M\to A_{\M}$ is flat and it remains to show that all its geometric fibers are integral. Since the abelianisation map commutes with arbitrary base change, we may and do assume that $S=\Spec(k)$ where $k$ is an algebraically closed field. Since the abelianisation map is a GIT quotient through the connected reductive group $G_{\der}\times G_{\der}$, its geometric fibers are connected. They are also noetherian and it suffices to prove that for any point $x\in \M$, the local rings of the geometric fiber $\M_{\pi(x)}$ are integral. This is equivalent to proving that the subset 
    \[E= \set{x\in \M|\text{ the local ring }\O_{\M_{{\pi(x)}},x} \text{ is integral}}\]
    equals $\M$. According to \cite[12.1.1(viii)]{EGAIV8a15}, the set $E$ is open in $\M$. Clearly it is also invariant under the action of $G\times G$. Thus it suffices to show that $e_{\min}\in E$ (then by Corollary \ref{cor:min-orbit bis} the set $E$ will contain the minimal orbit $\O_{\min}$ which lies in the closure of any $G\times G$-orbit of $\M$, so the complement of $E$, being a closed $G\times G$-invariant subset of $\M$, will be empty). Note that  $\pi(e_{\min})=e_0$ is the idempotent considered above. So it suffices to show that the geometric fiber $\M_{e_0}$ is integral.\par
    First we make a reduction. Consider the torus $G_{\L_\ab}\coloneqq\Spec(k[\L_\ab^*])$. Then there are morphisms of monoid schemes $G\into\M\xrightarrow{\beta} G_{\L_\ab}$ whose composition is a surjective group homomorphism. Let $\M'\coloneqq\beta^{-1}(1)$. Then $\M'$ is the closed subscheme of $\M$ corresponding to the ideal of $k[\M]$ generated by the elements $\{1_\chi^*-1,\chi\in\L_\ab^*\}$. Let $G'\coloneqq\ker(G\to G_{\L_\ab})$. Then $\M'$ is a reductive $G'$-monoid with weight monoid $\L/\L_\ab^*\subset X/\L_\ab^*$ and abelianisation $A_{\M'}=\Spec k[\L_\ab/\L_\ab^*]$, which is a commutative reductive monoid with unit group $\Spec(k[X_\ab/\L_\ab^*])=\ker(G_\ab\to G_{\L_\ab})$. The idempotent $e_0$ is the zero of $A_{\M'}$ and $\M_{e_0}$ is the fiber of the abelianisation map $\M'\to A_{\M'}$ over zero.\par 
    Thus we are reduced to the case where $\L_\ab^*=\{0\}$. In this case the ideal $I_{e_0}$ is generated by the elements $\{1_\chi^*,\chi\in\L_\ab\setminus\{0\}\}$.\par
    By condition (3) and \cite[Lemma 3.10]{CJ1}, we get that $I_{e_0}={}_k\bfO(\L\setminus\fM_1)$, i.e. the elements in $\bigcup_{\la\in\L\setminus\fM_1}\dot{\bfB}[\la]^*$ form a $k$-basis of $I_{e_0}$. \par
    We claim that $\L\setminus\fM_1$ is a downward closed prime ideal in $\L$. According to condition (3), any element in $\L$ has the form $(\chi+\theta^*(\la),\la)$ with $\la\in X_\der^+$ and $\chi\in\L_\ab$. Suppose that $(\theta^*(\la),\la)\le(\chi+\theta^*(\mu),\mu)$ where $\la,\mu\in X_\der^+$ and $\chi\in\L_\ab$. Then we have $\theta^*(\la)=\chi+\theta^*(\mu)$ and $\la\le\mu$. Since $\L$ is downward closed and $(\theta^*(\mu),\mu)\in\L$ we have $(\theta^*(\mu),\la)\in\L$ and hence $\theta^*(\mu)-\theta^*(\la)\in\L_\ab$. Then we get that $\chi=\theta^*(\la)-\theta^*(\mu)\in\L_\ab\cap(-\L_\ab)$. Hence $\chi=0$ due to our assumption that $\L_\ab^*=\{0\}$. This shows that $\fM_1$ is upward closed and so $\L\setminus\fM_1$ is downward closed in $\L$. Next we show that $\L\setminus\fM_1$ is an ideal of $\L$. Take any element $(\chi+\theta^*(\la),\la)\in\L\setminus\fM_1$. Then we have $\chi\ne0$. For any element $(\chi'+\theta^*(\la'),\la')\in\L$, we have $\chi'+\chi\ne0$ by the assumption that $\L_\ab^*=\{0\}$ and hence $(\chi'+\theta^*(\la'),\la')+(\chi+\theta^*(\la),\la)\in\L\setminus\fM_1$. Thus $\L\setminus\fM_1$ is an ideal of $\L$. Moreover, since $\fM_1$ is a submonoid of $\L$ we get that $\L\setminus\fM_1$ is a prime ideal of $\L$ and the claim follows.\par  
    Then we conclude by Lemma \ref{lem:M(L/J)} that $k[\M_{e_0}]$ is a normal integral domain. 
\end{proof}

\begin{rqe}
    In the language of cones, condition (3) above can be reformulated as follows. Let $D_\L\coloneqq\QQ_{\ge0}\L_{\ab}\subset X_{\ab,\QQ}$ be the convex cone associated to $\L_{\ab}$. Then $D_\L$ satisfies $D_\L-D_\L=X_{\ab,\QQ}$ and $\theta^*(\alpha_i)\in D_\L$ for all $i\in\Delta$ and the cone $\QQ_{\ge0}\L\subset X_\QQ$ has the form
    \[\QQ_{\ge0}\L=\{(\chi,\la)\in X_{\ab,\QQ}\oplus X_{\der,\QQ}=X_\QQ\mid \la\in \QQ_{\ge0}X_\der^+,\;\chi-\theta^*(\la)\in D_\L\}.\]
\end{rqe}

\begin{lem}\label{lem:pullback-abelization}
    Keep notations as above and suppose that the reductive monoid $\M$ over $S$ is very flat. Let $f\colon A'\to A_\M$ be a morphism of reductive monoid schemes where $A'$ is a rigidified reductive monoid whose unit group $(A')^\times$ is a split torus over $S$.
    Let $G'\coloneqq(A')^\times\times_{G_{\ab}}G$. Then the fiber product $\M'\coloneqq A'\times_{A_\M}\M$ is a very flat rigidified reductive $G'$-monoid and its weight monoid $\L(\M')$ equals the image of $\L(A')\oplus\L$ in the cokernel 
    \[X'\coloneqq \mathrm{coker}(X_\ab\to X'_{\ab}\oplus X)\] 
    where $X'_{\ab}$ is the character group of the torus $(A')^\times$ and the map is induced by $f^*$ on the first factor and the negative of the natural embedding $X_\ab\into X$ on the second factor.
\end{lem}
\begin{proof}
    We may assume that $S=\Spec(\k)$ is affine. By checking on functors of points, it is clear that $\M'$ is a monoid scheme over $S$ with unit group $G'$. Since $X/X_{\ab}\cong X_{\der}$ is torsion-free, the cokernel $X'$ is a free abelian group of finite rank. Then the based root datum of $G'$ has character lattice $X'$ and the simple roots/coroots for $G'$ are identified with those for $G$. In particular, the dominant integral weight monoid $X'^+$ of $G'$ is the image of $X'_{\ab}\oplus X^+$ in $X'$.\par 
    Let $\L'$ be the image of $\L(A')\oplus\L$ in $X'$. We claim that $\L'$ is a weight monoid in $X'^+$. Clearly $\L'$ is finitely generated since both $\L(A')$ and $\L$ are finitely generated. The fact that $\L'$ generates the abelian group $X'$ follows from similar facts for $\L(A')$ and $\L$, and the fact that $\L\subset X^+$ is downward closed immediately implies that $\L'$ is downward closed. Thus it remains to show that $\L'$ is saturated. Let $\alpha\in X'_{\ab},\beta\in X$ represent an element $(\alpha,\beta)\in X'$. Suppose there exists an integer $n\ge1$ and elements $\mu\in\L(A'),\la\in\L$ such that $n(\alpha,\beta)=(\mu,\la)$ in $X'$. Then by definition there exists $x\in X_{\ab}$ such that $\mu+f^*(x)=n\alpha$ and $\la-x=n\beta$. Since $\M$ is very flat, by Proposition \ref{prop:very-flat-characterisation} we have $\L=\L_{\ab}\oplus\fM_1$ where $\fM_1\subset\L$ is a submonoid. This induces a decomposition $X=X_{\ab}\oplus\fM_1^{\gp}$. Under these decompositions we write $\beta=\beta_0+\beta_1$ with $\beta_0\in X_{\ab}$ and $\beta_1\in\fM_1^{\gp}$ and also $\la=\la_0+\la_1$ with $\la_0\in\L_{\ab}$ and $\la_1\in\fM_1$. Then we have $x=\la_0-n\beta_0$ and $\la_1=n\beta_1$. Since $\L$ is saturated in $X$, the direct factor $\fM_1$ is also saturated in $X$ and therefore $\beta-\beta_0=\beta_1\in\fM_1\subset\L$. On the other hand, we have $\mu+f^*(\la_0)=n(\alpha+f^*(\beta_0))\in\L(A')$ and from the fact that $\L(A')$ is saturated in $X'_{\ab}$ we get that $\alpha+f^*(\beta_0)\in\L(A')$. Since $(\alpha,\beta)=(\alpha+f^*(\beta_0),\beta-\beta_0)$ in $X'$, we conclude that $(\alpha,\beta)\in\L'$ and hence $\L'$ is saturated.\par     
    Let ${}_\k\bfO'=\k[G']$ (resp. $\dot{\bfB}'^*$) be the ring of regular functions on $G'$ (resp. its dual canonical basis as $\k$-module). By construction $\M'$ is flat over $A'$ and hence flat over $S$. Moreover, the structure morphism $\M'\to S$ has integral geometric fibers by \cite[Proposition 4.6.5(ii)]{EGAIV2a7}. Then Lemma \ref{lem:quotient-flat} implies that $\k[\M']=\k[\L'_{ab}]\otimes_{\k[\L_{\ab}]}\k[\M]$ is a subalgebra of ${}_\k\bfO'=\k[G']$. By \cite[Lemma 3.14]{CJ1} and \cite[Lemma 3.10]{CJ1}, we have a canonical isomorphism $\k[X'_{\ab}]\otimes_{\k[X_{\ab}]}{}_\k\bfO\cong{}_\k\bfO'$ that maps $e^{\chi'}\otimes\dot{\bfB}[\la]^*$ bijectively onto $\dot{\bfB}'[(\chi',\la)]^*$ for each $\chi'\in X'_{\ab}$ and $\la\in X^+$. Under this isomorphism the subring $\k[\M']$ is mapped onto the subring ${}_\k\bfO'(\L')$ spanned by the elements in $\bigsqcup_{\la'\in\L'}\dot{\bfB'}[\la']^*$. Since $\L'$ is a weight monoid in $X'^+$, $\M'$ is a rigidified reductive $G'$-monoid over $S$ with weight monoid $\L(\M')=\L'$. Moreover, by construction $A'$ is the abelianisation of $\M'$ and $\M'$ is very flat.
\end{proof}

\subsection{The Vinberg monoid}
We define the enhanced group for $G$ to be the reductive group scheme $G_\enh\coloneqq(T\times G)/Z_G$ where $Z_G$ embeds anti-diagonally. It inherits a pinning from $G$ with Borel subgroup $B_\enh\coloneqq(T\times B)/Z_G$, maximal torus $T_\enh\coloneqq(T\times T)/Z_G$ and character lattice
\[X_\enh\coloneqq X^*(T_\enh)=\{(\mu,\la)\in X\oplus X\mid \mu-\la\in\ZZ^\Delta\}\]
in which the simple roots are $\Delta_\enh=\{(0,\alpha)\mid\alpha\in\Delta\}$.
\begin{defi}\label{defi:Vinberg monoid}
    The \emph{Vinberg monoid}, or \emph{universal enveloping monoid} for $G$, denoted $\mathrm{Vin}_G$, is the reductive $G_\enh$-monoid over $S$ associated to the weight monoid
    \[\L_\Vin\coloneqq\{(\mu,\la)\in X\oplus X\mid\la\in X^+,\mu\ge\la\}\subset X_\enh.\]
    It comes with a natural rigidification $\iota\colon G_\enh\cong(\Vin_G)^\times$. 
\end{defi}
We have 
\[\L_\Vin^*=\L_\Vin\cap(-\L_\Vin)=\{(\mu,\mu)\in X^2\mid \mu\in X_{\ab}=X^+\cap(-X^+)\}.\]
Then we deduce that 
\[\Vin_G\cong (Z\times\Vin_{G_\der})/Z_\der\]
and hence the study of $\Vin_G$ can be reduced to the case where $G$ is semisimple. Let $r$ be the cardinality of $\Delta$ and write $\Delta=\{\alpha_1,\dotsc,\alpha_r\}$.
\begin{prop}
    The Vinberg monoid $\mathrm{Vin}_G$ is a very flat reductive monoid. Suppose that $G$ is semisimple. Then $\Vin_G$ has a zero and its abelianisation map is identified with the map
    \[\mathrm{Vin}_G\to A_{\mathrm{Vin}_G}\cong\AA^r_S\]
    induced by the characters $\chi_{(\alpha_i,0)}=(\alpha_i,0)\colon G_\enh\to\GG_m$, $1\le i\le r$. If moreover $G$ is semisimple and simply connected, then the abelianisation map together with the characters $\chi_{(\omega_i,\omega_i)},1\le i\le r$, induce an isomorphism
    \[\mathrm{Vin}_G//\Ad(G)\cong\AA^{2r}_S.\]
\end{prop}
\begin{proof}
    By the discussion above we may assume that $G$ is semisimple. By definition we have $\L_{\mathrm{Vin},\ab}=\NN^\Delta\subset X$ and the description of the abelianisation is clear. Also $\L_{\mathrm{Vin}}$ is a free $\L_{\mathrm{Vin},\ab}$-set with basis $\{(\mu,\mu)\mid\mu\in X^+\}$. So by Proposition \ref{prop:very-flat-characterisation} we see that $\mathrm{Vin}_G$ is very flat. Since $\L^{\mathrm{Vin}}\cap(-\L^{\mathrm{Vin}})=\{0\}$ and $\L^{\mathrm{Vin}}\cap(0\oplus\NN^\Delta)=\{0\}$, we see that $\mathrm{Vin}_G$ has a zero by Proposition \ref{prop:characterization zero of reductive monoid}.\par 
    Now assume that $G$ is semisimple and simply connected. Then $X^+$ is the free commutative monoid generated by the fundamental weights $\omega_1,\dotsc,\omega_r$. Therefore $\L_{\mathrm{Vin}}$ is a free monoid with basis $\{(\alpha_i,0)\mid 1\le i\le r\}\cup\{(\omega_i,\omega_i)\mid 1\le i\le r\}$. Since the ring of invariants under $\Ad(G)$ and $\Ad(G_\enh)$ coincide, the last statement follows from Theorem \ref{theo:Chevalley iso}.
\end{proof}

\begin{theo}\label{theo: characterization veryflat}
  Let $\M$ be a reductive monoid with unit group $G$ and let $A_\M$ be the abelianisation of $\M$. Then $\M$ is very flat if and only if there exists a morphism of monoid schemes $A_{\M}\to A_{\mathrm{Vin}_{G_\der}}$ such that $\M\simeq A_{\M}\times_{A_{\mathrm{Vin}_{G_\der}}}\mathrm{Vin}_{G_\der}$.
\end{theo}
\begin{proof}
    For simplicity, in the following proof we denote $G_{+}\coloneqq G_{\der,\enh}=(T_{\der}\times G_{\der})/Z_{G_\der}$ and $\L^\Vin\coloneqq\L(\Vin_{G_\der})$.\par 
    The sufficiency is clear since $\mathrm{Vin}_{G_\der}$ is very flat, so is every base change of it. Now we assume that $\M$ is very flat. Let $\L\coloneqq\L(\M)\subset X^+$ be the commutative monoid associated to $\M$ and let $\L_{\ab}\coloneqq\L\cap X_{ab}$. By Proposition \ref{prop:very-flat-characterisation} there exists a homomorphism $\theta\colon Z\to T_{\der}$ that restricts to the identity on $Z_{\der}$ such that
    \[\L=\{(\nu,\lambda)\in X^*(T)^+\subset X_Z\times X_{\der}^+\mid \nu-\theta^*(\lambda)\in\mathcal{L}_{\ab}\}.\]
    We have a commutative diagram
    \[\xymatrix@C=2cm{
    G=(Z\times G_{\der})/Z_{\der}\ar[r]^{\theta\times\mathrm{Id}_{G_{\der}}}\ar[d] & G_{+}=(T_{\der}\times G_{\der})/Z_{G_\der}\ar[d]\\
    G_{\ab}=Z/Z_{\der}\ar[r]^{\theta} & G_{+,\ab}=T_{\der}/Z_{G_\der}
    }\]
    realizing $G$ as the fiber product $G_{\ab}\times_{G_{+,\ab}}G_+$. Accordingly we have an isomorphism of character lattice $X\cong (X_{\ab}\oplus X^*(T_+))/\ZZ^{\Delta}$ where $\ZZ^{\Delta}=X^*(T_{\der}/Z_{G_\der})$ embeds via the map $\alpha\mapsto(\theta^*(\alpha),(-\alpha,0))\in X_{\ab}\oplus X^*(T_+)$. By Lemma \ref{lem:pullback-abelization}, it remains to show that this isomorphism induces an isomorphism of monoids
    $\L\cong\L_{\ab}\oplus_{\NN^\Delta}\L^{\mathrm{Vin}}$. The map sends $(\nu,\la)\in\L$ to (the image of) $(\nu-\theta^*(\la),(\la,\la))$ and the inverse map sends $(\chi,(\mu,\la))$ with $\chi\in\L_{\ab}$ and $(\mu,\la)\in\L^{\mathrm{Vin}}$ to $(\theta^*(\mu)+\chi,\la)\in\L$. It is straightforward to check that these maps are well-defined and inverse bijections of each other.
\end{proof}

\begin{defi}
    A homomorphism of reductive monoids $\M\to \mathcal{N}$ is called \emph{excellent} if the following commutative diagram is cartesian:
    \[\xymatrix{\M\ar[r]\ar[d]& \mathcal{N}\ar[d]\\
    A_{\M}\ar[r] & A_{\mathcal{N}}}\]
    where the vertical arrows are the abelianisation maps.
\end{defi}

\begin{cor}
     Let $\M$ be a very flat reductive monoid scheme over $S$ with unit group $G$. Then there exists an excellent homomorphism $\M\to \Vin_{G_\der}$ that restricts to the identity on $G_{\der}$ and such a homomorphism is uniquely determined by the induced homomorphism between abelianisations $A_{\M}\to A_{\Vin_{G_\der}}$. If moreover $\M$ has a zero, then there is a unique such homomorphism.
\end{cor}
\begin{proof}
    The existence is a direct consequence of Theorem \ref{theo: characterization veryflat}. Let $u,u'\colon \M\to \mathrm{Vin}_{G_\der}$ be two excellent homomorphisms that restrict to the identity on $G_{\der}$ and induce the same map between abelianisations. Then $u$ and $u'$ differ by an automorphism of the monoid scheme $\M$ that commutes with the abelianisation map $\M\to A_{\M}$ and restricts to the identity on the derived group of $\M^\times$. We need to prove that such an automorphism is the identity. Since $G\to \M$ is scheme-theoretically dominant, it suffices to prove that the automorphism restricts to the identity on $G$. Since it restricts to the identity on the derived group $G_{\der}$ we only need to prove that it restricts to the identity on the central torus $Z$. The induced automorphism of $Z$ restricts to the identity on $Z_{\der}$ and induces the identity map on the quotient $G_{\ab}=Z/Z_{\der}$. On character groups we thus get an automorphism of $X_Z$ that restricts to the identity on the finite-index subgroup $X_{\ab}$ and hence must be the identity map. \par
    Now assume that $\M$ has a zero and let again $u,u'\colon \M\to \Vin_{G_\der}$ be two excellent homomorphisms that restrict to the identity on $G_{\der}$. Let $\theta,\theta'\colon Z\to T_{\der}$ be the restrictions of $u$ and $u'$ to the central torus. It suffices to prove $\theta=\theta'$. Let 
    $$\L\coloneqq\{(\mu,\lambda)\in X^+\subset X_Z\oplus X_{\der}^+\mid \mu-\theta^*(\lambda)\in\L_{\ab}\},$$
    $$\L'\coloneqq\{(\mu,\lambda)\in X^+\subset X_Z\oplus X_{\der}^+\mid \mu-\theta'^*(\lambda)\in\L_{\ab}\}.$$ 
    According to Lemma \ref{lem:pullback-abelization}, the excellent morphisms $u$, $u'$ induce isomorphisms of reductive monoids
    $$\M(\L)\xleftarrow[\sim]{}\M\xrightarrow[\sim]{} \M(\L')$$ which commute with the rigidifications and the abelianisation maps. Thus we have $\L=\L'$ and for any $(\mu,\lambda)\in X^*(T)^+$,
    $$\mu-\theta^*(\lambda)\in \L_{\ab} \Leftrightarrow \mu-\theta'^*(\lambda) \in \L_{\ab}.$$ 
    Replacing $\mu$ with $\theta^*(\lambda)$ or $\theta'^*(\lambda)$ we deduce that $\theta^*(\lambda)-\theta'^*(\lambda)$ and $\theta'^*(\lambda)-\theta^*(\lambda)$ both belong to $\L_{\ab}$. Since $\M$ has a zero, by Proposition \ref{prop:characterization zero of reductive monoid} we get $\L_{\ab}\cap(-\L_{\ab})=0$ and hence $\theta'^*(\lambda)=\theta^*(\lambda)$ for all $\la\in X_{\der}^+$. Since $X^+_{\der}$ generates $X_{\der}$, we conclude that $\theta=\theta'$. 
\end{proof}

\begin{defi}
    Assume that $G$ is semisimple. Let $I,J$ be two subsets of $\Delta$. We set 
    \[D_I\coloneqq\bigoplus_{i\in I}\QQ_{\geq 0}\a_i\subset X_{\QQ},\quad C_J\coloneqq\bigoplus_{j\in J}\QQ_{\geq 0}\w_j\subset X_{\QQ}\] 
    and
    \[F_{I,J}\coloneqq\set{(\mu,\lambda)\in X_{\QQ}\oplus X_{\QQ}\mid\mu-\lambda \in D_I,\lambda\in C_J}.\]
    We say that the pair $(I,J)$ is \emph{essential} if, at the level of the Dynkin diagram, $I$ does not contain any connected component of the complement of $J$.
\end{defi}
\begin{prop}
    Assume that $G$ is semisimple. We have the following parametrization of (downward closed) prime ideals of $\L_{\Vin,\ab}$ and $\L_\Vin$.
   \begin{enumerate}
       \item The map $I\mapsto\NN^\Delta\setminus D_I=\L_{\Vin,\ab}\setminus D_I$ induces a bijection between subsets of $\Delta$ and prime ideals of $\L_{\Vin,\ab}$. 
       \item The map $(I,J)\mapsto \L^{\mathrm{Vin}}\setminus F_{I,J}$ induces a bijection between essential pairs of subsets of $\Delta$ and downward closed prime ideals of $\L_{\mathrm{Vin}}$.
   \end{enumerate} 
\end{prop}
\begin{proof} 
    (1) Note that the sets $D_I$, when $I$ runs over subsets of $\Delta$, are exactly the faces of $\QQ_{\ge0}\L_{\Vin,\ab}=\QQ_{\ge0}^{\Delta}$. Then the result follows from Lemma \ref{lem:face-prime-ideal}.\par
    (2) By definition we have an isomorphism of convex polyhedral cones
    \[\QQ_{\ge0}\L_\Vin\cong\QQ_{\ge0}^\Delta\oplus\QQ_{\ge0}X^+=D_\Delta\oplus C_\Delta\]
    defined by sending $(\mu,\la)$ to $(\mu-\la,\la)$. Then we see that the faces of $\QQ_{\ge0}\L_\Vin$ are exactly $F_{I,J}$ for pairs $(I,J)$ of subsets of $\Delta$. By Lemma \ref{lem:face-prime-ideal} it remains to show that the pair $(I,J)$ is essential if and only if the prime ideal $\L_\Vin\setminus F_{I,J}$ is downward closed, or equivalently $F_{I,J}\cap\L_\Vin$ is upward closed.\par 
    First suppose that $(I,J)$ is an essential pair. Let $(\mu,\lambda)\leq(\mu',\lambda')$ be elements in $\L_\Vin$ with $(\mu,\lambda)\in F_{I,J}$. Then we have $\mu=\mu'$ and the following inequalities hold:
    \[0\leq \lambda'-\lambda\leq \mu-\lambda,\quad 0\le\mu'-\lambda'\leq \mu-\lambda.\]
    The latter inequality forces $\mu'-\lambda'\in D_I$ whereas the former implies that $\la'-\la\in D_I$ and hence $\lambda'\in (D_I+C_J)\cap X^+$. Then there exists a subset $I'\subset I$ and positive rational numbers $n_i>0$, $i\in I'$ such that $\lambda'\in \sum_{i\in I'}n_i\a_i + C_J$. We claim that $I'\subset J$. Suppose on the contrary that $I'\cap(\Delta\setminus J)\ne\emptyset$. Let $K$ be a connected component of $\Delta\setminus J$ such that $I'\cap K\ne\emptyset$. Since the pair $(I,J)$ is essential, $K$ is not contained in $I'$. So there exists $i'\in I'$ and $k\in K\setminus I'\subset\Delta\setminus(I'\cup J)$ such that $i'$ and $k$ are connected by an edge in the Dynkin diagram. Then we have $\langle\a_k^\vee,\a_{i'}\rangle<0$ and hence 
    \[\langle\a_k^\vee,\lambda'\rangle =\sum_{i\in I'}n_i\langle\a_k^\vee, \a_{i}\rangle\leq n_{i'}\langle\a_k^\vee,\a_{i'}\rangle<0.\] 
    But this contradicts the fact that $\la'\in X^+$. Consequently we have $I'\subset J$ and hence $\la'\in (D_J+C_J)\cap X^+$. For any $k\in\Delta\setminus J$ (if this set is nonempty), since the coroot $\alpha_k^\vee$ is orthogonal to elements in $C_J$ and is non-positive on $D_J$, we have $\langle\alpha_k^\vee,\la'\rangle\le0$. Moreover, since $\la'\in X^+$ we must have $\langle\alpha_k^\vee,\la'\rangle=0$ and hence $\lambda'\in C_J$. Therefore $(\mu',\lambda')\in F_{I,J}$ and we conclude that $F_{I,J}$ is upward closed.\par 
    Conversely let $(I,J)$ be any pair of subsets of $\Delta$ and suppose that the prime ideal $\L_\Vin\setminus F_{I,J}$ is downward closed. Suppose that $(I,J)$ is not an essential pair, that is, there exists a connected component $\tilde{I}$ of $\Delta\setminus J$ such that $\tilde{I}\subset I$. Let $G_{\tilde{I}}$ be the Levi subgroup of $G$ whose root system has simple roots $\tilde{I}$. Let $\alpha$ be the highest root of $G_{\tilde{I}}$. Then we have 
    \[\alpha=\sum_{i\in\tilde{I}}m_i\alpha_i=\sum_{i\in\Delta}l_i\omega_i\]
    where $m_i>0$ for all $i\in\tilde{I}$ and 
    \[l_i\begin{cases}
        \ge0&\text{ for }i\in\tilde{I},\\
        \le0&\text{ for }i\in J,\\
        =0&\text{ for }i\notin \tilde{I}\cup J. 
    \end{cases}\]
    Moreover, there exists $i\in\tilde{I}$ such that $l_i>0$. Consider the element $\la\coloneqq\sum_{i\in\tilde{I}}l_i\omega_i$. Then we have
    \[\la\in C_\Delta\setminus C_J,\quad\La\coloneqq(\la,\la)\in\QQ_{\ge0}\L_\Vin\setminus F_{I,J}.\]
    On the other hand, we have 
    \[\La-(0,\alpha)=(\la,\la-\alpha)=(\la,-\sum_{j\in J}l_j\omega_j)\in F_{I,J}\]
    since $l_j\le0$ for all $j\in J$ and $\alpha\in D_I$. We can find a positive integer $n>0$ such that $n\la\in X^+$. Then we have $n\La\in\L_\Vin\setminus F_{I,J}$ and $n\La-n\cdot(0,\alpha)\in\L_\Vin\cap F_{I,J}$ which contradicts the fact that $\L_\Vin\setminus F_{I,J}$ is downward closed in $\L_\Vin$. Thus we conclude that the pair $(I,J)$ is essential and the proof is finished.
\end{proof}

\begin{cor}
    Suppose that the base scheme is $S=\Spec(k)$ where $k$ is an algebraically closed field and $G$ is semisimple. Then we have order-preserving bijections:
    \[\Set{ \text{subsets of }\Delta}\xleftrightarrow{1:1}\Set{T_{\ad}\text{ orbits in } {A_{\mathrm{Vin}_G}}},\]
    \[\Set{ \text{essential pairs of subsets of }\Delta}\xleftrightarrow{1:1}\Set{G_\enh\times G_\enh \text{-orbits in } {\mathrm{Vin}_{G}}}\]
    where the order on orbits is defined by inclusion of closures and the order on essential pairs is defined by componentwise inclusions. 
\end{cor}
\begin{proof}
        This follows from the previous Proposition and Theorem \ref{theo:parametrization GxG orbits}.
\end{proof}

\begin{defi}
    The \emph{asymptotic semigroup} of $G$ is the semigroup scheme over $S$ defined by 
    \[\As(G)\coloneqq\Spec(\gr\O_S[G])\]
    where the semigroup structure is induced from the coalgebra structure on $\gr(\O_S[G])$, see \cite[Definition 4.4]{CJ1}. 
\end{defi}

\begin{prop}
    Let $\M\coloneqq\M(\L)$ be the very flat rigidified reductive $G$-monoid over $S$ with weight monoid $\L$. Let $\theta\colon Z\to T_{\der}$ be a homomorphism as in Proposition \ref{prop:very-flat-characterisation}.
    \begin{enumerate}
        \item Let $p\colon \L\to\L_{\ab}$ be the map defined by $(\mu,\lambda)\mapsto \mu-\theta^*(\lambda)$. Let $\I\subset \L_{\ab}$ be a prime ideal. Then $p^{-1}(\I)\subset\L$ is a downward closed prime ideal and we have a canonical isomorphism:
        \[\M(\L_{\ab}/\I)\times_{A_{\M}}\M\cong\M(\L/p^{-1}(\I)).\]
        \item Suppose that $\L_\ab^*=\{0\}$ and $\theta^*(\alpha_i)\ne0$ for all $i\in\Delta$. Then the fiber of the abelianisation map over the zero of $A_\M$ is identified with the asymptotic semigroup of $G_{\der}$:
        \[\alpha_\M^{-1}(0_{A_\M})\cong\Spec_S(\mathrm{gr}(\O_S[G_{\der}])).\]
        In particular, this describes the zero fiber of the abelianisation map for $\Vin_{G_\der}$. 
    \end{enumerate}
\end{prop}
\begin{proof}
    We may assume that $S=\Spec(\k)$ is affine.\par 
    (1) Let $\fM_1\coloneqq\{(\theta^*(\la),\la)\mid\la\in X_{\der}^+\}$. Then $\fM_1$ is a submonoid of $\L$ and by Proposition \ref{prop:very-flat-characterisation} we have $\L=\L_{\ab}\oplus\fM_1$. Combined with \cite[Lemma 3.10]{CJ1} we deduce that $\k[\I]\cdot{}_\k\bfO(\L)={}_\k\bfO(\I+\fM_1)$ and then (1) follows from the observation $p^{-1}(\I)=\I+\fM_1$.\par
    (2) Now suppose that the condition in (2) is satisfied. By Proposition \ref{prop:characterization zero of reductive monoid}, the complement $\L\setminus\{0\}$ is a downward closed prime ideal of $\L$. Thus $\L_{\ab}\setminus\set{0}$ is a prime ideal of $\L_{\ab}$ and determines a zero of $A_{\M}$.\par  
    The $\k$-algebra $\k[\M]={}_\k\bfO(\L)$ admits a natural grading by the submonoid $\theta^*(X_{\der}^+)+\L_{\ab}\subset X_Z$:
    \[\k[\M]=\bigoplus_{\mu\in\theta^*(X_{\der}^+)+\L_{\ab}}\k[\M]_\mu\]
    where $\k[\M]_{\mu}$ is the $\k$-submodule of $\k[G]={}_\k\bfO$ spanned by the following subset of the dual canonical basis:
    \[\bigsqcup_{\substack{\la\in X_{\der}^+\\ \mu-\theta^*(\la)\in\L_{\ab}}}\dot{\bfB}[(\mu,\la)]^*.\]
    Let $\dot{\bfB}_{\der}^*$ denote the dual canonical basis for $\k[G_{\der}]$ and for all $\la\in X_{\der}^+$ let $\dot{\bfB}_{\der}[\la]^*$ be the corresponding two-sided cell. For any $\mu\in\theta^*(X_{\der}^+)+\L_{\ab}$, let $\k[G_{\der}]_{\le\mu}$ (resp. $\k[G_{\der}]_{<\mu}$) be the $\k$-submodule of $\k[G_{\der}]$ spanned by the following subset of the dual canonical basis:
    \[\bigsqcup_{\substack{\la\in X_{\der}^+\\ \mu-\theta^*(\la)\in\L_{\ab}}}\dot{\bfB}_{\der}[\la]^*\quad\quad(\text{resp. }\bigsqcup_{\substack{\la\in X_{\der}^+\\ \mu-\theta^*(\la)\in\L_{\ab}\setminus\{0\}}}\dot{\bfB}_{\der}[\la]^*).\]
    Since $\theta^*(\NN^\Delta)\subset\L_{\ab}$, the indexing set $\{\la\in X_{\der}^+\mid \mu-\theta^*(\la)\in\L_{\ab}\}$ is a downward closed subset of $X_{\der}^+$. Then each $\k[G_{\der}]_{\le\mu}$ is a sub-coalgebra of $\k[G_{\der}]$ by \cite[Proposition 4.2]{CJ1} and for any $\mu_1,\mu_2\in\theta^*(X_{\der}^+)+\L_{\ab}$ we have \[\k[G_{\der}]_{\le\mu_1}\cdot\k[G_{\der}]_{\le\mu_2}\subset\k[G_{\der}]_{\le\mu_1+\mu_2}\] 
    by Proposition \ref{prop:finite-invariant-review}(2). In this way $\k[G_{\der}]$ becomes a $(\theta^*(X_{\der}^+)+\L_{\ab})$-filtered $\k$-algebra. \par
    Recall that the surjective homomorphism $X\to X_{\der}$ induced by the natural embedding $G_{\der}\into G$ is given by $(\mu,\la)\mapsto\la$ where $\mu\in X_Z$ and $\la\in X_{\der}$. Then by \cite[Lemma 3.14]{CJ1}, the natural $\k$-algebra homomorphism ${}_\k\bfO=\k[G]\to\k[G_{\der}]$ restricts to canonical bijections $\dot{\bfB}[(\mu,\la)]^*\xrightarrow{\sim}\dot{\bfB}_{\der}[\la]^*$ for all $(\mu,\la)\in X^+$ where $\mu\in X_Z$ and $\la\in X_{\der}^+$. Consequently we have a canonical isomorphism of $(\theta^*(X_{\der}^+)+\L_{\ab})$-graded $\k$-algebras
    \[\k[\M]=\bigoplus_{\mu\in \theta^*(X_{\der}^+)+\L_{\ab}}\k[\M]_\mu\xrightarrow{\sim}\bigoplus_{\mu\in \theta^*(X_{\der}^+)+\L_{\ab}}\k[G_{\der}]_{\le\mu}\]
    where the right-hand side is the Rees algebra associated to the $(\theta^*(X_{\der}^+)+\L_{\ab})$-filtered algebra $\k[G_{\der}]$.  By \cite[Lemma 3.10]{CJ1}, this isomorphism maps the ideal $\k[\L_{\ab}\setminus\{0\}]\cdot\k[\M]$ onto 
    \[\bigoplus_{\mu\in \theta^*(X_{\der}^+)+\L_{\ab}}\k[G_{\der}]_{<\mu}\] 
    and so induces an isomorphism
    \begin{align*}
        \frac{\k[\M]}{\k[\L_{\ab}\setminus\{0\}]\cdot\k[\M]}\xrightarrow{\sim}&\bigoplus_{\mu\in \theta^*(X_{\der}^+)+\L_{\ab}}\frac{\k[G_{\der}]_{\le\mu}}{\k[G_{\der}]_{<\mu}}\\
        &=\bigoplus_{\mu\in\theta^*(X_{\der}^+)}\bigoplus_{\substack{\la\in X_{\der}^+\\ \theta^*(\la)=\mu}}\frac{\k[G_{\der}]_{\le\la}}{\k[G_{\der}]_{<\la}}=\mathrm{gr}(\k[G_{\der}]).
    \end{align*}
    This gives the description of the zero fiber of the abelianisation map.
\end{proof}

\subsection{Comparison with other definitions}\label{sec:comparison}
In this section we work over a field $k$ and compare our construction of the Vinberg monoid with other constructions in the literature. 
\subsubsection{Rees algebra for the canonical filtration}
Recall from \cite{XiaoZhu2019} that for any $G$-module $V$, the \emph{canonical filtration} on $V$, labeled by the monoid $X^+_{\pos}\coloneqq X^++\NN^\Delta$ generated by the dominant weights and the positive roots, is defined as follows: for any $\mu\in X^+_{\pos}$, we let $V_{\le\mu}$ be the maximal $G$-submodule of $V$ such that the $\nu$-weight space $V_{\le\mu}(\nu)\ne0$ implies $\nu\le\mu$ for any $\nu\in X$. For the $G\times G$-module $k[G]$ we have its canonical filtration $k[G]_{\le(\mu_1,\mu_2)}$ where $(\mu_1,\mu_2)\in X^+_{\pos}\times X^+_{\pos}$ is a weight for $G\times G$. For any $\la\in X^+_{\pos}$, define $k[G]_{\preccurlyeq\la}\coloneqq k[G]_{\le(-w_0(\la),\la)}$. Then we get a $X^+_{\pos}$-filtration on $k[G]$. We show that this coincides with the filtration in \cite[Definition 3.8]{CJ1} given by the canonical basis. 

\begin{lem}
    With notations as above, we have $k[G]_{\preccurlyeq\la}={}_k\bfO_{\le\la}$ for any $\la\in X^+_{\pos}$. 
\end{lem}
\begin{proof}
    Recall that ${}_k\bfO_{\le\la}$ is a sub-coalgebra of $k[G]$ and hence also a $G\times G$-submodule. Then by definition we have ${}_k\bfO_{\le\la}\subset k[G]_{\preccurlyeq\la}$. By \cite[Lemma 3.2.1]{XiaoZhu2019} we have 
    \[\dim_k(k[G]_{\preccurlyeq\la})=\sum_{\mu\in X^+,\mu\le\la}(\dim_k W_\mu)^2\]
    where $W_\mu$ is the Weyl module with highest weight $\mu$ (see \cite[\S5]{CJ1}). On the other hand, by \cite[Proposition 29.2.2]{Lusztig-IntroQuantumGroup} we have
    \[\dim_k({}_k\bfO_{\le\la})=\sum_{\mu\in X^+,\mu\le\la}\#\dot{\bfB}[\mu]=\sum_{\mu\in X^+,\mu\le\la}(\dim_k\La_\mu)^2.\]
    Then from \cite[Proposition 5.3]{CJ1} we get $\dim_k(k[G]_{\preccurlyeq\la})=\dim_k({}_k\bfO_{\le\la})$ and hence $k[G]_{\preccurlyeq\la}={}_k\bfO_{\le\la}$.
\end{proof}
Consequently we have 
\[k[\mathrm{Vin}_G]=\bigoplus_{\la\in X^+_{\pos}}{}_k\bfO_{\le\la}=\bigoplus_{\la\in X^+_{\pos}}k[G]_{\preccurlyeq\la}\] 
and so our definition of the Vinberg monoid coincides with the definition in \cite{XiaoZhu2019} (which is over a general field) and \cite{Zhu2020} (which is over $\ZZ$). Moreover, if $k$ is a field of characteristic $0$, from the definition of the canonical filtration it is easy to see that under the decomposition \eqref{eq:Peter-Weyl} we have
\[k[G]_{\preccurlyeq\la}=\bigoplus_{\mu\in X^+,\mu\le\la}V_\mu^*\otimes V_\mu.\]
Thus over a field of characteristic $0$ our definition of $\mathrm{Vin}_G$ coincides with those in \cite{Vinberg1995,Ganev,BZv-Ganev}.

\subsubsection{Description of cones}\label{sec:description-cone}
Let us first spell out the various cones in \S\ref{sec:cones} for the Vinberg monoid. Let $D\coloneqq\sum_{i\in\Delta}\QQ_{\ge0}\alpha_i\subset X_{\QQ}$ be the positive root cone. Then we have
\[\QQ_{\ge0}\L_\Vin=\Set{(\mu,\la)\in X_{\QQ}\oplus X_{\QQ}\mid\la\in\QQ_{\ge0}X^+\cap(\mu-D)}.\]
It is the intersection of the dominant weight cone $\QQ_{\ge0}X_\enh^+$ with
\[K_{\L^{\mathrm{Vin}}}\coloneqq\Set{(\mu,\la)\in X_{\QQ}\oplus X_{\QQ}\mid\la\in\mu-D}.\]
Let $\mathrm{Vin}_T$ be the schematic closure of $T_\enh$ in $\mathrm{Vin}_G$. It is the normal affine embedding of $T_\enh$ associated to the $W$-stable cone 
\[\tilde{K}_{\L_\Vin}=\Set{(\mu,\la)\in X_{\QQ}\oplus X_{\QQ}\mid \mu-w(\la)\in D\text{ for all }w\in W}.\]
This is the cone for $\mathrm{Vin}_G$ in Renner's classification. \par 
Next we consider the dual cones. We assume that $G$ is semisimple from now on. In \cite{RittatoreMonoidArticle,RittatoreVeryflat,RittatoreThese}, Rittatore applied the Luna-Vust theory of spherical embeddings to classify reductive monoids over a field and in particular defines the Vinberg monoid (which is called ``universal enveloping monoid" in \emph{loc. cit.}). Recall from  \cite[Definition 3]{RittatoreVeryflat} that the \emph{universal enveloping monoid} $\mathrm{Env}(G)$ is the affine spherical embedding of $G_\enh$, viewed as a spherical $G_\enh\times G_\enh$-homogeneous space, whose colored cone $\mathcal{C}_{\mathrm{Env}(G)}$ is the cone in $X_*(T_\enh)_\QQ=X_*(T)_\QQ\oplus X_*(T)_\QQ$ generated by the elements\footnote{We warn the reader that according to our convention, the first factor in $X_*(T_\enh)_\QQ$ corresponds to the center and the second factor corresponds to the derived group. This is opposite to Rittatore's convention. Also, Rittatore has made a non-canonical identification between $X_*(T_\enh)_\QQ$ and $X^*(T_\enh)_\QQ$.} $\Set{(\omega_i^\vee,-\omega_i^\vee)\mid 1\le i\le r}$ and $\{(0,\alpha_i^\vee)\mid 1\le i\le r\}$ (here $\omega_i^\vee$ are the fundamental coweights) in which the colors are elements in the latter set (they correspond to all the Bruhat divisors in $G_\enh$). Then the dual of $\mathcal{C}_{\mathrm{Env}(G)}$ is the cone
\[\mathcal{C}_{\mathrm{Env}(G)}^\vee\coloneqq\{(\mu,\la)\in X^*(T_\enh)_\QQ\mid\langle\mu-\la,\omega_i^\vee\rangle\ge0\text{ and }\langle\la,\alpha_i^\vee\rangle\ge0,\ \forall 1\le i\le r\}\]
which is exactly $\QQ_{\ge0}\L_\Vin$ by definition. Then according to \cite[Proposition 13]{RittatoreMonoidArticle} we get 
\[\L(\mathrm{Env}(G))=\L_\Vin\] 
and therefore $\mathrm{Env}(G)\cong\mathrm{Vin}_{G}$ by Theorem \ref{theo:monoid-classification}.

\subsubsection{Closure in matrix monoids}
As a consequence of Proposition \ref{prop:closure-in-matrix-algebra}, we have a down-to-earth construction of the Vinberg monoid as the normalization of the closure of $G_\enh$ in a product of matrix monoids. The description is especially simple when $G$ is semisimple and simply connected. Under this assumption, the monoid $X^+$ is freely generated by the fundamental weights $\omega_1,\dotsc,\omega_r$ and hence $\L_\Vin$ is a free commutative monoid with generators $(\alpha_1,0),\dotsc,(\alpha_r,0),(\omega_1,\omega_1),\dotsc,(\omega_r,\omega_r)$. For $1\le i\le r$, let $V_i$ be the Weyl module of $G$ with highest weight $\omega_i$. This extends to a representation
\[\rho_i^+\colon G_\enh\to\mathrm{End}(V_i)\]
by requiring that the central torus $T\subset G_\enh$ acts via $\omega_i$. The resulting representation is the Weyl module of $G_\enh$ with highest weight $(\omega_i,\omega_i)$. These representations, together with the one-dimensional representations $(\alpha_i,0)$, induce a locally closed embedding
\[G_\enh\to\AA^r\times\prod_{i=1}^r\mathrm{End}(V_i)\]
and the normalization of the closure of $G_\enh$ in the target is isomorphic to $\mathrm{Vin}_G$ by Proposition \ref{prop:closure-in-matrix-algebra}. Thus our definition of the Vinberg monoid coincides with the definitions in \cite{Ngo-PS,Bouthier-Springer,Chi-KV,Wang-FL}. 

\appendix
\section{Recollection on commutative monoids}\label{sec:appendix-monoid}
We summarize some basic facts on commutative monoids that are used throughout the article. The main reference is \cite[Chapter I]{Ogus-Log}.
\begin{defi}
    A \emph{monoid} is a triple $(\sM,*,e_\sM)$ consisting of a set $\sM$, an associative binary operation $*$ on $\sM$ and a two-sided identity element $e_\sM\in\sM$ for $*$. The monoid is \emph{commutative} if the binary operation $*$ is commutative. For a commutative monoid $\sM$, we will use the additive notation ``$+$" to denote the binary operation and use $0=0_\sM$ to denote the identity element. We have obvious notions of submonoids and homomorphisms of monoids. The collection of monoids forms a category $\mathrm{Mon}$ and the commutative monoids form a subcategory $\mathrm{CMon}\subset\mathrm{Mon}$.
\end{defi}
The forgetful functor $\mathrm{Ab}\to\mathrm{CMon}$ from the category of abelian groups to the category of commutative monoids has a left adjoint $\sM\mapsto\sM^{\gp}$. The associated natural transformation $\la_\sM\colon\sM\to\sM^{\gp}$ is a universal monoid homomorphism from $\sM$ to a group. The forgetful functor also has a right adjoint $\sM\mapsto\sM^*$ where $\sM^*$ is the maximal subgroup of $\sM$, i.e. the group of invertible elements in $\sM$.
\begin{defi}\label{def:monoid-properties}
    A monoid $\sM$ is 
    \begin{itemize}
        \item \emph{finitely generated} if there exists a finite subset $S\subset\sM$ such that the only submonoid of $\sM$ containing $S$ is $\sM$ itself.
        \item \emph{integral} if it is commutative and whenever any elements $m,m',m''\in\sM$ satisfy $m+m'=m+m''$, we have $m'=m''$. 
        \item \emph{fine} if it is integral and finitely generated.
        \item \emph{saturated} if it is integral and if whenever $nx\in\sM$ for some $x\in\sM^{\gp}$ and some positive integer $n$, then $x\in\sM$.
    \end{itemize}
\end{defi}

When $\sM$ is integral, the map $\la_\sM\colon\sM\to\sM^{\gp}$ is injective and identifies $\sM$ with a submonoid of the abelian group $\sM^{\gp}$. In this case, we also denote $\sM^{\gp}\coloneqq\sM-\sM$ and view its elements as formal differences between elements in $\sM$. Moreover we have $\sM^*=\sM\cap(-\sM)$. 

\begin{defi}\label{def:monoid-ideal}
    Let $\sM$ be a commutative monoid. An \emph{ideal} of $\sM$ is a subset $\sI\subset\sM$ such that $m+x\in\sI$ for all $x\in\sI$ and all $m\in\sM$. A \emph{prime ideal} of $\sM$ is a \emph{proper} ideal $\sI\subset\sM$ such that the complement $\sM\setminus\sI$ is a submonoid of $\sM$, i.e. $m_1+m_2\notin\sI$ for all $m_1,m_2\in\sM\setminus\sI$. (We must also have $0\in\sM\setminus\sI$ since otherwise we would get $\sI=\sM$.)
\end{defi}
We note that arbitrary unions of (prime) ideals of $\sM$ are (prime) ideals. Also, arbitrary intersections of ideals are ideals. But an intersection of prime ideals may not be a prime ideal.

\begin{defi}\label{def:M-set}
    Let $\sM$ be a monoid. An \emph{$\sM$-set} is a set $S$ with an action of $\sM$, i.e. a monoid homomorphism $\sM\to\mathrm{End}(S)$. In particular, the monoid $\sM$ is an $\sM$-set over itself. For each element $m\in\sM$, the endomorphism on an $\sM$-set $S$ induced by $m$ will be denoted by $s\mapsto m\cdot s,\forall s\in S$. \par   
    An $\sM$-set $S$ is \emph{free} (or more precisely, \emph{$\sM$-free}) if there exists a subset $T\subset S$ such that the action map $\sM\times T\to S$ defined by $(m,t)\mapsto m\cdot t$ is bijective. In this case, we say that $T$ is a \emph{basis} (or more precisely, an \emph{$\sM$-basis}) for $S$. More generally, for any $\sM$-set $S$, we say that a subset $T\subset S$ \emph{generates $S$} if the action map $\sM\times T\to S$ is surjective.\par 
    An $\sM$-set $S$ is \emph{finitely generated} if there exists a finite subset $T\subset S$ that generates $S$. It is \emph{noetherian} if any of its $\sM$-stable subsets is finitely generated. The monoid $\sM$ is \emph{noetherian} if it is noetherian as an $\sM$-set over itself. 
\end{defi}
For example, $\sM$ itself is an $\sM$-set and any ideal of $\sM$ is an $\sM$-set.
\begin{defi}\label{def:preorder-monoid-set}
    Let $\sM$ be a commutative monoid and let $S$ be an $\M$-set. We define a pre-order $\preceq$ and an equivalence relation $\sim$ on $S$ by:
    \begin{itemize}
        \item $s_1\preceq s_2$ if there exists $m\in\sM$ such that $m\cdot s_1=s_2$,
        \item $s_1\sim s_2$ if $s_1\preceq s_2$ and $s_2\preceq s_1$. 
    \end{itemize}
    Let $S_{\min}\subset S$ be the subset of minimal elements under $\preceq$. In other words, $s\in S_{\min}$ if whenever an element $s'\in S$ satisfies $s'\preceq s$ we have $s'\sim s$. Then $S_{\min}$ is a union of $\sM^*$-orbits. 
\end{defi}
For example, when $S=\sM$ we have $\sM_{\min}=\sM^*$. In general $S_{\min}$ might be empty (for instance when $\sM=\NN$ and $S=\ZZ$).
\begin{defi}
    Let $\k$ be a commutative ring and $\sM$ be a commutative monoid. The \emph{monoid algebra (with coefficients in $\k$)} is the free $\k$-module with basis $\{e^m,m\in\sM\}$, endowed with the unique $\k$-algebra structure whose multiplication is determined by the rule $e^{m_1}\cdot e^{m_2}=e^{m_1+m_2}$ for all $m_1,m_2\in M$. In particular, the multiplicative unit is $1\coloneqq e^0\in\k[\sM]$. \par 
    For any $\sM$-set $S$, we let $\k[S]$ be the free $\k$-module with basis $S$ and endowed with the unique $\k[\sM]$-module structure compatible with the action of $\sM$ on $S$. In particular, for any ideal $\sI\subset\sM$, the $\k$-submodule $\k[\sI]\subset\k[\sM]$ spanned by $\sI$ is an ideal of $\k[\sM]$. 
\end{defi}

\begin{lem}\label{lem:monoid-algebra-property}
    Let $\sM$ be an integral monoid such that $\sM^{\gp}$ is torsion-free and let $\k$ be an integral domain.
    \begin{enumerate}
        \item The monoid algebra $\k[\sM]$ is an integral domain and for any prime ideal $\sI\subset\sM$, the ideal $\k[\sI]\subset\k[\sM]$ is a prime ideal.
        \item If moreover $\k$ is normal and $\sM$ is saturated and finitely generated, then $\k[\sM]$ is normal.
    \end{enumerate}
\end{lem}
\begin{proof}
    (1) The first statement follows from \cite[Proposition 3.4.1, part 1]{Ogus-Log}. For any prime ideal $\sI\subset\sM$, the submonoid $\sM\setminus\sI$ is also integral and $(\sM\setminus\sI)^{\gp}$ is a subgroup of $\sM^{\gp}$ and hence torsion-free by assumption. Thus we can apply the first statement to the quotient ring $\k[\sM]/\k[\sI]$, which is canonically isomorphic to $\k[\sM\setminus\sI]$, and conclude that $\k[\sI]$ is a prime ideal.\par 
    (2) follows from \cite[Proposition 3.4.1, part 2]{Ogus-Log}.
\end{proof}

\begin{defi}
    A homomorphism of monoids $\theta\colon\sP\to\sQ$ is called:
    \begin{itemize}
        \item \emph{local} if $\theta^{-1}(\sQ^*)=\sP^*$. 
        \item \emph{integral} if $\sP$ and $\sQ$ are both integral monoids and the following condition is satisfied: whenever $q_1,q_2\in\sQ$ and $p_1,p_2\in\sP$ satisfy $\theta(p_1)+q_1=\theta(p_2)+q_2$, there exists $q'\in\sQ$ and $p_1',p_2'\in\sP$ such that 
        \[q_1=\theta(p_1')+q',\quad q_2=\theta(p_2')+q',\quad p_1+p_1'=p_2+p_2'.\]
    \end{itemize}
\end{defi}

\begin{lem}\label{lem:free-set-flat}
    Let $\theta\colon\sP\to\sQ$ be an injective and local homomorphism of fine monoids and view $\sQ$ as a $\sP$-set via $\theta$. Then the following conditions are equivalent:
    \begin{enumerate}
        \item $\sQ$ is a free $\sP$-set.
        \item $\ZZ[\sQ]$ is a free $\ZZ[\sP]$-module.
        \item There exists a field $k$ such that $k[\sQ]$ is a flat $k[\sP]$-module.  
    \end{enumerate}
    Moreover, if these conditions are satisfied, then any set of representatives of $\sQ_{\min}/\sP^*$ forms a $\sP$-basis of $\sQ$, where $\sQ_{\min}$ is the set of $\sP$-minimal elements in $\sQ$. 
\end{lem}
\begin{proof}
    All the references in the following proof are from \cite[Chapter I]{Ogus-Log}. The implications $(1)\Rightarrow(2)\Rightarrow(3)$ are obvious. Assume (3) holds. By Proposition 4.6.8, the homomorphism $\theta$ is integral and then (1) follows from Corollary 4.6.11. By Proposition 4.6.9, each connected component of the $\sP$-set $\sQ$ is isomorphic to $\sP$ and by Remark 4.6.10 the generator of each component is $\sP$-minimal. This proves the last statement.    
\end{proof}

Finally let us elaborate on the close connection between monoids and convex cones. For us, ``cones" always mean ``rational cones", i.e. nonempty subsets of $\QQ$-vector spaces that are stable under scalar multiplication by $\QQ_{\ge0}$.\par   
Let $X$ be a free abelian group of finite rank and let $Y\coloneqq\Hom(X,\ZZ)$ be the dual abelian group. Let $X_\QQ\coloneqq X\otimes_\ZZ\QQ$ and $Y_\QQ\coloneqq Y\otimes_\ZZ\QQ$ be the $\QQ$-vector spaces they generate. Recall that a \emph{convex polyhedral cone} in $X_\QQ$ can be defined in two equivalent ways:
\begin{itemize}
    \item as a finitely generated convex cone in $X_\QQ$, i.e. a set of the form
    \[C=\Set{\sum_{i=1}^na_ix_i\mid a_1,\dotsc,a_n\in\QQ_{\ge0}}\subset X_\QQ\]
    where $x_1,\dotsc,x_n\in X_\QQ$;
    \item as a finite intersection of half-spaces of $X_\QQ$, where by a ``half-space" we mean a subset of $X_\QQ$ of the form $y^{-1}(\QQ_{\ge0})$ for some $0\ne y\in Y_\QQ$. By convention, $X_\QQ$ itself is viewed as an empty intersection of half-spaces. 
\end{itemize}
For the equivalence of the two definitions, see for example \cite[Theorem 19.1]{Rock-convex}. Let $C\subset X_\QQ$ be a convex polyhedral cone. The \emph{dual} of  $C$ is the convex polyhedral cone in $Y_\QQ$ defined by
\[C^\vee\coloneqq\Set{y\in Y_\QQ\mid y(C)\subset\QQ_{\ge0}}.\]
A \emph{face} of $C$ is a subset of $C$ of the form $y^{-1}(0)\cap C$ where $y\in C^\vee$. In particular, taking $y=0$ we get $C$ as a face of itself. The \emph{dimension} of a face $F\subset C$, denoted $\dim(F)$, is the dimension of the $\QQ$-vector space spanned by it. The \emph{codimension} of a face $F\subset C$ is the difference $\dim(C)-\dim(F)$. A face of $C$ of codimension $1$ is often called a \emph{facet} in the literature. 

\begin{lem}\label{lem:monoid-cone}
    The maps $\sM\mapsto\QQ_{\ge0}\sM$ and $C\mapsto C\cap X$ define mutually inverse bijections between finitely generated submonoids of $X$ that are saturated in $X$ and convex polyhedral cones in $X_\QQ$. Here $\QQ_{\ge0}\sM$ is the set of all finite linear combinations of elements in $\sM$ with coefficients in $\QQ_{\ge0}$. 
\end{lem}
\begin{proof}
    If $\sM\subset X$ is a finitely generated submonoid, then $\QQ_{\ge0}\sM$ is a finitely generated convex cone and hence a convex polyhedral cone by the remarks above. Conversely if $C\subset X_\QQ$ is a convex polyhedral cone, then the intersection $C\cap X$ is obviously a saturated submonoid and it is finitely generated by Gordan's lemma, see for example \cite[Theorem 2.3.19]{Ogus-Log}. Thus the two maps are well defined and they are clearly mutually inverse to each other. 
\end{proof}

\begin{lem}\label{lem:face-prime-ideal}
    Let $\sM$ be a finitely generated saturated submonoid of $X$ and let $C_\sM\coloneqq\QQ_{\ge0}\sM$ be the corresponding convex polyhedral cone. Then the map $F\mapsto\sM\setminus F$ defines an order-reversing bijection between faces of $C_\sM$ and prime ideals of $\sM$.
\end{lem}
\begin{proof}
    Let us define a \emph{face submonoid} of $\sM$ to be a submonoid $f\subset\sM$ such that for any $m_1,m_2\in\sM$ such that $m_1+m_2\in f$ we have $m_1\in f$ and $m_2\in f$. In \cite[Definition 1.4.1]{Ogus-Log} such a submonoid is simply called a ``face" of $\sM$. By definition it is clear that the map $f\mapsto \sM\setminus f$ defines an order-reversing bijection between face submonoids of $\sM$ and prime ideals of $\sM$. On the other hand, by \cite[Proposition 2.2.1, or Theorem 2.3.12(2)]{Ogus-Log}, the map $F\mapsto F\cap\sM$ defines an order-preserving bijection between faces of $C_\sM$ and face submonoids of $\sM$, with the inverse map given by $f\mapsto\QQ_{\ge0}f$.
\end{proof}


\end{document}